\documentclass[12pt]{amsart}
\usepackage[hmargin=2.5cm,vmargin=2.5cm]{geometry}
\usepackage{amsfonts, amstext, amsmath, amsthm, amscd, amssymb, amsbsy}
\usepackage{graphicx, color}

\usepackage{listings}
\usepackage[noadjust]{cite}
\usepackage[hidelinks,pagebackref]{hyperref}
\usepackage{enumitem}
\usepackage{setspace}
\usepackage{float}
\usepackage{subcaption}
\usepackage{mathrsfs}
\usepackage{tikz}
\usepackage{pinlabel}
\usepackage{multirow}

\usetikzlibrary{
	knots,
	hobby,
	decorations.pathreplacing,
	decorations.markings,
	shapes.geometric,
	arrows.meta,
	calc,
	arrows.meta,
	spath3,
	intersections
}

\tikzset{
	knot diagram/every strand/.append style={
		ultra thick,
		red!30
	},
	show curve controls/.style={
		postaction=decorate,
		decoration={show path construction,
			curveto code={
				\draw [blue, dashed]
				(\tikzinputsegmentfirst) -- (\tikzinputsegmentsupporta)
				node [at end, draw, solid, red, inner sep=2pt]{};
				\draw [blue, dashed]
				(\tikzinputsegmentsupportb) -- (\tikzinputsegmentlast)
				node [at start, draw, solid, red, inner sep=2pt]{}
				node [at end, fill, blue, ellipse, inner sep=2pt]{}
				;
			}
		}
	},
	show curve endpoints/.style={
		postaction=decorate,
		decoration={show path construction,
			curveto code={
				\node [fill, blue, ellipse, inner sep=2pt] at (\tikzinputsegmentlast) {}
				;
			}
		}
	}
}

\AtBeginDocument{
	\def\MR#1{}
}

\pagecolor{white}

\allowdisplaybreaks

\newcommand{\Z}{\mathbb{Z}}

\newcommand{\R}{\mathbb{R}}

\newtheorem{thm}{Theorem}
\numberwithin{thm}{section}

\newtheorem{conj}[thm]{Conjecture}
\newtheorem{prop}[thm]{Proposition}
\newtheorem{lemma}[thm]{Lemma}
\newtheorem{cor}[thm]{Corollary}

\newtheorem*{namedtheorem}{\theoremname}
\newcommand{\theoremname}{testing}
\newenvironment{named_thm}[1]{\renewcommand{\theoremname}{#1}\begin{namedtheorem}}{\end{namedtheorem}}
\newcommand{\refthm}[1]{Theorem~\ref{thm:#1}}

\theoremstyle{definition}
\newtheorem{defn}[thm]{Definition}
\newtheorem*{nameddef}{\defname}
\newcommand{\defname}{testing}

\theoremstyle{definition}
\newtheorem{rmk}[thm]{Remark}
\newtheorem{conv}[thm]{Convention}

\newcommand{\tikzone}{
	\mathord{
		\tikz[baseline={(0,0)}]{
			\useasboundingbox (-0.75,-1.5) rectangle (2.25,1.5);
			
			\coordinate (top) at (0,1.5);
			\coordinate (mid) at (0,0.5);
			\coordinate (bot) at (0,-1.5);
			
			\fill[yellow!20] (mid) to[out=-45,in=45,looseness=1.5] (bot) to[out=135,in=-135,looseness=1.5] (mid);
			\draw[thick,black] (mid) to[out=-45,in=45,looseness=1.5] (bot);
			\draw[thick,black] (bot) to[out=135,in=-135,looseness=1.5] (mid);
			\node at (0,-0.5) {$G_1$};
			
			\fill[gray!30] (top) to[out=-10,in=10,looseness=1.5] (bot) to[out=-10,in=10,looseness=2.5] (top);
			\draw[thick,black] (top) to[out=-10,in=10,looseness=1.5] (bot);
			\draw[thick,black] (bot) to[out=-10,in=10,looseness=2.5] (top);
			\node at (1.7,0) {$G_2$};
			
			\fill[black] (top) circle (2.5pt);
			\fill[black] (mid) circle (2.5pt);
			\fill[black] (bot) circle (2.5pt);
			
			\draw[-Triangle, thick, black] (mid) -- node[left] {$e$} (0,1.4);
			
			\draw[thick, red] (mid) circle (5pt);
			\node[thick, red] at (0.6, 0.5) {$*$};
		}
	}
}

\newcommand{\tikzonef}{
	\mathord{
		\tikz[baseline={(0,0)}]{
			\useasboundingbox (-0.75,-1.75) rectangle (2.25,1.5);
			
			\coordinate (top) at (0,1.5);
			\coordinate (mid) at (0,-0.5);
			\coordinate (bot) at (0,-1.5);
			
			\fill[yellow!20] (top) to[out=-45,in=45,looseness=1.5] (mid) to[out=135,in=-135,looseness=1.5] (top);
			\draw[thick,black] (top) to[out=-45,in=45,looseness=1.5] (mid);
			\draw[thick,black] (mid) to[out=135,in=-135,looseness=1.5] (top);
			\node at (0,0.5) {$\rho(G_1)$};
			
			\fill[gray!30] (top) to[out=-10,in=10,looseness=1.5] (bot) to[out=-10,in=10,looseness=2.5] (top);
			\draw[thick,black] (top) to[out=-10,in=10,looseness=1.5] (bot);
			\draw[thick,black] (bot) to[out=-10,in=10,looseness=2.5] (top);
			\node at (1.7,0) {$G_2$};
			
			\fill[black] (top) circle (2.5pt);
			\fill[black] (mid) circle (2.5pt);
			\fill[black] (bot) circle (2.5pt);
			
			\draw[thick, black] (bot) edge[-Triangle] (0,-0.6);
			
			\draw[thick, red] (top) circle (5pt);
			\node[thick, red] at (-0.5, -0.5) {$*$};
		}
	}
}

\newcommand{\tikztwo}{
	\mathord{
		\tikz[baseline={(0,-0.5)}]{
			\useasboundingbox (-0.75,-1.5) rectangle (1.25,1.5);
						
			\coordinate (mid) at (0,0.5);
			\coordinate (bot) at (0,-1.5);
			
			\fill[yellow!20] (mid) to[out=-45,in=45,looseness=1.5] (bot) to[out=135,in=-135,looseness=1.5] (mid);
			\draw[thick,black] (mid) to[out=-45,in=45,looseness=1.5] (bot);
			\draw[thick,black] (bot) to[out=135,in=-135,looseness=1.5] (mid);
			\node at (0,-0.5) {$G_1$};
			
			\fill[black] (mid) circle (2.5pt);
			\fill[black] (bot) circle (2.5pt);
			
			\draw[thick, red] (mid) circle (5pt);
			\node[thick, red] at (1, -0.5) {$*$};
		}
	}
}

\newcommand{\tikztwof}{
	\mathord{
		\tikz[baseline={(0,-0.5)}]{
			\useasboundingbox (-1.25,-1.5) rectangle (1,1.5);
			
			\coordinate (mid) at (0,0.5);
			\coordinate (bot) at (0,-1.5);
			
			\fill[yellow!20] (mid) to[out=-45,in=45,looseness=1.5] (bot) to[out=135,in=-135,looseness=1.5] (mid);
			\draw[thick,black] (mid) to[out=-45,in=45,looseness=1.5] (bot);
			\draw[thick,black] (bot) to[out=135,in=-135,looseness=1.5] (mid);
			\node at (0,-0.5) {$\rho(G_1)$};
			
			\fill[black] (mid) circle (2.5pt);
			\fill[black] (bot) circle (2.5pt);
			
			\draw[thick, red] (mid) circle (5pt);
			\node[thick, red] at (-1, -0.5) {$*$};
		}
	}
}

\newcommand{\tikzthree}{
	\mathord{
		\tikz[baseline={(0,0)}]{
			\useasboundingbox (-0.25,-1.5) rectangle (1.5,1.5);
			
			\coordinate (top) at (0,0.5);
			\coordinate (bot) at (0,-0.5);
			
			\fill[gray!30] (top) to[out=-10,in=10,looseness=2] (bot) to[out=-20,in=20,looseness=4.5] (top);
			\draw[thick,black] (top) to[out=-10,in=10,looseness=2] (bot);
			\draw[thick,black] (bot) to[out=-20,in=20,looseness=4.5] (top);
			\node at (0.9,0) {$G_2$};
			
			\fill[black] (top) circle (2.5pt);
			\fill[black] (mid) circle (2.5pt);
			\fill[black] (bot) circle (2.5pt);
			
			\draw[thick, black] (bot) edge[-Triangle] (0,0.4);
			
			\draw[thick, red] (bot) circle (5pt);
			\node[thick, red] at (0.25, 0) {$*$};
		}
	}
}

\newcommand{\tikzthreef}{
	\mathord{
		\tikz[baseline={(0,0)}]{
			\useasboundingbox (-0.5,-1.5) rectangle (1.5,1.5);
			
			\coordinate (top) at (0,0.5);
			\coordinate (bot) at (0,-0.5);
			
			\fill[gray!30] (top) to[out=-10,in=10,looseness=2] (bot) to[out=-20,in=20,looseness=4.5] (top);
			\draw[thick,black] (top) to[out=-10,in=10,looseness=2] (bot);
			\draw[thick,black] (bot) to[out=-20,in=20,looseness=4.5] (top);
			\node at (0.9,0) {$G_2$};
			
			\fill[black] (top) circle (2.5pt);
			\fill[black] (mid) circle (2.5pt);
			\fill[black] (bot) circle (2.5pt);
			
			\draw[thick, black] (bot) edge[-Triangle] (0,0.4);
			
			\draw[thick, red] (top) circle (5pt);
			\node[thick, red] at (-0.35, 0) {$*$};
		}
	}
}

\newcommand{\tikzfour}{
	\mathord{
		\tikz[baseline={(0,0)}]{
			\useasboundingbox (-2,-1.5) rectangle (0.25,1.5);
			
			\coordinate (bot) at (0,-0.01);
			\coordinate (top) at (0,0.01);
			
			\fill[yellow!20] (top) to[out=135,in=-135,looseness=200] (bot) to[out=-120,in=120,looseness=650] (top);
			\draw[thick,black] (bot) to[out=-120,in=120,looseness=650] (top);
			\draw[thick,black] (top) to[out=135, in=-135,looseness=200] (bot);
			\node at (-1.25, 0) {$G_1$};
			
			\fill[black] (0,0) circle (2.5pt);
			
			\draw[thick, red] (0,0) circle (5pt);
			\node[thick, red] at (-0.5, 0) {$*$};
		}
	}
}

\newcommand{\tikzfourf}{
	\mathord{
		\tikz[baseline={(0,0)}]{
			\useasboundingbox (-0.25,-1.5) rectangle (2,1.5);
			
			\coordinate (bot) at (0,-0.01);
			\coordinate (top) at (0,0.01);
			
			\fill[yellow!20] (top) to[out=-55,in=55,looseness=200] (bot) to[out=60,in=-60,looseness=650] (top);
			\draw[thick,black] (bot) to[out=60,in=-60,looseness=650] (top);
			\draw[thick,black] (top) to[out=-55, in=55,looseness=200] (bot);
			\node at (1.25, 0) {$\rho(G_1)$};
			
			\fill[black] (0,0) circle (2.5pt);
			
			\draw[thick, red] (0,0) circle (5pt);
			\node[thick, red] at (0.35, 0) {$*$};
		}
	}
}

\newcommand{\tikzfourff}{
	\mathord{
		\tikz[baseline={(0,0)}]{
			\useasboundingbox (-2.5,-1.5) rectangle (0.25,1.5);
			
			\coordinate (bot) at (0,-0.01);
			\coordinate (top) at (0,0.01);
			
			\fill[yellow!20] (top) to[out=135,in=-135,looseness=150] (bot) to[out=-120,in=120,looseness=650] (top);
			\draw[thick,black] (bot) to[out=-120,in=120,looseness=650] (top);
			\draw[thick,black] (top) to[out=135, in=-135,looseness=150] (bot);
			\node at (-1.25, 0) {$\rho(G_1)$};
			
			\fill[black] (0,0) circle (2.5pt);
			
			\draw[thick, red] (0,0) circle (5pt);
			\node[thick, red] at (-2.25, 0) {$*$};
		}
	}
}

\newcommand{\tikzfive}{
	\mathord{
		\tikz[baseline={(0,0)}]{
			\useasboundingbox (-0.25,-1.5) rectangle (2,1.5);
			
			\coordinate (top) at (0,1);
			\coordinate (mid) at (0,0);
			\coordinate (bot) at (0,-1);
			
			\fill[gray!30] (top) to[out=-10,in=10,looseness=1.5] (bot) to[out=-10,in=10,looseness=2.8] (top);
			\draw[thick,black] (top) to[out=-10,in=10,looseness=1.5] (bot);
			\draw[thick,black] (bot) to[out=-10,in=10,looseness=2.8] (top);
			\node at (1.2,0) {$G_2$};
			
			\fill[black] (top) circle (2.5pt);
			\fill[black] (mid) circle (2.5pt);
			\fill[black] (bot) circle (2.5pt);
			
			\draw[thick, black] (mid) edge[-Triangle] (0,0.9);
			
			\draw[thick, red] (mid) circle (5pt);
			\node[thick, red] at (0.5, 0) {$*$};
		}
	}
}

\newcommand{\tikzfivef}{
	\mathord{
		\tikz[baseline={(0,0)}]{
			\useasboundingbox (-0.5,-1.5) rectangle (2,1.5);
			
			\coordinate (top) at (0,1);
			\coordinate (mid) at (0,0);
			\coordinate (bot) at (0,-1);
			
			\fill[gray!30] (top) to[out=-10,in=10,looseness=1.5] (bot) to[out=-10,in=10,looseness=2.8] (top);
			\draw[thick,black] (top) to[out=-10,in=10,looseness=1.5] (bot);
			\draw[thick,black] (bot) to[out=-10,in=10,looseness=2.8] (top);
			\node at (1.2,0) {$G_2$};
			
			\fill[black] (top) circle (2.5pt);
			\fill[black] (mid) circle (2.5pt);
			\fill[black] (bot) circle (2.5pt);
			
			\draw[thick, black] (bot) edge[-Triangle] (0,-0.1);
			
			\draw[thick, red] (top) circle (5pt);
			\node[thick, red] at (-.35, 0) {$*$};
		}
	}
}

\newcommand{\tikzdone}{
	\mathord{
		\tikz[baseline={(0,0)}]{
			\useasboundingbox (-2.25,-2.4) rectangle (2.25,2.4);
			
			\draw[ultra thick, red!30] (-1.4,0.65) to[out=135,in=-90] (-1.75,1) to[out=90,in=90] (1.75,1) -- (1.75,-1) to[out=-90,in=-90] (-1.75,-1) -- (-1.75,0) to[out=90,in=-90] (-0.75,1) to[out=90,in=90] (0.75,1) -- (0.75,-1) to[out=-90,in=-90] (-0.75,-1) -- (-0.75,0) to[out=90,in=-45] (-1.1,0.35); 
			
			\draw[thick, black,fill=yellow!20] (-2,0) rectangle (-0.5,-1) node[pos=.5] {$T_1$};
			\draw[thick, black,fill=gray!30] (2,1) rectangle (0.5,-1) node[pos=.5] {$T_2$};
		}
	}
}

\newcommand{\tikzdtwo}{
	\mathord{
		\tikz[baseline={(0,0)}]{
			\useasboundingbox (-2.25,-2.4) rectangle (2.25,2.4);
			
			\draw[ultra thick, red!30] (-1.4,-0.35) to[out=135,in=-90] (-1.75,0) -- (-1.75,1) to[out=90,in=90] (1.75,1) -- (1.75,-1) to[out=-90,in=-90] (-1.75,-1) to[out=90,in=-90] (-0.75,0) -- (-0.75,1) to[out=90,in=90] (0.75,1) -- (0.75,-1) to[out=-90,in=-90] (-0.75,-1) to[out=90,in=-45] (-1.1,-0.65); 
			
			\draw[thick, black,fill=yellow!20] (-2,1) rectangle (-0.5,0) node[pos=.5] {$T_1$};
			\draw[thick, black,fill=gray!30] (2,1) rectangle (0.5,-1) node[pos=.5] {\reflectbox{$T_2$}};
		}
	}
}

\newcommand{\tikzdthree}{
	\mathord{
		\tikz[baseline={(0,0)}]{
			\useasboundingbox (-2.25,-2.4) rectangle (2.25,2.4);
			
			\draw[ultra thick, red!30] (1.1,-0.35) to[out=135,in=-90] (0.75,0) -- (0.75,1) to[out=90,in=90] (-0.75,1) -- (-0.75,-1) to[out=-90,in=-90] (0.75,-1) to[out=90,in=-90] (1.75,0) -- (1.75,1) to[out=90,in=90] (-1.75,1) -- (-1.75,-1) to[out=-90,in=-90] (1.75,-1) to[out=90,in=135] (1.4,-0.65);
			
			\draw[thick, black,fill=yellow!20] (2,1) rectangle (0.5,0) node[pos=.5] {\reflectbox{$T_1$}};
			\draw[thick, black,fill=gray!30] (-2,1) rectangle (-0.5,-1) node[pos=.5] {$T_2$};
		}
	}
}

\newcommand{\tikzdfour}{
	\mathord{
		\tikz[baseline={(0,0)}]{
			\useasboundingbox (-2.25,-2.4) rectangle (2.25,2.4);
			
				\draw[ultra thick, red!30] (-1.4,-0.35) to[out=135,in=-90] (-1.75,0) -- (-1.75,1) to[out=90,in=90] (1.75,1) -- (1.75,-1) to[out=-90,in=-90] (-1.75,-1) to[out=90,in=-90] (-0.75,0) -- (-0.75,1) to[out=90,in=90] (0.75,1) -- (0.75,-1) to[out=-90,in=-90] (-0.75,-1) to[out=90,in=-45] (-1.1,-0.65); 
			
			\draw[thick, black,fill=yellow!20] (-2,1) rectangle (-0.5,0) node[pos=.5] {\reflectbox{$T_1$}};
			\draw[thick, black,fill=gray!30] (2,1) rectangle (0.5,-1) node[pos=.5] {$T_2$};
		}
	}
}

\begin{document}
	\title[An Alexander Polynomial Refinement]{An Alexander Polynomial Refinement for Alternating Links, with Trapezoidal Properties}
	\author{Joe Boninger}
	\address{Department of Mathematics, Boston College, Chestnut Hill, MA}
	\email{boninger@bc.edu}
	\maketitle
	
	\begin{abstract}
		We define an invariant of alternating links---a homogeneous, four-variable Laurent polynomial---that encodes the symmetrized Alexander polynomial, the signature, and other topological data. Along the way, we extend a spanning tree formulation of the Alexander polynomial due to Murasugi and Stoimenow from special alternating links to all alternating links. This project is motivated by Fox's trapezoidal conjecture; accordingly, we prove certain sequences associated to our invariant are trapezoidal for all alternating links. We also conjecture our polynomial has $M$-convex support, and that it satisfies symmetry and log-concavity properties. We prove a partial symmetry result.
	\end{abstract}
	
	\section{Introduction}
	
	A link $K \subset S^3$ is called {\em alternating} if it admits a diagram such that crossings are met alternately at over- and underpasses when the link is traversed. In this paper we associate a four-variable polynomial invariant $P_K \in \Z[x^{-1},y^{-1},z,w]$ to any oriented alternating link. The invariant $P_K$ has many appealing qualities, including:
	\begin{enumerate}[label=(\roman*)]
		\item $P$ is multiplicative under connect sum, i.e.~$P_{K\#K'} = P_K P_{K'}$ for any two alternating links $K$ and $K'$.
		\item $P_K$ is homogeneous, with degree equal to the link signature $\sigma(K)$.
		\item $P_K$ specializes to the {\em symmetrized} Alexander polynomial $\Delta_K$ of $K$ in three ways:
		\begin{align*}
		\Delta_K(t) &= P_K(t^{-1/2}, -t^{1/2},t^{-1/2},-t^{1/2}) \\
		&= (-t)^{-\sigma(K)/2}P_K(-t,1,-t,1) \\
		&= (-t)^{-\sigma(K)/2}P_K(1,-t,1,-t).
		\end{align*}
		\item If $m(K)$ denotes the mirror of $K$, then 
		$$
		P_{m(K)}(x,y,z,w) = P_K(z^{-1},w^{-1}, x^{-1}, y^{-1}).
		$$
		In particular, if $K$ is amphichiral then $P_K(x,y,z,w) = P_K(z^{-1},w^{-1}, x^{-1}, y^{-1})$.
		\item If $-K$ denotes the link $K$ with the orientations of all components reversed, then
		$$
		P_{-K}(x,y,z,w) = P_K(y,x,w,z).
		$$
		In particular, if $K$ is invertible then $P_K(x,y,z,w) = P_K(y,x,w,z)$.
	\end{enumerate}
	In fact, we conjecture that the last equality of (v) holds for all alternating links.
	
	\begin{conj}
		\label{conj:symmetry}
		For any alternating link $K$,
		$$
		P_K(x,y,z,w) = P_K(y,x,w,z).
		$$
	\end{conj}
	
	We prove a result in this direction.
	
	\begin{thm}
		\label{thm:sym_one}
		For any alternating link $K$,
		$$
		P_K(x,y,1,1) = P_K(y,x,1,1)
		$$
		and
		$$
		P_K(1,1,z,w) = P_K(1,1,w,z).
		$$
	\end{thm}
	
	An alternating link is called {\em special} if it admits an alternating diagram with only positive crossings. For special alternating links, the polynomial $P_K$ is a certain homogenization of the usual Alexander polynomial $\Delta_K$---see Corollary \ref{cor:special_p} below. However, for the generic case of non-special alternating links, $P_K$ is a much richer invariant than $\Delta_K$. For example, in Section \ref{sec:examples} we show that the figure eight knot $4_1$ and eight-crossing alternating knot $8_{10}$ have
	$$
		P_{4_1} = 1 + x^{-1}z + y^{-1}z + x^{-1}w + y^{-1}w
	$$
	and
	\begin{align*}
		P_{8_{10}} &= (x^{-2} + x^{-1} y^{-1} + y^{-2})(w^4+w^3z+w^2z^2+wz^3+z^4) \\
		&+ (x^{-1} + y^{-1})(w+z)(w^2+wz+z^2).
	\end{align*}
	The first polynomial satisfies
	$$
	P_{4_1}(x,y,z,w) = P_{4_1}(z^{-1},w^{-1},x^{-1},y^{-1})
	$$
	while the second does not, reflecting the fact that $8_{10}$ is chiral.
	
	In Section \ref{sec:dimer_det} we explain how $P$ can be easily computed as the determinant of a matrix similar to the Fox matrix of a Dehn presentation of a link group. In spite of this simple formulation, it is remarkable to us that $P$ is a link invariant. The proof of invariance takes up one third of our paper.
	
	\subsection{Fox's trapezoidal conjecture}
	
	The problem of classifying alternating links was effectively solved with Menasco and Thistlethwaite's resolution of the flyping conjecture \cite{meth91}. Thus, the reader may want a reason to care about $P$ beyond the five properties listed above. The author's discovery of $P$ came from researching Fox's trapezoidal conjecture; for this, if a link $K$ has Alexander polynomial
	$$
		\Delta_K(t) = \sum_{i = 0}^n a_i t^i,
	$$
	then define its {\em unsigned Alexander polynomial} $|\Delta|_K$ by
	$$
		|\Delta|_K(t) = \sum_{i = 0}^n |a_i| t^i.
	$$
	Fox's conjecture states:
	
	\begin{conj}[\cite{fox61}]
		\label{conj:fox}
		For any alternating link $K$, the sequence of coefficients of $|\Delta|_K$ is {\em trapezoidal}. In other words, there exist indices $0 \leq i_1 \leq i_2 \leq n$ such that
		$$
		|a_0| < |a_1| < \dots < |a_{i_1}| = |a_{i_1 + 1}| = \cdots = |a_{i_2}| > \cdots > |a_{n - 1}| > |a_n|.
		$$
	\end{conj}
	
	A strengthening of Conjecture \ref{conj:fox}, due to Stoimenow, suggests the sequence is actually {\em log-concave}:
	
	\begin{conj}[\cite{sto05}]
		\label{conj:lc}
		For any alternating link $K$, the sequence of coefficients of $|\Delta|_K$ satisfies
		$$
		|a_i|^2 \geq |a_{i - 1}||a_{i + 1}|
		$$
		for all $i = 1, \dots, n - 1$. Furthermore, none of the $a_i$ are zero.
	\end{conj}
	
	It is an exercise to show Conjecture \ref{conj:lc} implies Conjecture \ref{conj:fox}.
	
	Fox's conjecture has received renewed attention in recent years due to groundbreaking work of June Huh and others on log-concave sequences in combinatorics. In particular, in 2020 Br\"and\'en and Huh defined {\em Lorentzian polynomials}, which are multi-variable, homogeneous polynomials satisfying strong log-concavity properties \cite{brhu20} (see also \cite{agv23, algv24, alogv24}). This theory has led to two proofs of Conjectures \ref{conj:fox} and \ref{conj:lc} for special alternating links, due to Hafner-M\'esz\'aros-Vidinas \cite{hmv24} and K\'alm\'an-M\'esz\'aros-Postnikov \cite{kmp25}, which follow the same strategy:
	\begin{enumerate}[label=\arabic*.]
		\item For any special alternating link $K$, define a multi-variable, homogeneous polynomial $p_K$ which specializes to $|\Delta|_K$.
		\item Prove that $p_K$ is {\em denormalized Lorentzian}.
	\end{enumerate}
	Following step two, one can use properties of Lorentzian polynomials to conclude that $|\Delta|_K$ is log-concave. Azarpendar, Juh\'asz and K\'alm\'an \cite{ajk24} also use a multi-variable Alexander polynomial refinement to prove trapezoidal-type inequalities for alternating 3-braids, though their polynomial is neither denormalized Lorentzian nor a link invariant in general.
	
	Our polynomial $P$ arose from an attempt to generalize K\'alm\'an, M\'esz\'aros, and Postnikov's work \cite{kmp25} from special alternating links to all alternating links. In this we were unsuccessful: although $P_K$ is a multi-variable, homogeneous polynomial specializing to $|\Delta|_K$, it is not always denormalized Lorentzian.\footnote{Technically, denormalized Lorentzian polynomials do not have negative exponents. Therefore, to make this discussion well defined, it is necessary to first multiply the polynomial $P_K$ by $x^ky^k$ for some sufficiently high $k$. The choice of (large enough) $k$ does not affect whether the result is denormalized Lorentzian, by \cite[Corollary 3.8]{brhu20}. We also omit the definition of (denormalized) Lorentzian polynomials since we do not discuss them after this, but refer the reader to \cite[Section 2.3]{hmv24} for a short summary.} Of the 563 alternating knots with eleven or fewer crossings, 138 have polynomials $P$ which are not denormalized Lorentzian---the polynomial $P_{8_{10}}$ above is one such failure. In spite of this, we prove certain sequences associated to $P$ are trapezoidal for {\em all} alternating links.
	
	\begin{thm}
		\label{thm:trap}
		For any alternating link $K$, each of the following sequences is trapezoidal:
		\begin{itemize}
			\item The even-degree coefficients of the polynomial $P_K(t^{-1}, t,1,1)$.
			\item The odd-degree coefficients of the polynomial $P_K(t^{-1}, t,1,1)$.
			\item The even-degree coefficients of the polynomial $P_K(1, 1,t^{-1},t)$.
			\item The odd-degree coefficients of the polynomial $P_K(1, 1,t^{-1},t)$.
		\end{itemize}
	\end{thm}
	The full coefficient sequences of $P_K(t^{-1},t,1,1)$ and $P_K(1,1,t^{-1},t)$ need not be trapezoidal---for example, the polynomial $P_{4_1}$ above has $P_{4_1}(t^{-1},t,1,1) = 2t^{-1} + 1 + 2t$.
	
	While there is no straight line from Theorem \ref{thm:trap} to Conjecture \ref{conj:fox}, the specializations in Theorem \ref{thm:trap} are similar to the specialization
	$$
	P_K(t^{-1},t,t^{-1},t) = |\Delta|_K(t^2)
	$$
	which we prove below. Theorem \ref{thm:trap} also implies some special cases of Fox's conjecture, though many of these have already been verified by Azarpendar, Juh\'asz and K\'alm\'an  \cite[Theorem 2.22]{ajk24}. For example, the polynomial $P_{8_{10}}$ has the form
	$$
	P_{8_{10}}(x,y,z,w) = Q_\text{even}(x,y)R_\text{even}(z,w) + Q_\text{odd}(x,y)R_\text{odd}(z,w),
	$$
	where $Q_\text{even}$ and $Q_\text{odd}$ are precisely the even and odd degree terms of $P_{8_{10}}(x,y,1,1)$, and $R_\text{even}$ and $R_\text{odd}$ are the even and odd degree terms of $P_{8_{10}}(1,1,z,w)$. Then
	$$
	|\Delta|_{8_{10}}(t^2) = Q_\text{even}(t^{-1},t)R_\text{even}(t^{-1},t) + Q_\text{odd}(t^{-1},t)R_\text{odd}(t^{-1},t).
	$$
	Since each polynomial on the right is trapezoidal and symmetric about zero (by Theorems \ref{thm:sym_one} and \ref{thm:trap}, or by inspection), one can show that $|\Delta|_{8_{10}}$ is as well, proving Fox's conjecture in this case. Unfortunately, the polynomial $P$ does not decompose this way for most alternating links.
		
	In addition to Theorem \ref{thm:trap}, we conjecture $P$ shares two key properties with denormalized Lorentzian polynomials---the first is {\em $M$-convex support}. The {\em support} of a Laurent polynomial $P(x,y,z,w)$ is the set of points
	$$
	\{(a,b,c,d) \mid x^ay^bz^cw^d \text{ appears with nonzero coefficient in $P$}\} \subset \Z^4.
	$$
	Separately, a subset $J \subset \Z^n$ is called {\em M-convex} if for any index $i$ and any $\alpha, \beta \in J$ whose $i$th coordinates satisfy $\alpha_i > \beta_i$, there is an index $j$ satisfying
	\begin{equation}
		\label{eq:m_convex}
		\alpha_j < \beta_j, \ \alpha - e_i + e_j \in J, \text{ and } \beta - e_j + e_i \in J,
	\end{equation}
	where $e_i$ denotes the $i$th standard basis vector. $M$-convex sets are closely related to generalized permutahedra in combinatorics, and (\ref{eq:m_convex}) can also be thought of as generalizing the basis exchange axiom from matroid theory. We posit:
	
	\begin{conj}
		\label{conj:m_convex}
		For any alternating link $K$, the polynomial $P_K$ has $M$-convex support in $\Z^4$.
	\end{conj}
	
	Second, $P$ appears to satisfy a log-concavity property.
	
	\begin{conj}
		\label{conj:lc_me}
		Let $K$ be an alternating link, and given $\alpha \in \Z^4$ let $c_\alpha$ denote the coefficient of $x^{\alpha_1}y^{\alpha_2}z^{\alpha_3}w^{\alpha_4}$ in $P_K$. Then for any such $\alpha$ and any $i, j \in \{1,2,3,4\}$, we have
		$$
		c^2_\alpha \geq c_{\alpha + e_i - e_j} c_{\alpha - e_i + e_j}.
		$$
	\end{conj}
	
	The properties in Conjectures \ref{conj:m_convex} and \ref{conj:lc_me} are both satisfied by denormalized Lorenztian polynomials: the first is part of the definition, and the second is \cite[Proposition 4.4]{brhu20}. We have also verified these conjectures, along with Conjecture \ref{conj:symmetry}, for all alternating knots with eleven crossings or fewer. This preponderance of evidence leads us to believe studying $P$ could yield insight into Fox's trapezoidal conjecture, and more broadly into the combinatorial and geometric structure of alternating links.
	
	\subsection{Overview and additional findings}
	
	In \cite{must03}, Murasugi and Stoimenow give a graph-theoretic interpretation of the Alexander polynomial of special alternating links. Let $K$ be a link with special alternating diagram $D$, with Tait graph $G \subset \R^2$---see below for definitions. Since $D$ is special alternating $G$ can be chosen to be Eulerian, and we fix orientations on $G$'s edges so they alternate between incoming and outgoing around each vertex in the plane. We also choose a root vertex $r$.
	
	Let $\mathcal{T}$ be the set of spanning trees of $G$ and given $T \in \mathcal{T}$, let $\iota(T)$ be the number of edges in $T$ which point away from the root $r$. Then Murasugi and Stoimenow prove:
	$$
	\Delta_K(t) = \sum_{T \in \mathcal{T}} (-t)^{\iota(T)}.
	$$
	K\'alm\'an, M\'esz\'aros and Postnikov \cite{kmp25} reinterpret $\iota(T)$ as an {\em activity measure}: an edge of a tree is {\em internally semi-active} if it points away from $r$, and {\em internally semi-passive} otherwise.
	
	Before defining our invariant $P$, we extend Murasugi and Stoimenow's result to all alternating links. To do so, we introduce a canonical orientation on the Tait graph of any alternating link diagram. We then recall a notion of {\em external semi-activity} $\varepsilon$ for edges not contained in spanning trees, which also appears in \cite{kmp25}. While internal semi-activity relies on a choice of root, our external semi-activity depends on a basepoint in $\R^2$. We then prove:
	
	\begin{thm}
		\label{thm:alex_ids}
		Let $K$ be a link with alternating diagram $D$, and let $G \subset \R^2$ be its Tait graph. Orient $G$ as described in Section \ref{sec:tait_digraph}, and fix a root vertex for $G$ and a basepoint in an adjacent face. Then the {\em symmetrized} Alexander polynomial of $K$, $\Delta_K$, is given by
		\begin{align*}
		\Delta_K(t) &= \Delta_K(t) = \sum_{T \in \mathcal{T}} (t^{1/2})^{\iota_+(T) - \varepsilon_-(T)} (-t^{1/2})^{\overline{\varepsilon}_-(T) - \overline{\iota}_+(T)} \\
		&= (-t)^{\sigma(K)/2} \sum_{T \in \mathcal{T}} (-t)^{\varepsilon_-(T) - \iota_+(T)} \\
		&= (-t)^{\sigma(K)/2} \sum_{T \in \mathcal{T}} (-t)^{\overline{\varepsilon}_-(T) - \overline{\iota}_+(T)},
		\end{align*}
		where $\sigma(K)$ is the signature of $K$.
	\end{thm}
	
	The quantities $\iota_+$ and $\overline{\iota}_+$ count internally semi-active and semi-passive edges through positive crossings of $D$, while $\varepsilon_-$ and $\overline{\varepsilon}_-$ count externally semi-active and semi-passive edges through negative crossings; see Definition \ref{def:plus_minus_acts} and the preceding discussion. Theorem \ref{thm:alex_ids} may be known to experts, as it isn't far from Kauffman's spanning tree formulation of the Alexander polynomial \cite{kauf83}. To our knowledge, however, our canonically oriented Tait graph has not appeared in this context before.
	
	In Section \ref{sec:act_poly} we define our polynomial $P_K$ as follows. Let $K$ be a link with alternating diagram $D$, and $G$ its (directed) Tait graph. Fix a root vertex for $G$ and an adjacent basepoint; then:
	$$
	P_K(x,y,z,w) = \sum_{T \in \mathcal{T}} x^{-\iota_+(T)} y^{- \overline{\iota}_+(T)} z^{\varepsilon_-(T)} w^{\overline{\varepsilon}_-(T)} \in \Z[x^{-1},y^{-1},z,w].
	$$
	The five properties of $P$ listed above follow from this definition.
	
	To prove $P$ is a well-defined link invariant we must show it depends neither on the choice of adjacent root and basepoint, nor on the alternating diagram. Our proof of the former uses a determinant formulation of $P$, while for diagram independence we employ the flyping theorem \cite{meth91}.
	
	Finally, our proofs of Theorems \ref{thm:sym_one} and \ref{thm:trap} rely on results of Hafner-M\'esz\'aros-Vidinas \cite[Theorem 1.3]{hmv25} and Gao-Yuan \cite[Corollary 1.3]{gayu26} respectively. The latter is a recent and deep theorem, whose proof relies on the $g$-theorem from polytope theory \cite{stan80, zie95}.
	
	\subsection{Outline}
	
	Section \ref{sec:back} contains background information. In Section \ref{sec:trees} we prove Theorem \ref{thm:alex_ids}, and in Section \ref{sec:def_props} we define the polynomial $P$ and verify the properties (i)--(v). In Section \ref{sec:dimer_det} we explain how $P$ can be defined as a determinant and as a weighted dimer count, and in Section \ref{sec:examples} we discuss the examples of the knots $4_1$ and $8_{10}$.
	
	We prove $P$ does not depend on the choice of root and basepoint in Section \ref{sec:root_indep}, and that $P$ does not depend on the choice of alternating diagram in Section \ref{sec:link_invar}. Finally, in Section \ref{sec:symmetry} we prove Theorems \ref{thm:sym_one} and \ref{thm:trap}.
	
	\subsection{Computer code}
	
	The author is happy to provide the Python programs used to check Conjectures \ref{conj:symmetry}, \ref{conj:m_convex} and \ref{conj:lc_me} over email, as well as a program which computes the polynomial $P$ from a planar diagram code. These programs were written primarily by Claude Sonnet 5, and tested on information from the KnotInfo database \cite{knotinfo}.
	
	\subsection{Acknowledgments}
	
	The author thanks Josh Greene and John Baldwin for encouraging him to think about log-concavity problems, and Matt Harper for a helpful conversation about Lorentzian polynomials.

	\section{Background and Conventions}
	\label{sec:back}
	
	\subsection{Graph theory}
	
	A {\em digraph} is a directed graph, and a {\em plane graph} (or digraph) is a planar graph with a fixed embedding in $\R^2$. Given a plane digraph $G$, we specify a digraph structure on its planar dual $G^*$ by orienting its edges according to the convention in Figure \ref{fig:dual_o}. For any edge $e$ or set of edges, we denote by $G / e$ the graph obtained by contracting $e$; similarly, $G \setminus e$ is the deletion of $e$ from $G$. If deleting a set of edges produces an isolated vertex, then we implicitly delete that vertex as well.
	
	\begin{figure}
		\begin{tikzpicture}
			\draw[thick, -Triangle] (-1,0) -- (1,0);
			\node at (-1,0.2) {$e$};
			
			\draw[thick, dashed, -Triangle] (0,1) -- (0,-1);
			\node at (0.3,1) {$e^*$};
		\end{tikzpicture}
	
		\caption{Orienting dual edges}
		\label{fig:dual_o}
	\end{figure}
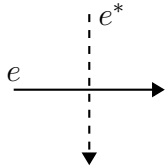
	
	Finally, we recall that a {\em spanning tree} of a graph $G = (V, E)$ is a maximal edge set which contains no cycles. Any two spanning trees of $G$ have the same size, which is called the {\em rank} of $G$ and denoted rk$(G)$. If $G$ is a plane graph, then $T$ is a spanning tree of $G$ if and only if its complement $E \setminus T$ is a spanning tree of $G^*$, making the natural identification of their edge sets. The quantity $|E \setminus T|$ is called the {\em corank} of $G$, and coincides with rk$(H_1(G))$.
	
	\subsection{Link diagrams, shadings and the Tait graph}
	
	We use the usual convention on crossing signs in link diagrams, shown in Figures \ref{fig:pos} and \ref{fig:neg}. We also assume:
	
	\begin{conv}
		\label{conv:split}
		All links are oriented and non-split unless stated otherwise.
	\end{conv}
	
	Let $D \subset \R^2$ be a link diagram, and $p(D) \subset \R^2$ the four-valent graph obtained by forgetting crossing information. The components of $\R^2 \setminus p(D)$, called {\em regions} of $D$, may be {\em checkerboard colored} in two possible ways by selectively shading them so that no pair of shaded regions or unshaded regions share a common boundary edge. Coloring a diagram this way endows each crossing with a {\em type}, either {\em type I} or {\em type II}, as shown in Figures \ref{fig:I} and \ref{fig:II}.\footnote{The type I and II terminology is usually used to mean something different (cf.~\cite[Figure 2]{gl78}), while our type I and II crossings are called type $a$ and $b$ crossings in \cite{g17}.} It's a useful fact, and straightforward to check, that alternating link diagrams are characterized by having only one crossing type.
	
	\begin{lemma}
		\label{lem:alt_crossing_signs} A (non-split) link diagram is alternating if and only if, after choosing a checkerboard shading, every crossing has the same type.
	\end{lemma}
	
	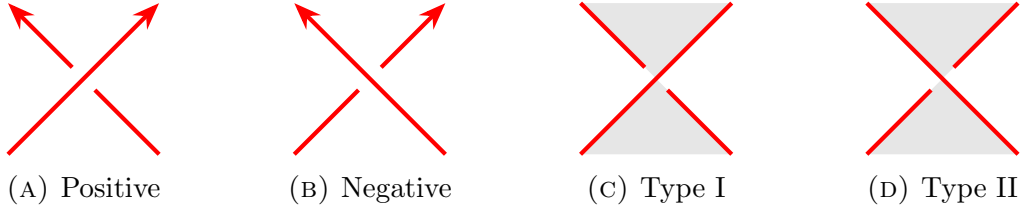
\begin{figure}
		\centering
		
		\begin{subfigure}[t]{0.22\textwidth}
			\centering
			\begin{tikzpicture}
				\draw[-{Stealth}, ultra thick, red] (-1,-1) -- (1,1);
				\draw[ultra thick, red] (1,-1) -- (0.15,-0.15);
				\draw[{Stealth}-, ultra thick, red] (-1,1) -- (-0.15,0.15);
			\end{tikzpicture}
			\caption{Positive}
			\label{fig:pos}
		\end{subfigure}
		\begin{subfigure}[t]{0.22\textwidth}
			\centering
			\begin{tikzpicture}
				\draw[-{Stealth}, ultra thick, red] (1,-1) -- (-1,1);
				\draw[ultra thick, red] (-1,-1) -- (-0.15,-0.15);
				\draw[{Stealth}-, ultra thick, red] (1,1) -- (0.15,0.15);
			\end{tikzpicture}
			\caption{Negative}
			\label{fig:neg}
		\end{subfigure}
		\begin{subfigure}[t]{0.22\textwidth}
			\centering
			\begin{tikzpicture}
				\fill[gray!20] (0,0) -- (1,1) -- (-1,1) --  cycle;
				\fill[gray!20] (0,0) -- (1,-1) -- (-1,-1) --  cycle;
				\draw[ultra thick, red] (-1,-1) -- (1,1);
				\draw[ultra thick, red] (1,-1) -- (0.15,-0.15);
				\draw[ultra thick, red] (-1,1) -- (-0.15,0.15);
			\end{tikzpicture}
			\caption{Type I}
			\label{fig:I}
		\end{subfigure}
		\begin{subfigure}[t]{0.22\textwidth}
			\centering
			\begin{tikzpicture}
				\fill[gray!20] (0,0) -- (1,1) -- (-1,1) --  cycle;
				\fill[gray!20] (0,0) -- (1,-1) -- (-1,-1) --  cycle;
				\draw[ultra thick, red] (1,-1) -- (-1,1);
				\draw[ultra thick, red] (-1,-1) -- (-0.15,-0.15);
				\draw[ultra thick, red] (1,1) -- (0.15,0.15);
			\end{tikzpicture}
			\caption{Type II}
			\label{fig:II}
		\end{subfigure}
		
		\caption{Crossing attributes}
		\label{fig:crossing_shading}
	\end{figure}
	
	If a checkerboard-colored alternating link diagram $D$ has type I crossings, then the opposite shading of $D$ will have type II crossings. We can therefore fix a shading for any alternating diagram.
	
	\begin{conv}
		\label{conv:type}
		For any alternating link diagram $D$, we choose the checkerboard shading with type I crossings.
	\end{conv}
	
	A checkerboard-colored link diagram can also be used to build a plane graph, called the {\em Tait graph}, by placing a vertex in each shaded region. We then draw an edge through each crossing of the diagram, connecting the vertices in its adjacent shaded regions as in Figure \ref{fig:tait}.
	
	\subsection{The symmetrized Alexander polynomial}
	\label{sec:sym_alex}
	
	Our exposition here follows \cite{msbv25}, but see \cite{codaru14, kauf83} for earlier references. Let $D \subset \R$ be any diagram of a link $K \subset S^3$, and $G$ its Tait graph. Let $G^*$ be the planar dual of $G$, and let $\mathcal{H} = \mathcal{H}(G)$ be the set
	$$
	\mathcal{H} = G^* \cup G \subset \R^2.
	$$
	We endow $\mathcal{H}$ with the structure of a bipartite graph, with vertex classes $\mathcal{V}_1$ and $\mathcal{V}_2$, as follows:
	\begin{itemize}
		\item The vertex set $\mathcal{V}_1$ is given by
		$$
		\mathcal{V}_1 = V(G) \cup V(G^*),
		$$
		and the set $\mathcal{V}_2$ is
		$$
		\mathcal{V}_2 = G \cap G^*.
		$$
		\item The edges of $\mathcal{H}$, which we denote $\mathcal{E}$, are then the components of
		$$
		\mathcal{H} \setminus (\mathcal{V}_1 \cup \mathcal{V}_2).
		$$
	\end{itemize}
	The vertices $\mathcal{V}_2$ are in bijection with edges of $G$, since the points of $G \cap G^*$ are precisely the intersections of edges with their duals. It follows that each edge of $G$ and each edge of $G^*$ yields two edges of $\mathcal{H}$.
	
	We call the bipartite graph $\mathcal{H} = (\mathcal{V}_1 \sqcup \mathcal{V}_2, \mathcal{E})$ the {\em double overlay} of $G$---see Figure \ref{fig:overlay} for an example. As the figure suggests, $\mathcal{H}$ can also be defined directly from the diagram $D$: each region of $D$ contains a $\mathcal{V}_1$ vertex, each crossing is a $\mathcal{V}_2$ vertex, and we connect each region's vertex to the crossings on its border.
	
	\begin{figure}
		\centering
		
		\begin{subfigure}[t]{0.3\textwidth}
			\centering
			\begin{tikzpicture}[use Hobby shortcut]
				\useasboundingbox (-1.5,-2) rectangle (1.5,2);
				
				\begin{knot}[
					consider self intersections=true,
					ignore endpoint intersections=false,
					flip crossing=8,
					]
					\strand ([closed]1,-1.5) .. (0,-1) .. (-0.5,-0.5) .. (0,0) .. (0.5,0.5) .. (0,1) .. (-1,1.5) .. (-1.5,0) .. (-1,-1.5) .. (0,-1) .. (0.5,-0.5) .. (0,0) .. (-0.5, 0.5) .. (0,1) .. (1, 1.5) .. (1.5,0) .. (1, -1.5);
				\end{knot}
				
				\path[fill=gray!200, opacity=0.1, even odd rule] ([closed]1,-1.5) .. (0,-1) .. (-0.5,-0.5) .. (0,0) .. (0.5,0.5) .. (0,1) .. (-1,1.5) .. (-1.5,0) .. (-1,-1.5) .. (0,-1) .. (0.5,-0.5) .. (0,0) .. (-0.5, 0.5) .. (0,1) .. (1, 1.5) .. (1.5,0) .. (1, -1.5);
				
				\draw[-{Stealth}, line width=1pt, red!30] (0.9,1.5) -- (1.1,1.5);
				\draw[-{Stealth}, line width=1pt, red!30] (-0.9,1.5) -- (-1.1,1.5);
				
				\draw[thick] (-1,0) -- (1,0);
				\draw[thick] (-1,0) to[out=60,in=180] (0,1);
				\draw[thick] (-1,0) to[out=-60,in=180] (0,-1);
				\draw[thick] (1,0) to[out=120,in=0] (0,1);
				\draw[thick] (1,0) to[out=-120,in=0] (0,-1);
				\draw[fill=black] (-1,0) circle (1.5pt);
				\draw[fill=black] (1,0) circle (1.5pt);
			\end{tikzpicture}
			\caption{The Tait graph $G$}
			\label{fig:tait}
		\end{subfigure}
		\begin{subfigure}[t]{0.3\textwidth}
			\centering
			\begin{tikzpicture}[use Hobby shortcut]
				\useasboundingbox (-1.5,-2) rectangle (1.5,2);
								
				\begin{knot}[
					consider self intersections=true,
					ignore endpoint intersections=false,
					flip crossing=8,
					]
					\strand ([closed]1,-1.5) .. (0,-1) .. (-0.5,-0.5) .. (0,0) .. (0.5,0.5) .. (0,1) .. (-1,1.5) .. (-1.5,0) .. (-1,-1.5) .. (0,-1) .. (0.5,-0.5) .. (0,0) .. (-0.5, 0.5) .. (0,1) .. (1, 1.5) .. (1.5,0) .. (1, -1.5);
				\end{knot}
				
				\draw[-{Stealth}, line width=1pt, red!30] (0.9,1.5) -- (1.1,1.5);
				\draw[-{Stealth}, line width=1pt, red!30] (-0.9,1.5) -- (-1.1,1.5);
				
				\draw[thick] (-1,0) -- (1,0);
				\draw[thick] (-1,0) to[out=60,in=180] (0,1);
				\draw[thick] (-1,0) to[out=-60,in=180] (0,-1);
				\draw[thick] (1,0) to[out=120,in=0] (0,1);
				\draw[thick] (1,0) to[out=-120,in=0] (0,-1);
				\draw[thick] (0,-1) -- (0,1);
				\draw[thick] (0,1) to[out=75,in=90,looseness=2] (2,0);
				\draw[thick] (0,-1) to[out=-75,in=-90,looseness=2] (2,0);
				\draw[fill=black] (-1,0) circle (1.5pt);
				\draw[fill=black] (1,0) circle (1.5pt);
				\draw[fill=black] (0,0.5) circle (1.5pt);
				\draw[fill=black] (0,-0.5) circle (1.5pt);
				\draw[fill=black] (2,0) circle (1.5pt);
				\draw[fill=white] (0,-1) circle (1.5pt);
				\draw[fill=white] (0,0) circle (1.5pt);
				\draw[fill=white] (0,1) circle (1.5pt);
			\end{tikzpicture}
			\caption{The double overlay $\mathcal{H}$}
			\label{fig:overlay}
		\end{subfigure}
		\begin{subfigure}[t]{0.3\textwidth}
			\centering
			\begin{tikzpicture}[use Hobby shortcut]
				\useasboundingbox (-1.5,-2) rectangle (1.5,2);
				
				\begin{knot}[
					consider self intersections=true,
					ignore endpoint intersections=false,
					flip crossing=8,
					]
					\strand ([closed]1,-1.5) .. (0,-1) .. (-0.5,-0.5) .. (0,0) .. (0.5,0.5) .. (0,1) .. (-1,1.5) .. (-1.5,0) .. (-1,-1.5) .. (0,-1) .. (0.5,-0.5) .. (0,0) .. (-0.5, 0.5) .. (0,1) .. (1, 1.5) .. (1.5,0) .. (1, -1.5);
				\end{knot}
				
				\draw[-{Stealth}, line width=1pt, red!30] (0.9,1.5) -- (1.1,1.5);
				\draw[-{Stealth}, line width=1pt, red!30] (-0.9,1.5) -- (-1.1,1.5);
				
				\draw[thick] (-1,0) -- node[above] {$1$} (0,0);
				\draw[thick] (-1,0) to[out=60,in=180] node[left] {$1$} (0,1);
				\draw[thick] (-1,0) to[out=-60,in=180] node[left] {$1$} (0,-1);
				\draw[thick] (0,-1) -- (0,1);
				\node[] at (0.5,0.25) {$-t^{1/2}$};
				\node[] at (0.5,0.75) {$t^{-1/2}$};
				\node[] at (0.5,-0.25) {$t^{-1/2}$};
				\node[] at (0.5,-0.75) {$-t^{1/2}$};
				\draw[fill=black] (-1,0) circle (1.5pt);
				\draw[fill=black] (0,0.5) circle (1.5pt);
				\draw[fill=black] (0,-0.5) circle (1.5pt);
				\draw[fill=white] (0,-1) circle (1.5pt);
				\draw[fill=white] (0,0) circle (1.5pt);
				\draw[fill=white] (0,1) circle (1.5pt);
			\end{tikzpicture}
			\caption{An $\mathcal{H}^{red}$ with weights}
			\label{fig:hred}
		\end{subfigure}
		
		\caption{Graphs for the left trefoil}
		\label{fig:double_dual}
	\end{figure}
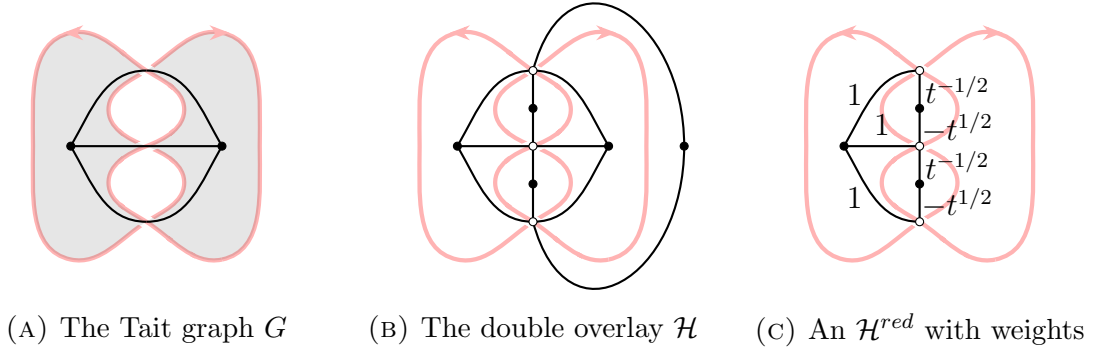
	
	Next, we endow the edges of $\mathcal{H}$ with weights in $\Z[t^{1/2}, t^{-1/2}]$ according to the crossing data of $D$, as shown in Figure \ref{fig:edge_weights}. We denote the weight function by
	$$
	\omega : \mathcal{E} \to \Z[t^{1/2}, t^{-1/2}].
	$$
	Finally, we fix two vertices $r, q \in \mathcal{V}_1$ such that $r \in V(G)$, $q \in V(G^*)$, and $r$ and $q$ are adjacent to a common vertex in $\mathcal{V}_2$. This last condition is equivalent to $r$ and $q$ being contained in adjacent regions of $D$, so we will sometimes refer to them as {\em adjacent vertices} even though they do not share an edge. We let
	$$
	\mathcal{H}^{red} = \mathcal{H}^{red}_{(r,q)}
	$$
	be the graph $\mathcal{H}$ with the vertices $r$ and $q$ deleted along with their adjacent edges, and we call $\mathcal{H}^{red}$ a {\em reduced double overlay} of $D$. The graph $\mathcal{H}^{red}$ inherits edge weights from $\mathcal{H}$, as well as the vertex bipartition
	$$
	V(\mathcal{H}^{red}) = (V_1 \setminus \{r, q\}) \sqcup V_2;
	$$
	Figure \ref{fig:hred} gives one example.
	
	\begin{figure}
		\centering
		
		\begin{subfigure}[t]{0.4\textwidth}
			\centering
			\begin{tikzpicture}
				\draw[-{Stealth}, ultra thick, red!30] (-1,-1) -- (1,1);
				\draw[ultra thick, red!30] (1,-1) -- (0.15,-0.15);
				\draw[{Stealth}-, ultra thick, red!30] (-1,1) -- (-0.15,0.15);
				
				\draw[thick] (0,0) -- node[above] {$1$} (1,0);
				\draw[thick] (0,0) -- node[above] {$1$} (-1,0);
				\draw[thick] (0,0) -- (0,1);
				\draw[thick] (0,0) -- (0,-1);
				\draw[fill=black] (-1,0) circle (1.5pt);
				\draw[fill=black] (1,0) circle (1.5pt);
				\draw[fill=black] (0,1) circle (1.5pt);
				\draw[fill=black] (0,-1) circle (1.5pt);
				\draw[fill=white] (0,0) circle (1.5pt);
				\node[anchor=west] at (1.1, 0.75) {$t^{1/2}$};
				\node[anchor=west] at (1.1, -0.75) {$-t^{-1/2}$};
				\draw[<-, thin] (0.1, 0.75) -- (1.1, 0.75);
				\draw[<-, thin] (0.1, -0.75) -- (1.1, -0.75);
			\end{tikzpicture}
			\caption{A positive crossing}
		\end{subfigure}
		\begin{subfigure}[t]{0.4\textwidth}
			\centering
			\begin{tikzpicture}
				\draw[-{Stealth}, ultra thick, red!30] (1,-1) -- (-1,1);
				\draw[ultra thick, red!30] (-1,-1) -- (-0.15,-0.15);
				\draw[{Stealth}-, ultra thick, red!30] (1,1) -- (0.15,0.15);
				
				\draw[thick] (0,0) -- node[above] {$1$} (1,0);
				\draw[thick] (0,0) -- node[above] {$1$} (-1,0);
				\draw[thick] (0,0) -- (0,1);
				\draw[thick] (0,0) -- (0,-1);
				\draw[fill=black] (-1,0) circle (1.5pt);
				\draw[fill=black] (1,0) circle (1.5pt);
				\draw[fill=black] (0,1) circle (1.5pt);
				\draw[fill=black] (0,-1) circle (1.5pt);
				\draw[fill=white] (0,0) circle (1.5pt);
				\node[anchor=west] at (1.1, 0.75) {$t^{-1/2}$};
				\node[anchor=west] at (1.1, -0.75) {$-t^{1/2}$};
				\draw[<-, thin] (0.1, 0.75) -- (1.1, 0.75);
				\draw[<-, thin] (0.1, -0.75) -- (1.1, -0.75);
			\end{tikzpicture}
			\caption{A negative crossing}
		\end{subfigure}
		
		\caption{Edge weights for $\mathcal{H}$}
		\label{fig:edge_weights}
	\end{figure}
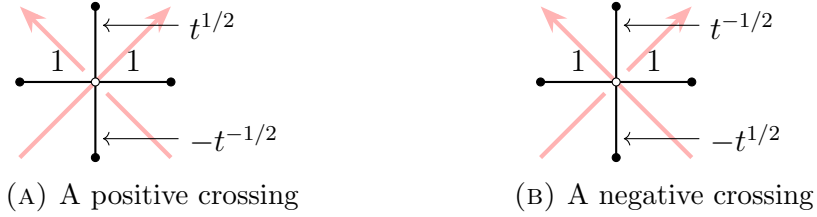
	
	A {\em dimer}, or {\em perfect matching}, of $\mathcal{H}^{red}$ is a set of edges $M$ such that each vertex of $\mathcal{H}^{red}$ meets exactly one edge in $M$. Let $\mathcal{M}$ be the set of dimers of $\mathcal{H}^{red}$; then Kauffman proved:
	
	\begin{thm}[\cite{kauf83}]
		\label{thm:kauff}
		With notation as above, the symmetrized Alexander polynomial of the link $K$ is given by
		\begin{equation}
			\label{eq:sym_delt}
			\Delta_K(t) = \sum_{M \in \mathcal{M}} \prod_{f \in M} \omega(f) \in \Z[t^{1/2},  t^{-1/2}].
		\end{equation}
		In particular, the above sum does not depend on the choice of adjacent vertices $r$ and $q$.
	\end{thm}
	
	This Alexander polynomial is called {\em symmetrized} because $\Delta_K(t) = \Delta_K(t^{-1})$ for any link, and readers unfamiliar with the Alexander polynomial may take Theorem \ref{thm:kauff} as a definition. To study the polynomial combinatorially, it will be convenient to ignore negative signs.
	
	\begin{defn}
		\label{def:unsigned_delt}
		For any link $K$, with Alexander polynomial
		$$
		\Delta_K(t) = \sum_{i = -n}^n a_i t^{i/2},
		$$
		we define the {\em unsigned Alexander polynomial} of $K$ by
		$$
		|\Delta|_K(t) = \sum_{i = -n}^n |a_i| t^{i/2}.
		$$
	\end{defn}
	
	For an alternating link, no information is lost in this simplification.
	
	\begin{lemma}
		\label{lem:unsigned_delt}
		Let $K$ be a link with alternating diagram $D$. Then, following the notation of Theorem \ref{thm:kauff},
		\begin{equation}
			\label{eq:unsigned_delt}
			|\Delta|_K(t) = \sum_{M \in \mathcal{M}} \prod_{f \in M} |\omega(f)|.
		\end{equation}
		Additionally, the coefficients of $\Delta_K$ alternate in sign, so $\Delta_K$ can be recovered from $|\Delta|_K$ by making every other coefficient negative.
	\end{lemma}
	
	\begin{rmk}
		\label{rem:indet_alex}
		Technically, Lemma \ref{lem:unsigned_delt} implies $\Delta_K$ can be recovered {\em up to sign} from $|\Delta|_K$. This is not an issue, since it is standard to consider the Alexander polynomial as defined up to multiplication by a unit of $\Z[t^{1/2}, t^{-1/2}]$.
	\end{rmk}
	
	Lemma \ref{lem:unsigned_delt} is well known to experts, and we sketch a proof for completeness.
	
	\begin{proof}
		Fix an alternating link diagram $D$. Let $m$ be a monomial summand on the right side of (\ref{eq:sym_delt}), corresponding to a dimer of $\mathcal{H}^{red}$, with deg$(m)$ the degree of $m$ in $t$ and sgn$(m) \in \{-1,1\}$ the sign of its coefficient. Then the statements in the lemma follow from the fact that
		\begin{equation}
			\label{eq:alt_alex_signs}
			\text{deg}(m) + \alpha \equiv \frac{\text{sgn}(m) + 1}{2} \ (\text{mod } 2),
		\end{equation}
		where $\alpha \in \{0, \pm\frac{1}{2}, 1\}$ is a fixed constant depending only on $D$. In particular, any two monomials with the same degree have the same sign, so there is no cancellation in the summation in (\ref{eq:sym_delt}). One way to prove (\ref{eq:alt_alex_signs}) is via Kauffman's ``clock theorem,'' which states that any two dimers of $\mathcal{H}$ are related by a sequence of ``clock moves'' \cite[Theorem 2.5]{kauf83}. One can check, using the alternating property of $D$, that applying a clock move to a dimer changes the degree of its weight by one and switches the sign.
	\end{proof}	
	
	\subsection{The spanning tree/dimer bijection}
	\label{sec:tree_dim_one}
	
	In addition to the perfect matching definition of the Alexander polynomial above, Kauffman gave an equivalent formulation using pairs of dual spanning trees \cite{kauf83}. For this let $D \subset \R^2$ be an alternating link diagram, $G$ its Tait graph, and
	$$
	\mathcal{H}^{red} = \mathcal{H}^{red}_{(r,q)} = ((\mathcal{V}_1 \setminus \{r, q\}) \sqcup \mathcal{V}_2, \mathcal{E}, \omega)
	$$
	a reduced double overlay with edges weights as in the preceding section. Then $r$ and $q$ are vertices in $V(G)$ and $V(G^*)$ respectively, which we view as distinguished {\em root vertices} of $G$ and $G^*$.
	
	Let $T$ be a spanning tree of $G$, and $T^*$ the complementary tree of $G^*$. For any edge $e \in T$ the subgraph $T \setminus \{e\}$ has two components, exactly one of which contains the root $r$. We call the end of $e$ which abuts the rooted component of $T \setminus \{e\}$ the {\em root end} of $e$, and the other end the {\em leaf end}. Similarly, we assign a root and leaf end to each edge $e' \in T^*$, based on which component of $T^* \setminus \{e'\}$ contains $q$.
	
	Recall that each edge $e \in E(G)$ and $e' \in E(G^*)$ gives two edges in the double overlay $\mathcal{H}$, with each new edge occupying half of the original. If $e \in T$ or $e' \in T^*$ for some fixed tree $T$, then we call the two corresponding edges of $\mathcal{H}$ the {\em root half} of $e$ (or $e'$) and the {\em leaf half} of $e$ (or $e'$), depending on which end of $e$ they occupy.
	
	Kauffman proved the following bijection between dimers and spanning trees, which was later re-proved independently by Kenyon, Propp and Wilson \cite{kpw00}.
	
	\begin{thm}[{\cite{kauf83}, \cite[Theorem 1]{kpw00}}]
		\label{thm:kpw}
		Let $\mathcal{T}$ be the set of spanning trees of a plane graph $G$, and let $\mathcal{M}$ be the set of perfect matchings of its reduced double overlay $\mathcal{H}^{red}_{(r,q)}$ for some adjacent pair of root vertices $r$ and $q$. Then there is a bijection 
		\begin{equation}
			\label{eq:bij}
			\Phi : \mathcal{T} \to \mathcal{M}
		\end{equation}
		defined as follows: for any tree $T \subset G$, let $T^*$ be the complementary tree of $G^*$. Then the matching $\Phi(T) \in \mathcal{M}$ is given by taking the leaf half of every edge in $T$ and $T^*$, as determined by $r$ and $q$.
	\end{thm}
	
	From Theorem \ref{thm:kpw} and Theorem \ref{thm:kauff}, we obtain another formulation of the symmetrized Alexander polynomial.
	
	\begin{thm}[\cite{kauf83}]
		\label{thm:tree_alex}
		Let $K$ be a link with diagram $D$. Let $G$ be the Tait graph of $D$, $G^*$ its planar dual, and $\mathcal{T}$ the set of spanning trees of $G$. Fix adjacent root vertices $r \in V(G)$ and $q \in V(G^*)$, and for any spanning tree $T$ of $G$ or $G^*$, let $\omega(T)$ denote the product of the weights of the leaf halves of all its edges in the double overlay $\mathcal{H}$. Then the symmetrized Alexander polynomial of $K$ is given by
		$$
		\Delta_K = \sum_{T \in \mathcal{T}} \omega(T)\omega(T^*),
		$$
		where $T^*$ indicates the dual of $T$.
	\end{thm}
	
	\begin{proof}
		Let $\Phi$ be the bijection of Theorem \ref{thm:kpw}. Then by definition we have
		$$
		\omega(T)\omega(T^*) = \prod_{f \in \Phi(T)} \omega(f)
		$$
		for any tree $T \in \mathcal{T}$. The result follows from Theorem \ref{thm:kauff} and the fact that $\Phi$ is a bijection.
	\end{proof}
	
	\section{A Spanning Tree Alexander Polynomial for Alternating Links}
	\label{sec:trees}
	
	\subsection{Orienting the Tait graph}
	\label{sec:tait_digraph}
	
	By Convention \ref{conv:type}, a Tait graph uniquely determines an alternating link up to orientation. To encode link orientations, we endow our Tait graphs with edge directions.
	
	\begin{defn}
		\label{def:ditait}
		Let $D$ be an oriented alternating link diagram, and $G$ its Tait graph. Then the {\em Tait digraph} is the plane digraph obtained by orienting each edge of $G$ to match the orientation of the corresponding overcrossing of $D$.
	\end{defn}
	
	Figure \ref{fig:ditait} shows one example. Although edge directions could be assigned to the Tait graph of any link diagram in this way, we will only consider Tait digraphs when the diagram is alternating.
	
	\begin{figure}
		\centering
		\begin{tikzpicture}[use Hobby shortcut]
			\useasboundingbox (-1.5,-2) rectangle (1.5,2);
			
			\begin{knot}[
				consider self intersections=true,
				ignore endpoint intersections=false,
				flip crossing=8,
				]
				\strand ([closed]1,-1.5) .. (0,-1) .. (-0.5,-0.5) .. (0,0) .. (0.5,0.5) .. (0,1) .. (-1,1.5) .. (-1.5,0) .. (-1,-1.5) .. (0,-1) .. (0.5,-0.5) .. (0,0) .. (-0.5, 0.5) .. (0,1) .. (1, 1.5) .. (1.5,0) .. (1, -1.5);
			\end{knot}
			
			\draw[-{Stealth}, line width=1pt, red!30] (0.9,1.5) -- (1.1,1.5);
			\draw[-{Stealth}, line width=1pt, red!30] (-0.9,1.5) -- (-1.1,1.5);
			
			\draw[thick, postaction={decorate}, decoration={
				markings,
				mark=at position 0.05 with {\arrowreversed{Triangle}}
			}] (-1,0) -- (1,0);
			\draw[thick, postaction={decorate}, decoration={
				markings,
				mark=at position 0.05 with {\arrowreversed{Triangle}}
			}] (-1,0) to[out=60,in=180] (0,1);
			\draw[thick, postaction={decorate}, decoration={
				markings,
				mark=at position 0.05 with {\arrowreversed{Triangle}}
			}] (-1,0) to[out=-60,in=180] (0,-1);
			\draw[thick] (1,0) to[out=120,in=0] (0,1);
			\draw[thick] (1,0) to[out=-120,in=0] (0,-1);
			\draw[fill=black] (-1,0) circle (1.5pt);
			\draw[fill=black] (1,0) circle (1.5pt);
		\end{tikzpicture}

		\caption{A Tait digraph for the left trefoil}
		\label{fig:ditait}
	\end{figure}
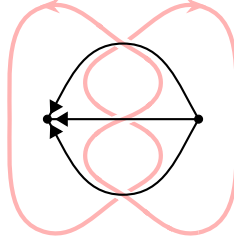
	
	\subsection{The spanning tree/dimer bijection for alternating links}
	\label{sec:dimer_tree}

	In this section we rewrite Theorem \ref{thm:tree_alex} for alternating diagrams, using our edge directions on the Tait graph. Since every crossing in an alternating diagram is type I by Convention \ref{conv:type}, the weights of edges around a crossing interact with shaded regions as shown in Figure \ref{fig:sign_comparison}. We've also added arrows to the figure to show the orientations of the relevant edges of the Tait digraph and its dual. Recalling that vertices of the Tait digraph correspond to shaded regions, Figure \ref{fig:sign_comparison} shows:
	
	\begin{lemma}
		\label{lem:weight_comp}
		Let $D \subset \R^2$ be an alternating link diagram with Tait digraph $G$, dual digraph $G^*$, and double overlay $\mathcal{H}$. Let $c$ be a crossing of $D$, with corresponding directed edges $e \in G$ and $e^* \in G^*$. Let $f_h$ and $f_t$ be the two edges in $\mathcal{H}$ contained in $e$, with $f_h$ at the head of $e$ and $f_t$ at the tail, and define $f'_h$ and $f'_t$ analogously for $e^*$. Then:
		\begin{enumerate}[label=(\roman*)]
			\item If $c$ is a positive crossing then $f_h$ has weight $t^{1/2}$, $f_t$ has weight $-t^{-1/2}$, and $f'_h$ and $f'_t$ both have weight $1$.
			\item If $c$ is a negative crossing then $f'_h$ has weight $t^{-1/2}$, $f'_t$ has weight $-t^{1/2}$, and $f_h$ and $f_t$ both have weight $1$.
		\end{enumerate}
	\end{lemma}
	
	\begin{figure}
		\centering
			
		\begin{subfigure}[t]{0.4\textwidth}
			\centering
			\begin{tikzpicture}
				\fill[gray!20] (0,0) -- (1,1) -- (-1,1) --  cycle;
				\fill[gray!20] (0,0) -- (1,-1) -- (-1,-1) --  cycle;
				\draw[-{Stealth}, ultra thick, red!30] (-1,-1) -- (1,1);
				\draw[ultra thick, red!30] (1,-1) -- (0.15,-0.15);
				\draw[{Stealth}-, ultra thick, red!30] (-1,1) -- (-0.15,0.15);
				
				\draw[thick, -Triangle] (0,0) -- node[above] {$1$} (0.95,0);
				\draw[thick] (0,0) -- node[above] {$1$} (-1,0);
				\draw[thick, -Triangle] (0,-1) -- (0,0.95);
				\draw[fill=black] (-1,0) circle (1.5pt);
				\draw[fill=black] (1,0) circle (1.5pt);
				\draw[fill=black] (0,1) circle (1.5pt);
				\draw[fill=black] (0,-1) circle (1.5pt);
				\draw[fill=white] (0,0) circle (1.5pt);
				\node[anchor=west] at (1.1, 0.75) {$t^{1/2}$};
				\node[anchor=west] at (1.1, -0.75) {$-t^{-1/2}$};
				\draw[<-, thin] (0.2, 0.75) -- (1.1, 0.75);
				\draw[<-, thin] (0.2, -0.75) -- (1.1, -0.75);
			\end{tikzpicture}
			\caption{The positive case}
		\end{subfigure}
		\begin{subfigure}[t]{0.4\textwidth}
			\centering
			\begin{tikzpicture}
				\fill[gray!20] (0,0) -- (1,1) -- (1,-1) --  cycle;
				\fill[gray!20] (0,0) -- (-1,1) -- (-1,-1) --  cycle;
				\draw[-{Stealth}, ultra thick, red!30] (1,-1) -- (-1,1);
				\draw[ultra thick, red!30] (-1,-1) -- (-0.15,-0.15);
				\draw[{Stealth}-, ultra thick, red!30] (1,1) -- (0.15,0.15);
				
				\draw[thick] (0,0) -- node[above] {$1$} (1,0);
				\draw[thick, -Triangle] (0,0) -- node[above] {$1$} (-0.95,0);
				\draw[thick, Triangle-] (0,0.95) -- (0,-1);
				\draw[fill=black] (-1,0) circle (1.5pt);
				\draw[fill=black] (1,0) circle (1.5pt);
				\draw[fill=black] (0,1) circle (1.5pt);
				\draw[fill=black] (0,-1) circle (1.5pt);
				\draw[fill=white] (0,0) circle (1.5pt);
				\node[anchor=west] at (1.1, 0.75) {$t^{-1/2}$};
				\node[anchor=west] at (1.1, -0.75) {$-t^{1/2}$};
				\draw[<-, thin] (0.2, 0.75) -- (1.1, 0.75);
				\draw[<-, thin] (0.2, -0.75) -- (1.1, -0.75);
			\end{tikzpicture}
			\caption{The negative case}
		\end{subfigure}
		\caption{Weights and edge directions near type I crossings}
		\label{fig:sign_comparison}
	\end{figure}
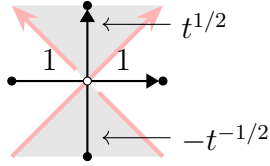
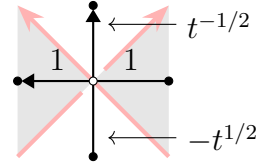
	
	This motivates the following definition.
	
	\begin{defn}
		\label{def:wt_fn}
		Let $D$ and $G$ be as above, and fix adjacent root vertices $r \in V(G)$ and $q \in V(G^*)$. Let $T$ be a spanning tree of $G$ with $T^*$ the complementary tree of $G^*$. We define a weight function
		$$
		\omega_T : E(G) \to \Z[t^{-1/2}, t^{1/2}],
		$$
		depending on $T$, as follows. For any edge $e \in E(G)$ with dual edge $e^* \in E(G^*)$, set:
		$$
		\omega_T(e) = \begin{cases}
			t^{1/2} & \text{$e \in T$, $e$ meets a positive crossing of $D$, $e$ points away from $r$ in $T$.} \\
			-t^{-1/2} & \text{$e \in T$, $e$ meets a positive crossing of $D$, $e$ points toward $r$ in $T$.}  \\
			t^{-1/2} & \text{$e \notin T$, $e$ meets a negative crossing of $D$, $e^*$ points away from $q$ in $T^*$.} \\
			-t^{1/2} & \text{$e \notin T$, $e$ meets a negative crossing of $D$, $e^*$ points toward $q$ in $T^*$.}  \\
			1 & \text{otherwise.}
		\end{cases}
		$$
	\end{defn}
	
	Then we have:
	
	\begin{lemma}
		\label{lem:init_alex}
		Let $D \subset \R^2$ be an alternating diagram of a link $K$, with Tait digraph $G$ and dual digraph $G^*$. Fix adjacent root vertices $r \in V(G)$ and  $q \in V(G^*)$, and let $\mathcal{T}$ be the set of spanning trees of $G$. Then the symmetrized Alexander polynomial of $K$ is given by
		$$
		\Delta_K(t) = \sum_{T \in \mathcal{T}} \prod_{e \in E(G)} \omega_T(e),
		$$
		where $\omega_T$ is the weight function of Definition \ref{def:wt_fn}.
	\end{lemma}
	
	\begin{proof}
		Using the notation of Section \ref{sec:tree_dim_one}, it's clear from Lemma \ref{lem:weight_comp} and Definition \ref{def:wt_fn} that for any edge $e \in T$, $\omega_T(e) = \omega(f)$, where $f$ is the leaf half of $e$ in $\mathcal{H}^{red}$. Similarly, if $e \notin T$ then $\omega_T(e) = \omega(f')$, where $f'$ is the leaf half of the dual edge $e^* \in E(G^*)$ in the complementary tree $T^*$. Therefore
		$$
		\prod_{e \in E(G)} \omega_T(e) = \omega(T)\omega(T^*)
		$$
		for any $T \in \mathcal{T}$, and the result follows from Theorem \ref{thm:tree_alex}.
	\end{proof}
	
	\subsection{Restatement using activities}
	\label{sec:act_alex}
	
	We can write Lemma \ref{lem:init_alex} more succinctly using spanning tree {\em activities}. Let $G = (E,V)$ be a digraph with a distinguished root $r \in V$, and $T$ a spanning tree of $G$. For any edge $e \in T$, as discussed above, the set $T \setminus \{e\}$ has two connected components $T_1$ and $T_2$. These determine a partition $V = V_1 \cup V_2$ on the vertices of $G$, where
	$$
	V_i = \{v \in V \mid v \in T_i\}
	$$
	for $i \in \{1,2\}$. The {\em fundamental cut} $\text{Cut}_T(e) \subset E$, determined by $e$ and $T$, is the set of edges which have one end in $V_1$ and the other in $V_2$. Dually, for any edge $e' \in E \setminus T$, the subgraph $T \cup \{e'\}$ contains a unique cycle. This cycle is the {\em fundamental cycle} determined by $e'$ and $T$, denoted $\text{Cyc}_T(e')$. By construction, $e \in \text{Cut}_T(e)$ and $e' \in \text{Cyc}_T(e')$.
	
	\begin{defn}
		\label{def:int}
		Using the notation above, we say an edge $e \in T$ is {\em internally semi-active (with respect to $T$)} if, as an element of $\text{Cut}_T(e)$, $e$ points away from the component $V_i \subset V$ containing the root vertex $r$. Otherwise, $e$ is {\em internally semi-passive}. We additionally define:
		\begin{align*}
			\text{Int}(T) &= \{e \in T \mid e \text{ is internally semi-active}\} \\
			\overline{\text{Int}}(T) &= \{e \in T \mid e \text{ is internally semi-passive}\} \\
			\iota(T) &= |\text{Int}(T)| \\
			\overline{\iota}(T) &= |\overline{\text{Int}}(T)|.
		\end{align*}
	\end{defn}
	
	Next, suppose one of the faces of $G$ is marked with a basepoint $q$. In this case we have a dual notion.
	
	\begin{defn}
		\label{def:ext}
		Let $e$ be an edge not in $T$, with fundamental cycle $\text{Cyc}_T(e)$, and orient $\text{Cyc}_T(e)$ so it runs counter-clockwise around $q$ in the one-point compactification $\R^2 \cup \{\infty\}$. Then $e$ is {\em externally semi-active} if the direction of $e$ agrees with this orientation of $\text{Cyc}_T(e)$. Otherwise we say $e$ is {\em externally semi-passive}, and we define:
		\begin{align*}
			\text{Ext}(T) &= \{e \in E \setminus T \mid e \text{ is externally semi-active}\} \\
			\overline{\text{Ext}}(T) &= \{e \in E \setminus T \mid e \text{ is externally semi-passive}\} \\
			\varepsilon(T) &= |\text{Ext}(T)| \\
			\overline{\varepsilon}(T) &= |\overline{\text{Ext}}(T)|.
		\end{align*}
	\end{defn}
	
	Definition \ref{def:ext} is simplest when $q$ lies in the unbounded face of $G$---in this case, a cycle of $G$ goes counter-clockwise around $q$ if it runs {\em clockwise} in the plane.
	
	\begin{rmk}
		The phrase {\em externally semi-active} appears in \cite{kmp25} with a similar meaning---see the proof of Lemma 5.2 in that paper, and our Lemma \ref{lem:int_ext} below. It is typically defined differently, however, as in \cite{kmp25, tot24}. These notions can be made to coincide when $G$ is a bipartite plane graph, but are generally distinct.
	\end{rmk}

	\begin{lemma}
		\label{lem:int_ext}
		Let $G$ be a plane digraph with dual digraph $G^*$, and fix roots $r \in V(G)$ and $q \in V(G^*)$. Dually, we consider $q$ as a basepoint of $G$ and $r$ as a basepoint of $G^*$. Let $T$ be a spanning tree of $G$, with dual tree $T^* \subset E(G^*)$. Then:
		\begin{enumerate}[label=(\roman*)]
			\item An edge $e \in E(G) \setminus T$ is externally semi-active with respect to $T$ if and only if its dual edge $e^* \in E(G^*)$ is internally semi-active with respect to $T^*$.
			\item An edge $e \in T$ is internally semi-active with respect to $T$ if and only if its dual $e^*$ is externally semi-passive with respect to $T^*$.
		\end{enumerate}
	\end{lemma}
	
	\begin{proof}
		Assume without loss of generality that $q$ lies in the unbounded face of $G$. Fix a spanning tree $T \subset E(G)$ and an edge $e \in E(G) \setminus T$. By our convention on dual orientations, the edge $e$ agrees with the clockwise orientation on $\text{Cyc}_T(e)$ in the plane (which is the counter-clockwise orientation around $q$) if and only if the dual edge $e^*$ points into the disk bounded by $\text{Cyc}_T(e)$. Since $q$ lies in the unbounded face of $\R^2 \setminus G$, this occurs exactly when $e^*$ is internally semi-active with respect to $T^*$. The second statement can be proved analogously---though note the asymmetry---by switching the roles of $G$ and $G^*$.
	\end{proof}
	
	For the special case of Tait digraphs, we introduce additional notation.
	
	\begin{defn}
		\label{def:plus_minus_acts}
		Let $D \subset \R^2$ be an alternating link diagram with Tait digraph $G = (V,E)$. Then we denote the edges of $G$ which pass through positive and negative crossings of $D$ by $E_+$ and $E_-$ respectively, and we refer to these as  {\em positive} and {\em negative edges}.\footnote{This should not be confused with the usual convention for Tait graphs, where edges are given signs based on crossing type rather than crossing sign.} For any spanning tree $T$ of $G$, we define
		\begin{align*}
			\iota_+(T) &= |\text{Int}(T) \cap E_+| \\
			\overline{\iota}_+(T) &= |\overline{\text{Int}}(T) \cap E_+| \\
			\varepsilon_-(T) &= |\text{Ext}(T) \cap E_-| \\
			\overline{\varepsilon}_-(T) &= |\overline{\text{Ext}}(T) \cap E_-|.
		\end{align*}
	\end{defn}
	
	We can now reframe Lemma \ref{lem:init_alex}.
	
	\begin{prop}
		\label{prop:activity_alex}
		Let $K$ be a link with alternating diagram $D \subset \R^2$, and $G$ its Tait digraph. Choose a root vertex $r \in V(G)$ and a basepoint $q$ in a face adjacent to $r$, and let $\mathcal{T}$ be the set of spanning trees of $G$. Then
		\begin{equation}
			\label{eq:sym_alex_one}
			\Delta_K(t) = \sum_{T \in \mathcal{T}} (t^{1/2})^{\iota_+(T) - \varepsilon_-(T)} (-t^{1/2})^{\overline{\varepsilon}_-(T) - \overline{\iota}_+(T)}.
		\end{equation}
	\end{prop}
		
	\begin{proof}
		Using Definitions \ref{def:int} and \ref{def:ext} along with Lemma \ref{lem:int_ext}, the weight function $\omega_T$ of Definition \ref{def:wt_fn} can be rewritten:
		$$
		\omega_T(e) = \begin{cases}
			t^{1/2} & \text{$e$ is positive and internally semi-active with respect to $r$ and $T$.} \\
			-t^{-1/2} & \text{$e$ is positive and internally semi-passive.}  \\
			t^{-1/2} & \text{$e$ is negative and externally semi-active with respect to $q$ and $T$.} \\
			-t^{1/2} & \text{$e$ is negative and externally semi-passive.}  \\
			1 & \text{otherwise.}
		\end{cases}
		$$
		Thus, for any spanning tree $T \in \mathcal{T}$,
		$$
		\prod_{e \in E(G)} \omega_T(e) = (t^{1/2})^{\iota_+(T) - \varepsilon_-(T)}(-t^{1/2})^{ \overline{\varepsilon}_-(T) - \overline{\iota}_+(T)},
		$$
		and we can substitute the right side into Lemma \ref{lem:init_alex}.
	\end{proof}
	
	 This identity can be simplified further using the next lemma.
	
	\begin{lemma}
		\label{lem:murasugi}
		With notation as in Proposition \ref{prop:activity_alex}, let $T$ be a spanning tree of $G$. Then
		$$
		\sigma(K) = \varepsilon_-(T) + \overline{\varepsilon}_-(T) - \iota_+(T) - \overline{\iota}_+(T),
		$$
		where $\sigma(K)$ is the signature of the link $K$.
	\end{lemma}
	
	\begin{proof}
		We first compute
		$$
		\text{rk}(G) = |T| = |T \cap E_+| + |T \cap E_-| = \big(\iota_+(T) + \overline{\iota}_+(T)\big) + \big(|E_-| -\varepsilon_-(T) -\overline{\varepsilon}_-(T)\big),
		$$
		so that
		$$
		\varepsilon_-(T) + \overline{\varepsilon}_-(T) - \iota_+(T) - \overline{\iota}_+(T) = |E_-| - \text{rk}(G) = \text{corank}(G) - |E_+|.
		$$

		On the other hand, let $S \subset S^3$ be the {\em checkerboard surface} corresponding to $G$. By definition, $S$ is constructed by connecting the shaded regions of $D$ with half-twists, as determined by crossings, so that $\partial S = K$. A famous result of Gordon and Litherland \cite[Theorem 6$''$]{gl78} states
		$$
		\sigma(K) = \sigma(S) - \#\{\text{positive type I crossings of $D$}\} + \#\{\text{negative type II crossings of $D$}\},
		$$
		where $\sigma(S)$ is the signature of the Gordon-Litherland pairing of $S$ (see \cite{gl78}). All crossings of $D$ are type I by Convention \ref{conv:type}, so the last term above vanishes and the second last equals $|E_+|$. Additionally, since $S$ is the type I checkerboard surface of an alternating link diagram, its Gordon-Litherland form is positive definite \cite[Proposition 4.1]{g17}. Therefore
		$$
		\sigma(S) = \text{rk}(H_1(S)) = \text{rk}(H_1(G)) = \text{corank}(G),
		$$
		and we conclude that 
		$$
		\sigma(K) = \text{corank}(G) - |E_+| = \varepsilon_-(T) + \overline{\varepsilon}_-(T) - \iota_+(T) - \overline{\iota}_+(T).
		$$
	\end{proof}
	
	\begin{thm}
		\label{thm:activity_alex}
		Let $K$ be oriented link with alternating diagram $D$. With notation as in Proposition \ref{prop:activity_alex}, we have
		$$
		\Delta_K(t) = (-t)^{-\sigma(K)/2} \sum_{T \in \mathcal{T}} (-t)^{\varepsilon_-(T) - \iota_+(T)} = (-t)^{-\sigma(K)/2} \sum_{T \in \mathcal{T}} (-t)^{\overline{\varepsilon}_-(T) - \overline{\iota}_+(T)}.
		$$
	\end{thm}
	
	\begin{proof}
		By Lemma \ref{lem:unsigned_delt} and Proposition \ref{prop:activity_alex}, the {\em unsigned} symmetrized Alexander polynomial is equal to
		$$
			|\Delta|_K(t) = \sum_{T \in \mathcal{T}} (t^{1/2})^{\iota_+(T) - \overline{\iota}_+(T) - \varepsilon_-(T) + \overline{\varepsilon}_-(T)}.
		$$
		Applying Lemma \ref{lem:murasugi} gives
		\begin{equation}
			\label{eq:ref_one}
		\iota_+(T) - \overline{\iota}_+(T) - \varepsilon_-(T) + \overline{\varepsilon}_-(T) = \sigma(K) + 2( \iota_+(T) - \varepsilon_-(T)),
		\end{equation}
		or alternatively
		\begin{equation}
			\label{eq:ref_two}
		\iota_+(T) - \overline{\iota}_+(T) - \varepsilon_-(T) + \overline{\varepsilon}_-(T) = -\sigma(K) + 2( \overline{\varepsilon}_-(T) - \overline{\iota}_+(T)).
		\end{equation}
		From the first identity, along with the symmetry $\Delta(t) = \Delta(t^{-1})$, we obtain
		$$
		|\Delta|_K(t) = t^{-\sigma(K)/2} \sum_{T \in \mathcal{T}} t^{\varepsilon_-(T) - \iota_+(T)},
		$$
		while the second identity yields
		$$
		|\Delta|_K(t) = t^{-\sigma(K)/2} \sum_{T \in \mathcal{T}} t^{\overline{\varepsilon}_-(T) - \overline{\iota}_+(T)}.
		$$
		The result follows from these by once again applying Lemma \ref{lem:unsigned_delt}.
	\end{proof}
	
	A link is {\em special alternating} if it admits an alternating diagram in which all crossings are positive---this diagram is also called a {\em special alternating diagram}. In the case of special alternating links, Theorem \ref{thm:activity_alex} recovers a result of Murasugi and Stoimenow.
	
	\begin{cor}[{\cite[Theorem 2]{must03}}]
		\label{cor:alex_special}
		Let $D$ be an oriented, special alternating diagram of a link $K$, and let $G = (V,E)$ be its Tait digraph. Then the symmetrized Alexander polynomial of $K$ is given by
		$$
			\Delta_K(t)  = (-t)^{\sigma(K)/2} \sum_{T \in \mathcal{T}} (-t)^{\iota(T)}.
		$$
	\end{cor}
	
	\begin{proof}
		By definition $E = E_+$, so $\overline{\varepsilon}_-(T) = 0$ and $\iota_+(T) = \iota(T)$ for any tree $T$.  We then consider the first equality of Theorem \ref{thm:activity_alex}, and use the symmetry $\Delta_K(t) = \Delta_K(t^{-1})$.
	\end{proof}
	
	An analogous statement using external activity holds if all crossings of $D$ are negative.
	
	\subsection{A digression}
	\label{sec:dig}
	
	The weights in Figure \ref{fig:edge_weights} are not the only ones which can be used to compute the Alexander polynomial as a weighted dimer count. An alternate set of weights, shown in Figure \ref{fig:alt_weights}, computes a {\em non-symmetrized} version \cite{kauf83, codaru14} (cf.~Remark \ref{rem:indet_alex}). Carrying out the above spanning tree reformulation with this other weight function, one obtains the identity
	\begin{equation}
		\label{eq:non_sym_alex}
		\Delta_K(t) = \sum_{T \in \mathcal{T}} (-t)^{\iota(T) + \overline{\varepsilon}(T)} = \sum_{T \in \mathcal{T}} (-t)^{\iota(T) + \text{corank}(G) - \varepsilon(T)} \sim \sum_{T \in \mathcal{T}} (-t)^{\iota(T) - \varepsilon(T)}.
	\end{equation}
	Here $K$ is a link with alternating diagram $D$, $G$ is the (rooted, basepointed) Tait graph of $D$, and $\mathcal{T}$ is its set of spanning trees. The symbol $\sim$ indicates that the two quantities agree up to multiplication by a unit of $\Z[t,t^{-1}]$---as Remark \ref{rem:indet_alex} discusses, this is a typical indeterminacy for the Alexander polynomial.
	
	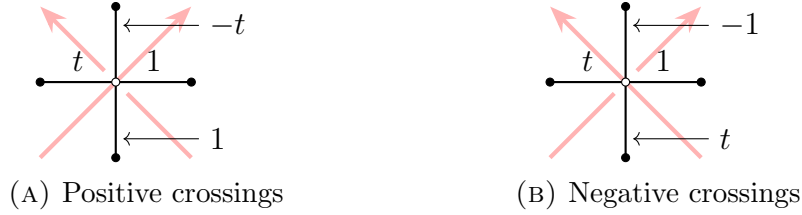
\begin{figure}
		\centering
		
		\begin{subfigure}[t]{0.4\textwidth}
			\centering
			\begin{tikzpicture}
				\draw[-{Stealth}, ultra thick, red!30] (-1,-1) -- (1,1);
				\draw[ultra thick, red!30] (1,-1) -- (0.15,-0.15);
				\draw[{Stealth}-, ultra thick, red!30] (-1,1) -- (-0.15,0.15);
				
				\draw[thick] (0,0) -- node[above] {$1$} (1,0);
				\draw[thick] (0,0) -- node[above] {$t$} (-1,0);
				\draw[thick] (0,0) -- (0,1);
				\draw[thick] (0,0) -- (0,-1);
				\draw[fill=black] (-1,0) circle (1.5pt);
				\draw[fill=black] (1,0) circle (1.5pt);
				\draw[fill=black] (0,1) circle (1.5pt);
				\draw[fill=black] (0,-1) circle (1.5pt);
				\draw[fill=white] (0,0) circle (1.5pt);
				\node[anchor=west] at (1.1, 0.75) {$-t$};
				\node[anchor=west] at (1.1, -0.75) {$1$};
				\draw[<-, thin] (0.1, 0.75) -- (1.1, 0.75);
				\draw[<-, thin] (0.1, -0.75) -- (1.1, -0.75);
			\end{tikzpicture}
			\caption{Positive crossings}
		\end{subfigure}
		\begin{subfigure}[t]{0.4\textwidth}
			\centering
			\begin{tikzpicture}
				\draw[-{Stealth}, ultra thick, red!30] (1,-1) -- (-1,1);
				\draw[ultra thick, red!30] (-1,-1) -- (-0.15,-0.15);
				\draw[{Stealth}-, ultra thick, red!30] (1,1) -- (0.15,0.15);
				
				\draw[thick] (0,0) -- node[above] {$1$} (1,0);
				\draw[thick] (0,0) -- node[above] {$t$} (-1,0);
				\draw[thick] (0,0) -- (0,1);
				\draw[thick] (0,0) -- (0,-1);
				\draw[fill=black] (-1,0) circle (1.5pt);
				\draw[fill=black] (1,0) circle (1.5pt);
				\draw[fill=black] (0,1) circle (1.5pt);
				\draw[fill=black] (0,-1) circle (1.5pt);
				\draw[fill=white] (0,0) circle (1.5pt);
				\node[anchor=west] at (1.1, 0.75) {$-1$};
				\node[anchor=west] at (1.1, -0.75) {$t$};
				\draw[<-, thin] (0.1, 0.75) -- (1.1, 0.75);
				\draw[<-, thin] (0.1, -0.75) -- (1.1, -0.75);
			\end{tikzpicture}
			\caption{Negative crossings}
		\end{subfigure}
		
		\caption{Non-symmetrized edge weights}
		\label{fig:alt_weights}
	\end{figure}

	We prefer Theorem \ref{thm:activity_alex} over (\ref{eq:non_sym_alex}) in this paper for its stronger symmetry and invariance properties---for example, changing the root or basepoint in (\ref{eq:non_sym_alex}) can shift the result by a unit of $\Z[t, t^{-1}]$. However, (\ref{eq:non_sym_alex}) does lead to an amusing observation. For any plane digraph $G$ with a fixed vertex and basepoint, one can define a two-variable polynomial $Q_G$ by
	$$
	Q_G(x,y) = \sum_{T \in \mathcal{T}} x^{\iota(T)}y^{\varepsilon(T)}.
	$$
	Here $\mathcal{T}$ is the set of spanning trees as usual, and we've suppressed the root and basepoint from our notation. Equation (\ref{eq:non_sym_alex}) then says that if $G$ is the Tait digraph of an alternating diagram of a link $K$, and if the root and basepoint are adjacent, then
	\begin{equation}
		\label{eq:alex_tutte}
		\Delta_K(t) = Q_G(-t, -t^{-1}).
	\end{equation}
	
	We now recall that the {\em Tutte polynomial}, a foundational invariant in graph theory, can be defined for any (undirected, non-rooted) graph $G$ by
	$$
	\widehat{T}_G(x,y) = \sum_{T \in \mathcal{T}} x^{\widehat{\iota}(T)}y^{\widehat{\varepsilon}(T)}.
	$$
	The quantities $\widehat{\iota}$ and $\widehat{\varepsilon}$ are different notions of internal and external activity which depend on a choice of linear ordering of the edges, though the resulting polynomial does not \cite{tut04}; we omit the definitions. It is a famous result of Thistlethwaite that if $K$ is a link with alternating diagram $D$, and $G$ is the (undirected) Tait graph of $D$, then the Jones polynomial $J_K$ of $K$ satisfies
	\begin{equation}
		\label{eq:jones_tutte}
		J_K(t) = \widehat{T}_G(-t, -t^{-1})
	\end{equation}
	up to multiplication by a unit of $\Z[t,t^{-1}]$ \cite{t87}. We do not have an explanation for the similarity between (\ref{eq:alex_tutte}) and (\ref{eq:jones_tutte}).
	
\section{The Alexander Polynomial Refinement}
	\label{sec:act_poly}
	
	\subsection{Definition and properties}
	\label{sec:def_props}
	
	Our work in the previous section suggests a generalization of the Alexander polynomial for alternating links.
	
	\begin{defn}
		\label{def:poly}
		Let $K$ be a link with alternating diagram $D$, and $G$ its Tait digraph. Fix a root $r$ of $G$ and a basepoint $q \in \R^2 \setminus G$ in a face adjacent to $r$, and let $\mathcal{T}$ be the set of all spanning trees of $G$. Then we define $P_K \in \Z[x^{-1}, y^{-1}, z, w]$ to be the Laurent polynomial
		$$
			P_K(x, y, z, w) = \sum_{T \in \mathcal{T}} x^{-\iota_+(T)} y^{-\overline{\iota}_+(T)} z^{\varepsilon_-(T)} w^{\overline{\varepsilon}_-(T)},
		$$
		where the quantities $\iota_+(T)$, $\overline{\iota}_+(T)$, $\varepsilon_-(T)$ and $\overline{\varepsilon}_-(T)$ are as in Definition \ref{def:plus_minus_acts}. Additionally, for any tree $T \in \mathcal{T}$, we call $x^{-\iota_+(T)} y^{-\overline{\iota}_+(T)} z^{\varepsilon_-(T)}w^{\overline{\varepsilon}_-(T)}$ the {\em weight} of $T$ in $P$.
	\end{defn}
	
	It will sometimes be convenient to be explicit about our choice of alternating diagram $D$ or root $r$ and basepoint $q$, or to think of $P$ as a function of the Tait digraph $G$. In these situations, we write $P_K$ variously as $P_D$, $P_G$, or $P_{(G,r,q)}$. Our choice to suppress this data from our notation in Definition \ref{def:poly} is justified by the next theorem, which we will prove in Sections \ref{sec:root_indep} and \ref{sec:link_invar}.
	
	\begin{thm}
		\label{thm:invariant}
		The polynomial $P_K$ does not depend on the choice of alternating diagram, or on the choice of adjacent root and basepoint. In other words, $P_K$ is an invariant of the alternating link $K$.
	\end{thm}
	
	The signs of the exponents in Definition \ref{def:poly} are motivated by Lemma \ref{lem:murasugi}, which implies:
	
	\begin{prop}
		\label{prop:hom}
		The polynomial $P_K$ is homogeneous for any alternating link $K$, with degree equal to its signature $\sigma(K)$.
	\end{prop}
	
	Furthermore, by Proposition \ref{prop:activity_alex} and Theorem \ref{thm:activity_alex}, the Alexander polynomial can be recovered from $P_G$ in no less than three distinct ways.
	
	\begin{prop}
		\label{prop:three_alex}
		For any alternating link $K$ with symmetrized Alexander polynomial $\Delta_K$, we have
		\begin{align*}
		\Delta_K(t) &= (-t)^{-\sigma(K)/2}P_K(-t, 1, -t, 1), \\
		\Delta_K(t) &= (-t)^{-\sigma(K)/2}P_K(1, -t, 1, -t), \text{ and} \\
		\Delta_K(t) &= P_K(t^{-1/2}, -t^{1/2}, t^{-1/2}, -t^{1/2}).
		\end{align*}
	\end{prop}
	
	Three additional identities are given by applying the involution $t \mapsto t^{-1}$ to the three above. If $K$ is a special alternating link, then $P_K$ is completely determined by $\Delta_K$ and $\sigma(K)$.
	
	\begin{cor}
		\label{cor:special_p}
		Let $K$ be a link with special alternating diagram $D$, and suppose the symmetrized Alexander polynomial of $K$ is given by
		$$
		\Delta_K = \sum_{i = -n}^n a_i t^{i/2}.
		$$
		for some $n \geq 0$ and integers $a_{-n}, a_{1 - n}, \dots, a_n$. Then
		$$
		P_K(x,y,z,w) = P_K(x,y,1,1) = x^{\sigma(K)/2}y^{\sigma(K)/2}\sum_{i = -n}^n |a_i| x^{i/2}y^{-i/2}.
		$$
	\end{cor}
	
	\begin{proof}
		The diagram $D$ has no negative crossings, so $P_K$ is independent of $z$ and $w$. Using this, the first identity in Proposition \ref{prop:three_alex}, and Lemma \ref{lem:unsigned_delt}, we have
		$$
		P_K(x, 1, z, w) = x^{\sigma(K)/2}|\Delta|_K(x) = x^{\sigma(K)/2}\sum_{i = -n}^n |a_i| x^{i/2}.
		$$
		Since $P_K$ is homogeneous with degree $\sigma(K)$, we conclude that
		$$
		P_K(x,y,z,w) = x^{\sigma(K)/2}y^{\sigma(K)/2}\sum_{i = -n}^n |a_i| x^{i/2}y^{-i/2}.
		$$
	\end{proof}
	
	As with Corollary \ref{cor:alex_special}, a similar identity using $z$ and $w$ can be obtained if $K$ admits an alternating diagram with all negative crossings. Next, we show that $P$ is multiplicative under connect sum.
	
	\begin{prop}
		\label{prop:sum}
		For any two alternating links $K$ and $K'$,
		$$
		P_{K \# K'} = P_K P_{K'},
		$$
		where $K \# K'$ denotes the connect sum along any component of $K$ and any component of $K'$.
	\end{prop}
	
	\begin{proof}
		Let $D \subset \R^2$ be an alternating diagram of $K$, $D'$ an alternating diagram of $K'$, and $D \# D'$ the diagrammatic connect sum of $D$ and $D'$, representing $K \# K'$. Let $G$ be the Tait digraph of $D$, and $G'$ the Tait digraph of $D'$; then the Tait digraph of $K \# K'$ is the graph $G \cup_{v \sim v'} G'$ formed by identifying a vertex $v \in V(G)$ with a vertex $v' \in V(G')$.
		
		Choose $v = v'$ as a root vertex for $G \cup_{v \sim v'} G'$ and let $q$ be a basepoint in the unbounded face of $G \cup_{v \sim v'} G'$, which is necessarily adjacent to $v$ in $\R^2$. Let $\mathcal{T}_{G \cup_{v \sim v'} G'}$, $\mathcal{T}_G$ and $\mathcal{T}_{G'}$ denote the spanning tree sets of each graph. Then there is a bijection
		$$
			\mathcal{T}_{G \cup_{v \sim v'} G'} \to \mathcal{T}_G \times \mathcal{T}_{G'}
		$$
		given by taking the restrictions of a tree $T \in \mathcal{T}_{G \cup_{v \sim v'} G'}$ to $E(G)$ and $E(G')$. For any tree $T \in \mathcal{T}_{G \cup_{v \sim v'} G'}$, one can check that the weight of $T$ in $P_{D \# D'}$ is the product of the weights of pair of trees in the image of $T$ under the above bijection. The result follows.
	\end{proof}
	
	Finally, for any alternating link $K$, let $m(K)$ be the mirror of $K$---that is, $K$ with all crossings switched---and let $-K$ be $K$ with the orientations of all components reversed. We consider how $P_K$ relates to $P_{m(K)}$ and $P_{-K}$.
	
	\begin{prop}
		\label{prop:companions}
		With notation as above,
		\begin{enumerate}[label=(\roman*)]
			\item $P_{-K}(x,y,z,w) = P_K(y,x,w,z)$.
			\item $P_{m(K)}(x,y,z,w) = P_K(z^{-1},w^{-1},x^{-1},y^{-1})$.
		\end{enumerate}
	\end{prop}
	
	\begin{proof}
		First, reversing the orientation of $K$ switches the direction of every edge in its Tait digraph and preserves the signs of crossings. Thus, for a fixed root and basepoint, any semi-active edges with respect to a given spanning tree become semi-passive and vice versa. The identity is then clear from Definition \ref{def:poly}.
		
		For the second identity we fix an alternating diagram $D$ of $K$, with Tait digraph $G = (V,E)$, and choose a root $r$ and adjacent basepoint $q$. Let $m(D)$ be the mirror of $D$, and let $G' = (V', E')$ be its Tait digraph. Additionally, let $E_+, E_- \subset E$ denote the positive and negative edges of $E$, and define $E'_+$ and $E'_-$ analogously for $E'$. Since switching a crossing of $D$ switches its type with regard to a checkerboard shading, the graph $G'$ coincides with the planar dual $G^*$ of $G$ as an {\em undirected} graph. We thus have an identification of edges $E \leftrightarrow E'$, and we may fix $q$ as a root vertex for $G'$ and $r$ as a basepoint. Since mirroring a crossing changes its sign, the bijection $E \leftrightarrow E'$ restricts to bijections $E_+ \leftrightarrow E'_-$ and $E_- \leftrightarrow E'_+$.
		
		We now consider edge orientations. For an edge $e \in E$ with corresponding edge $e' \in E'$, we write $e' = e^*$ if the direction of $e'$ agrees with the direction of the edge $e^*$ which is dual to $e$ in $G^*$---otherwise, we write $e' = -e^*$. It is straightforward to check that if $e$ passes through a positive crossing of $D$ then $e' = -e^*$, while if $e$ passes through a negative crossing of $D$ then $e' = e^*$. It follows that $G'$ is equal to the dual digraph $G^*$ of $G$, with the directions of all negative edges of $G'$ (dual to positive edges of $G$) reversed.
		
		Fix a spanning tree $T$ of $G$ with dual tree $T^*$ of $G^*$, and let $T'$ be the tree of $G'$ with the same edges as $T^*$. For an edge $e \in T \cap E_+$, let $e^* \in E(G^*)$ be its dual and let $e' \in E'_- \setminus T'$ be the corresponding edge of $G'$. Then by Lemma \ref{lem:int_ext} (ii), $e$ is internally semi-active in $T$ if and only if $e^* \in E(G^*)$ is externally semi-passive with respect to $T^*$, which occurs if and only if $e'$ is externally semi-active since $e' = -e^*$. Likewise, we can use Lemma \ref{lem:int_ext} (i) to conclude that an edge $e \in E_- \setminus T$ is externally semi-active if and only if the relevant edge $e' \in T' \cap E_+$ is internally semi-active. The result then follows from Definition \ref{def:poly}.
	\end{proof}
	
	Proposition \ref{prop:companions} yields:
	
	\begin{cor}
		Let $K$ be an alternating link. If $K$ is invertible, i.e.~if $K = -K$, then
			$$
		P_K(x,y,z,w) = P_K(y,x,w,z).
		$$
		If $K$ is amphichiral, i.e.~if $K = m(K)$, then
		$$
		P_K(x,y,z,w) = P_K(z^{-1},w^{-1},x^{-1},y^{-1}).
		$$
	\end{cor}
	
	\subsection{Dimer and determinant perspectives}
	\label{sec:dimer_det}
	
	Like the Alexander polynomial, our polynomial $P_K$ can be expressed as a weighted count of dimers. Let $G = (V, E_+ \sqcup E_-)$ be the Tait digraph of an alternating link diagram $D$, and $\mathcal{H} = (\mathcal{V}_1 \sqcup \mathcal{V}_2, \mathcal{E})$ the double overlay of $G$ as in Section \ref{sec:sym_alex}. Then we define a weight function
	$$
	\omega_P : \mathcal{E} \to \Z[x^{-1}, y^{-1}, z, w]
	$$
	by 
	$$
	\omega_P(f) = \begin{cases}
		x^{-1} & \text{$f$ lies at the head of a positive edge of $G$} \\
		y^{-1} & \text{$f$ lies at the tail of a positive edge of $G$} \\
		z & \text{$f$ lies to the right of a negative edge of $G$ (as part of an edge of $G^*$)} \\
		w & \text{$f$ lies to the left of a negative edge of $G$} \\
		1 & \text{otherwise}.
	\end{cases}
	$$
	
	Applying the spanning tree/dimer bijection as in Section \ref{sec:tree_dim_one}, we find:
	\begin{prop}
		\label{prop:p_dim}
		Let $G$ and $\mathcal{H}$ be as above, and $\mathcal{H}^{red}$ a reduced double overlay of $G$ as in Section \ref{sec:sym_alex}. Let $\mathcal{M}$ be the set of dimers of $\mathcal{H}^{red}$. Then
		$$
		P_G(x,y,z,w) = \sum_{M \in \mathcal{M}} \prod_{f \in M} \omega_P(f).
		$$
	\end{prop}
	The weights of edges around positive and negative crossings of a diagram $D$ are shown in Figure \ref{fig:dimer_weights}---of course, the weight function $\omega_P$ is a refinement of the weight function $\omega$ of Section \ref{sec:sym_alex}.
	
	\begin{figure}
		\centering
		
		\begin{subfigure}[t]{0.4\textwidth}
			\centering
			\begin{tikzpicture}
				\draw[-{Stealth}, ultra thick, red!30] (-1,-1) -- (1,1);
				\draw[ultra thick, red!30] (1,-1) -- (0.15,-0.15);
				\draw[{Stealth}-, ultra thick, red!30] (-1,1) -- (-0.15,0.15);
				
				\draw[thick] (0,0) -- node[above] {$1$} (1,0);
				\draw[thick] (0,0) -- node[above] {$1$} (-1,0);
				\draw[thick] (0,0) -- (0,1);
				\draw[thick] (0,0) -- (0,-1);
				\draw[fill=black] (-1,0) circle (1.5pt);
				\draw[fill=black] (1,0) circle (1.5pt);
				\draw[fill=black] (0,1) circle (1.5pt);
				\draw[fill=black] (0,-1) circle (1.5pt);
				\draw[fill=white] (0,0) circle (1.5pt);
				\node[anchor=west] at (1.1, 0.75) {$x^{-1}$};
				\node[anchor=west] at (1.1, -0.75) {$y^{-1}$};
				\draw[<-, thin] (0.1, 0.75) -- (1.1, 0.75);
				\draw[<-, thin] (0.1, -0.75) -- (1.1, -0.75);
			\end{tikzpicture}
			\caption{Positive crossings}
		\end{subfigure}
		\begin{subfigure}[t]{0.4\textwidth}
			\centering
			\begin{tikzpicture}
				\draw[-{Stealth}, ultra thick, red!30] (1,-1) -- (-1,1);
				\draw[ultra thick, red!30] (-1,-1) -- (-0.15,-0.15);
				\draw[{Stealth}-, ultra thick, red!30] (1,1) -- (0.15,0.15);
				
				\draw[thick] (0,0) -- node[above] {$1$} (1,0);
				\draw[thick] (0,0) -- node[above] {$1$} (-1,0);
				\draw[thick] (0,0) -- (0,1);
				\draw[thick] (0,0) -- (0,-1);
				\draw[fill=black] (-1,0) circle (1.5pt);
				\draw[fill=black] (1,0) circle (1.5pt);
				\draw[fill=black] (0,1) circle (1.5pt);
				\draw[fill=black] (0,-1) circle (1.5pt);
				\draw[fill=white] (0,0) circle (1.5pt);
				\node[anchor=west] at (1.1, 0.75) {$z$};
				\node[anchor=west] at (1.1, -0.75) {$w$};
				\draw[<-, thin] (0.1, 0.75) -- (1.1, 0.75);
				\draw[<-, thin] (0.1, -0.75) -- (1.1, -0.75);
			\end{tikzpicture}
			\caption{Negative crossings}
		\end{subfigure}
		
		\caption{Computing $P$ using dimers}
		\label{fig:dimer_weights}
	\end{figure}
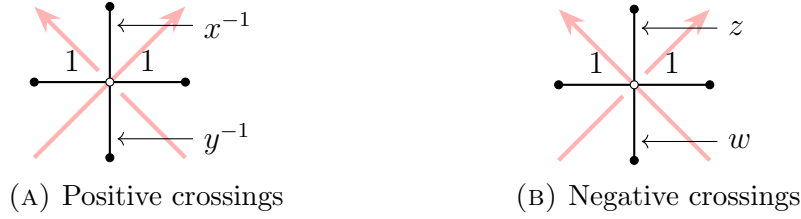
	
	Recall from Section \ref{sec:sym_alex} that the set $\mathcal{V}_2$ may be identified either with crossings of $D$ or with edges of $G$, and that the set $\mathcal{V}_1$ can be thought of either as regions of $D$ or as vertices and faces of $G$. Suppose $D$ has $k$ crossings---then $|\mathcal{V}_2| = k$, and an Euler characteristic argument shows $|\mathcal{V}_1| = k + 2$. Fix an ordering $\{c_1, \dots, c_k\}$ of the vertices in $\mathcal{V}_2$, let $\{v_1, \dots, v_{k + 2}\}$ be the vertices in $\mathcal{V}_1$, and define a $k$-by-$(k + 2)$ matrix
	$$
	\tilde{A} = \{\tilde{a}_{ij}\} \in M_{k, k + 2}(\Z[x^{-1}, y^{-1}, z, w])
	$$
	by
	\begin{equation}
		\label{eq:a_tilde}
		\tilde{a}_{ij} = \sum_\text{$f \in \mathcal{E}$ joining $c_i$ and $v_j$} \omega_P(f).
	\end{equation}
	Then the matrix $\tilde{A}$ encodes the weighted adjacency information of $\mathcal{H}$. Additionally, if we fix a pair of adjacent vertices $r, q \in \mathcal{V}_1$ (meaning the relevant regions of $D$ are adjacent), then we can represent $\mathcal{H}_{(r,q)}^{red}$ by we deleting the columns of $\tilde{A}$ corresponding to $r$ and $q$. We call the resulting matrix $\tilde{A}^{red} = \tilde{A}^{red}_{(r,q)}$. It is not difficult to show, using Proposition \ref{prop:p_dim}, that
	\begin{equation}
		\label{eq:perm}
		P_{(G,r,q)} = \text{perm}(\tilde{A}^{red}_{(r,q)}),
	\end{equation}
	 where perm is the {\em permanent} or unsigned determinant.
	
	Because $\mathcal{H}$ is planar, it is possible to improve on (\ref{eq:perm}) by expressing $P_{(G,r,q)}$ as a genuine determinant. Specifically, we can modify the matrix $\tilde{A}^{red}_{(r,q)}$ to obtain a matrix $A_{(r,q)}^{red}$ such that
	\begin{equation}
		\label{eq:p_det}
		\det(A_{(r,q)}^{red}) = \pm \text{perm}(\tilde{A}^{red}_{(r,q)}) = \pm P_{(G,r,q)}.
	\end{equation}
	The technique for building $A^{red}$ is due to Kasteleyn \cite{kas67}, and we refer the reader to \cite{codaru14} for more details that fit our context. In our case, Kasteleyn's theorem says there exists a sign function
	$$
	\eta : \mathcal{E} \to \{-1,1\}
	$$
	that makes the following hold. We define a $k$-by-$(k + 2)$ matrix $A =  \{a_{ij}\}$, with rows indexed by $\mathcal{V}_2$ and columns indexed by $\mathcal{V}_1$ as above, by the formula
	\begin{equation}
		\label{eq:k_mat}
		a_{ij} = \sum_\text{$f \in \mathcal{E}$ joining $c_i$ and $v_j$} \eta(f)\omega_P(f)
	\end{equation}
	We define $A_{(r,q)}^{red}$ by removing the $r$ and $q$ columns of $A$, and by our choice of $\eta$ $A_{(r,q)}^{red}$ satisfies (\ref{eq:p_det}).
	
	Building on Kasteleyn's work, Kauffman \cite{kauf83} identified a sign function $\eta$ which makes (\ref{eq:p_det}) true for the double overlay of any link diagram, with any choice of weights and adjacent $r$ and $q$. Kauffman's sign function can be defined locally around each crossing, and is shown in Figure \ref{fig:kauf_weights}.\footnote{It's no coincidence that the signs in Figure \ref{fig:kauf_weights} match the signs of the weights in Figure \ref{fig:edge_weights} (and the signs of the weights in Figure \ref{fig:alt_weights} if we reverse the orientation of the link). The Alexander polynomial was originally defined as a determinant, and Kauffman applied to Kasteleyn's theorem to rewrite it as a dimer count.} To summarize, we have the following definition and proposition.
	
		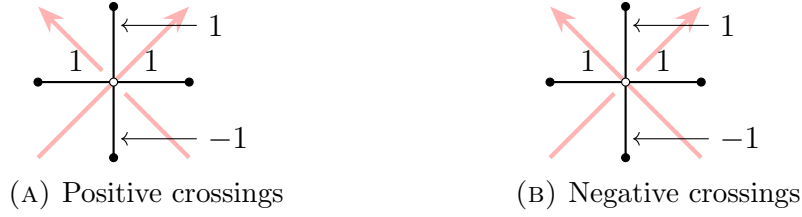
\begin{figure}
		\centering
		
		\begin{subfigure}[t]{0.4\textwidth}
			\centering
			\begin{tikzpicture}
				\draw[-{Stealth}, ultra thick, red!30] (-1,-1) -- (1,1);
				\draw[ultra thick, red!30] (1,-1) -- (0.15,-0.15);
				\draw[{Stealth}-, ultra thick, red!30] (-1,1) -- (-0.15,0.15);
				
				\draw[thick] (0,0) -- node[above] {$1$} (1,0);
				\draw[thick] (0,0) -- node[above] {$1$} (-1,0);
				\draw[thick] (0,0) -- (0,1);
				\draw[thick] (0,0) -- (0,-1);
				\draw[fill=black] (-1,0) circle (1.5pt);
				\draw[fill=black] (1,0) circle (1.5pt);
				\draw[fill=black] (0,1) circle (1.5pt);
				\draw[fill=black] (0,-1) circle (1.5pt);
				\draw[fill=white] (0,0) circle (1.5pt);
				\node[anchor=west] at (1.1, 0.75) {$1$};
				\node[anchor=west] at (1.1, -0.75) {$-1$};
				\draw[<-, thin] (0.1, 0.75) -- (1.1, 0.75);
				\draw[<-, thin] (0.1, -0.75) -- (1.1, -0.75);
			\end{tikzpicture}
			\caption{Positive crossings}
		\end{subfigure}
		\begin{subfigure}[t]{0.4\textwidth}
			\centering
			\begin{tikzpicture}
				\draw[-{Stealth}, ultra thick, red!30] (1,-1) -- (-1,1);
				\draw[ultra thick, red!30] (-1,-1) -- (-0.15,-0.15);
				\draw[{Stealth}-, ultra thick, red!30] (1,1) -- (0.15,0.15);
				
				\draw[thick] (0,0) -- node[above] {$1$} (1,0);
				\draw[thick] (0,0) -- node[above] {$1$} (-1,0);
				\draw[thick] (0,0) -- (0,1);
				\draw[thick] (0,0) -- (0,-1);
				\draw[fill=black] (-1,0) circle (1.5pt);
				\draw[fill=black] (1,0) circle (1.5pt);
				\draw[fill=black] (0,1) circle (1.5pt);
				\draw[fill=black] (0,-1) circle (1.5pt);
				\draw[fill=white] (0,0) circle (1.5pt);
				\node[anchor=west] at (1.1, 0.75) {$1$};
				\node[anchor=west] at (1.1, -0.75) {$-1$};
				\draw[<-, thin] (0.1, 0.75) -- (1.1, 0.75);
				\draw[<-, thin] (0.1, -0.75) -- (1.1, -0.75);
			\end{tikzpicture}
			\caption{Negative crossings}
		\end{subfigure}
		
		\caption{Kauffman's sign function}
		\label{fig:kauf_weights}
	\end{figure}
	
	\begin{defn}
		\label{def:kast}
		Let $D$ be an alternating link diagram with double overlay $\mathcal{H}$ and adjacent vertices $r, q \in \mathcal{V}_1$. Let $\eta$ be Kauffman's sign function on the edges of $\mathcal{H}$, as shown in Figure \ref{fig:kauf_weights}, and let $A$ be the matrix defined by (\ref{eq:k_mat}). Let $A^{red} = A^{red}_{(p, q)}$ be the matrix $A$ with the $r$ and $q$ columns removed. Then we call $A$ an {\em (unreduced) Kasteleyn matrix} for $D$, and $A^{red}$ a {\em reduced Kasteleyn matrix}.
	\end{defn}
	
	\begin{prop}
		\label{prop:p_det}
		Let $D$ be an alternating link diagram with Tait digraph $G$, and let $r$ be a vertex of $G$ and $q$ an adjacent basepoint. Let $A^{red}_{(r,q)}$ be the corresponding reduced Kasteleyn matrix---then
		$$
		P_{(G,r,q)}(x,y,z,w) = |\det(A^{red}_{(r,q)})|.
		$$
	\end{prop}
	
	Finally, we emphasize:
	
	\begin{lemma}
		\label{lem:a_signs}
		With notation as in Definition \ref{def:kast}, the entries of $A$ coincide exactly with the entries of $\tilde{A}$ except that all occurences of $y^{-1}$ and $w$ are negative.
	\end{lemma}
	
	\begin{proof}
		This is clear from the definitions.
	\end{proof}
	
	\subsection{Examples}
	\label{sec:examples}
	
	Figure \ref{fig:eight_diag} shows an alternating diagram of the figure eight knot $4_1$. Figure \ref{fig:eight_graphs} shows five copies of its Tait digraph $G$ which enumerate its five spanning trees, with positive and negative edges labelled. The solid edges in each graph are included in the tree, while the dashed ones are left out---by Definition \ref{def:poly}, the edges which contribute to a tree's weight are solid positive edges and dashed negative ones. Additionally, the red circle marks the root vertex, the red asterisk marks the adjacent basepoint, and the weight of each tree is shown below it. From the figure, we see that
	$$
	P_{4_1} = 1 + x^{-1}z + y^{-1}z + x^{-1}w + y^{-1}w.
	$$
	The plane digraph $G$ also has exactly one vertex-face pair which is not adjacent in $\R^2$. If we compute the total weights of the spanning trees with this invalid pair, we get a different answer: $1 + 2x^{-1}z + 2y^{-1}w$. This shows the adjacency hypothesis of Definition \ref{def:poly} is necessary for a well-defined invariant.
	
	\begin{figure}
		\centering
		
		\begin{subfigure}{0.3\textwidth}
			\centering
			\begin{tikzpicture}
				\useasboundingbox (-10.5,-4) rectangle (-5.5,4);
				\path[spath/save=fig8]
					(-8,-1.5) .. controls +(0:1) and +(-90:0.5) ..
					(-6.5,0) .. controls +(90:0.5) and +(0:0.5) ..
					(-7.25,0.75) .. controls +(180:0.5) and +(90:0.5) ..
					(-8.75,-0.25) .. controls +(-90:0.5) and +(90:0.5) .. 
					(-7.25,-1.5) .. controls +(-90:0.5) and +(0:0.5) ..
					(-8,-2.25) .. controls +(180:0.5) and +(-90:0.5) ..
					(-8.75,-1.5) .. controls +(90:0.5) and +(-90:0.5) .. 
					(-7.25,-0.25) .. controls +(90:0.5) and +(0:0.5) ..
					(-8.75,0.75) .. controls +(180:0.5) and +(90:0.5) .. 
					(-9.5,0) .. controls +(-90:0.5) and +(180:1) .. (-8,-1.5);			
				
				\tikzset{
					every fig8 component/.style={draw, ultra thick, red},
					spath/knot={fig8}{15pt}{2,4,6,8},
				}
				
				\draw[-{Stealth}, line width=1.5pt, red] (-6.5,0) -- (-6.5,-0.1);
				\draw[{Stealth}-, line width=1.5pt, red] (-9.5,0.1) -- (-9.5,-0);
				
			\end{tikzpicture}
			\subcaption{The figure eight knot}
			\label{fig:eight_diag}
		\end{subfigure}
		\begin{subfigure}{0.65\textwidth}
			\centering
			\begin{tikzpicture}
				\useasboundingbox (-6,-5) rectangle (6,3);
							
				\coordinate (m1) at (0,0);
				\coordinate (l1) at (-1,2);
				\coordinate (r1) at (1,2);
				\coordinate (a1) at (-1,0);
				
				\draw[fill=black] (m1) circle (1.5pt) node[below, yshift=-2mm] {$x^{-1}z$};
				\draw[fill=black] (l1) circle (1.5pt);
				\draw[fill=black] (r1) circle (1.5pt);
				\draw[thick, red] (m1) circle (5pt);
				\node[thick, red] at (a1) {$*$};
				
				\draw[thick, postaction={decorate}, decoration={
					markings,
					mark=at position 0.05 with {\arrowreversed{Triangle}}
				}] (m1) to[out=135, in=-90] node[left] {$-$} (l1);
				
				\draw[thick, dashed, postaction={decorate}, decoration={
					markings,
					mark=at position 0.05 with {\arrowreversed{Triangle}}
				}] (m1) to[out=45, in=-90] node[right] {$-$} (r1);
				
				\draw[thick, postaction={decorate}, decoration={
					markings,
					mark=at position 0.95 with {\arrow{Triangle}}
				}] (l1) to[out=-60,in=-120] node[below] {$+$} (r1);
				
				\draw[thick, dashed, postaction={decorate}, decoration={
					markings,
					mark=at position 0.95 with {\arrow{Triangle}}
				}] (r1) to[out=120, in=60] node[below] {$+$} (l1);
				
				\coordinate (m2) at (4,0);
				\coordinate (l2) at (3,2);
				\coordinate (r2) at (5,2);
				\coordinate (a2) at (3,0);
				
				\draw[fill=black] (m2) circle (1.5pt) node[below, yshift=-2mm] {$x^{-1}w$};
				\draw[fill=black] (l2) circle (1.5pt);
				\draw[fill=black] (r2) circle (1.5pt);
				\draw[thick, red] (m2) circle (5pt);
				\node[thick, red] at (a2) {$*$};
				
				\draw[thick, dashed, postaction={decorate}, decoration={
					markings,
					mark=at position 0.05 with {\arrowreversed{Triangle}}
				}] (m2) to[out=135, in=-90] node[left] {$-$} (l2);
				
				\draw[thick, postaction={decorate}, decoration={
					markings,
					mark=at position 0.05 with {\arrowreversed{Triangle}}
				}] (m2) to[out=45, in=-90] node[right] {$-$} (r2);
				
				\draw[thick, dashed, postaction={decorate}, decoration={
					markings,
					mark=at position 0.95 with {\arrow{Triangle}}
				}] (l2) to[out=-60,in=-120] node[below] {$+$} (r2);
				
				\draw[thick, postaction={decorate}, decoration={
					markings,
					mark=at position 0.95 with {\arrow{Triangle}}
				}] (r2) to[out=120, in=60] node[below] {$+$} (l2);
				
				\coordinate (m3) at (-4,0);
				\coordinate (l3) at (-5,2);
				\coordinate (r3) at (-3,2);
				\coordinate (a3) at (-5,0);
				
				\draw[fill=black] (m3) circle (1.5pt) node[below, yshift=-2mm] {$y^{-1}z$};
				\draw[fill=black] (l3) circle (1.5pt);
				\draw[fill=black] (r3) circle (1.5pt);
				\draw[thick, red] (m3) circle (5pt);
				\node[thick, red] at (a3) {$*$};
				
				\draw[thick, postaction={decorate}, decoration={
					markings,
					mark=at position 0.05 with {\arrowreversed{Triangle}}
				}] (m3) to[out=135, in=-90] node[left] {$-$} (l3);
				
				\draw[thick, dashed, postaction={decorate}, decoration={
					markings,
					mark=at position 0.05 with {\arrowreversed{Triangle}}
				}] (m3) to[out=45, in=-90] node[right] {$-$} (r3);
				
				\draw[thick, dashed, postaction={decorate}, decoration={
					markings,
					mark=at position 0.95 with {\arrow{Triangle}}
				}] (l3) to[out=-60,in=-120] node[below] {$+$} (r3);
				
				\draw[thick, postaction={decorate}, decoration={
					markings,
					mark=at position 0.95 with {\arrow{Triangle}}
				}] (r3) to[out=120, in=60] node[below] {$+$} (l3);

				\coordinate (m4) at (-2,-4);
				\coordinate (l4) at (-3,-2);
				\coordinate (r4) at (-1,-2);
				\coordinate (a4) at (-3,-4);
				
				\draw[fill=black] (m4) circle (1.5pt) node[below, yshift=-2mm] {$y^{-1}w$};
				\draw[fill=black] (l4) circle (1.5pt);
				\draw[fill=black] (r4) circle (1.5pt);
				\draw[thick, red] (m4) circle (5pt);
				\node[thick, red] at (a4) {$*$};
				
				\draw[thick, dashed, postaction={decorate}, decoration={
					markings,
					mark=at position 0.05 with {\arrowreversed{Triangle}}
				}] (m4) to[out=135, in=-90] node[left] {$-$} (l4);
				
				\draw[thick, postaction={decorate}, decoration={
					markings,
					mark=at position 0.05 with {\arrowreversed{Triangle}}
				}] (m4) to[out=45, in=-90] node[right] {$-$} (r4);
				
				\draw[thick, postaction={decorate}, decoration={
					markings,
					mark=at position 0.95 with {\arrow{Triangle}}
				}] (l4) to[out=-60,in=-120] node[below] {$+$} (r4);
				
				\draw[thick, dashed, postaction={decorate}, decoration={
					markings,
					mark=at position 0.95 with {\arrow{Triangle}}
				}] (r4) to[out=120, in=60] node[below] {$+$} (l4);

				\coordinate (m5) at (2,-4);
				\coordinate (l5) at (1,-2);
				\coordinate (r5) at (3,-2);
				\coordinate (a5) at (1,-4);
				
				\draw[fill=black] (m5) circle (1.5pt) node[below, yshift=-2mm] {$1$};
				\draw[fill=black] (l5) circle (1.5pt);
				\draw[fill=black] (r5) circle (1.5pt);
				\draw[thick, red] (m5) circle (5pt);
				\node[thick, red] at (a5) {$*$};
				
				\draw[thick, postaction={decorate}, decoration={
					markings,
					mark=at position 0.05 with {\arrowreversed{Triangle}}
				}] (m5) to[out=135, in=-90] node[left] {$-$} (l5);
				
				\draw[thick, postaction={decorate}, decoration={
					markings,
					mark=at position 0.05 with {\arrowreversed{Triangle}}
				}] (m5) to[out=45, in=-90] node[right] {$-$} (r5);
				
				\draw[thick, dashed, postaction={decorate}, decoration={
					markings,
					mark=at position 0.95 with {\arrow{Triangle}}
				}] (l5) to[out=-60,in=-120] node[below] {$+$} (r5);
				
				\draw[thick, dashed, postaction={decorate}, decoration={
					markings,
					mark=at position 0.95 with {\arrow{Triangle}}
				}] (r5) to[out=120, in=60] node[below] {$+$} (l5);
			\end{tikzpicture}
			\subcaption{Spanning trees and weights for \ref{fig:eight_diag}}
			\label{fig:eight_graphs}
		\end{subfigure}
		
		\caption{Spanning tree contributions for the figure eight digraph}
	\end{figure}
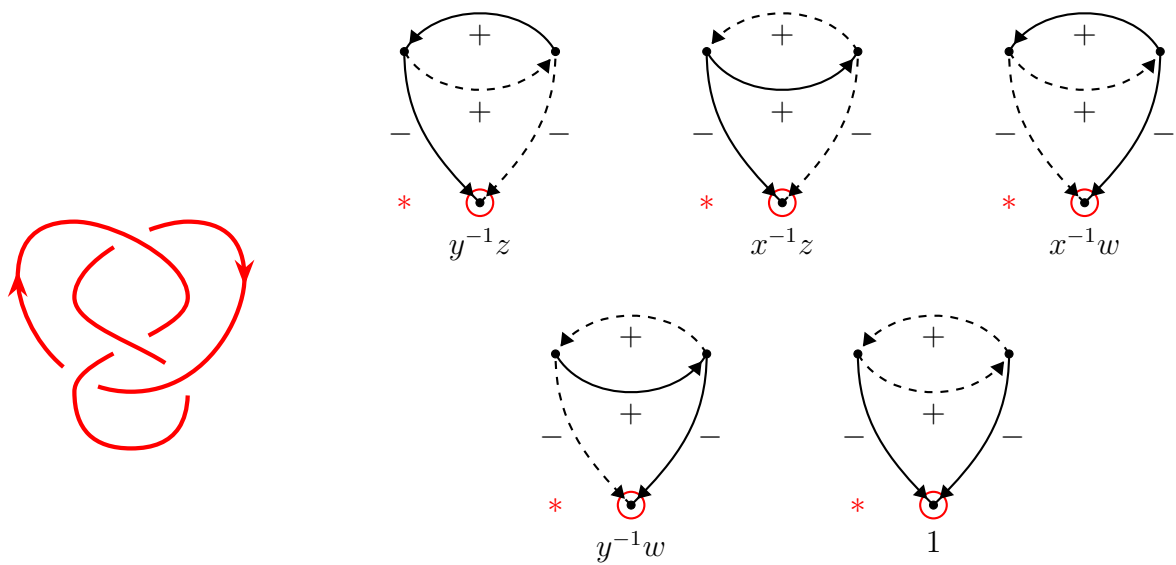
	
	For larger diagrams, enumerating spanning trees by hand is impractical. Figure \ref{fig:810} shows an alternating diagram of the knot $8_{10}$---to compute $P_{8_{10}}$, we use the region and crossing indices shown to write down the reduced Kasteleyn matrix $A$ below.
	
	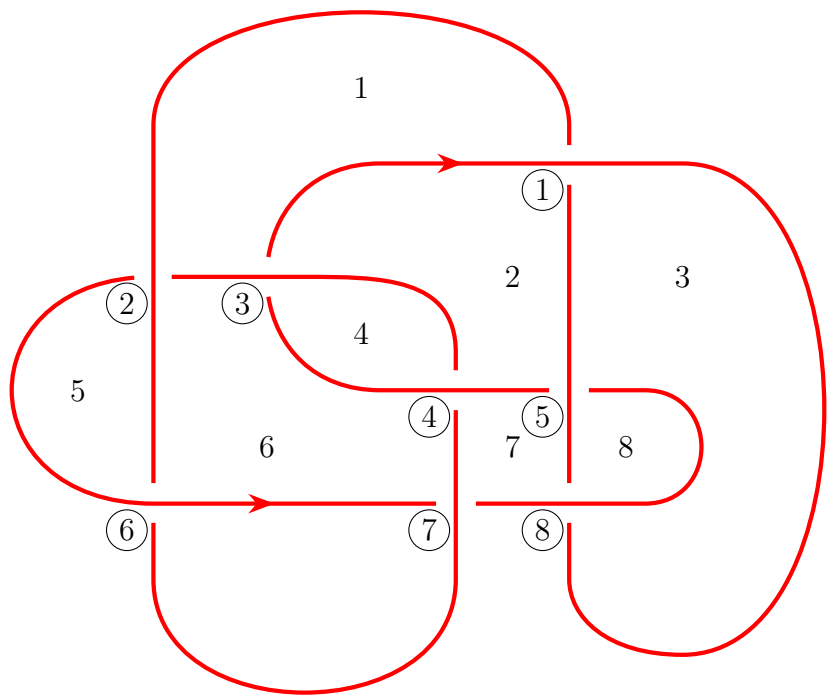
\begin{figure}
		\centering
		\begin{tikzpicture}
			\path[spath/save=fig810]
				(0,0) .. controls +(0:1) and +(180:1) .. 
				(3.5,0) .. controls +(0:1) and +(0:1) .. 
				(3.5,1.5) .. controls +(180:1) and +(0:1) .. 
				(0,1.5) .. controls +(180:2) and +(180:2) .. 
				(0,4.5) .. controls +(0:1) and +(180:1) .. 
				(4,4.5) .. controls +(0:2.5) and +(0:2.5) .. 
				(4,-2) .. controls +(180:1) and +(-90:0.5) .. 
				(2.5,-1) .. controls +(90:1) and +(-90:1) .. 
				(2.5,5) .. controls +(90:2) and +(90:2) .. 
				(-3,5) .. controls +(-90:1) and +(90:1) .. 
				(-3,-1) .. controls +(-90:2) and +(-90:2) ..
				(1,-1) .. controls +(90:1) and +(-90:1) ..
				(1,2) .. controls +(90:1) and +(0:1) .. 
				(-1,3) .. controls +(180:1) and +(0:1) .. 
				(-3,3) .. controls +(180:2.5) and +(180:2.5) .. 
				(-3,0) .. controls +(0:1) and (180:1) .. (0,0);		
			
			\tikzset{
				every fig810 component/.style={draw, ultra thick, red},
				spath/knot={fig810}{15pt}{1,3,5,8,10,12,14,16}
			}
			
			\node at (-1.5,0.75) {$6$};
			\node at (1.75,0.75) {$7$};
			\node at (3.25,0.75) {$8$};
			\node at (4,3) {$3$};
			\node at (1.75,3) {$2$};
			\node at (-0.25,2.25) {$4$};
			\node at (-4,1.5) {$5$};
			\node at (-0.25,5.5) {$1$};
			
			\node[shape=circle,draw,inner sep=2pt, xshift=-10pt, yshift=-10pt] at (1,0) {$7$};
			\node[shape=circle,draw,inner sep=2pt, xshift=-10pt, yshift=-10pt] at (2.5,0) {$8$};
			\node[shape=circle,draw,inner sep=2pt, xshift=-10pt, yshift=-10pt] at (1,1.5) {$4$};
			\node[shape=circle,draw,inner sep=2pt, xshift=-10pt, yshift=-10pt] at (2.5,1.5) {$5$};
			\node[shape=circle,draw,inner sep=2pt, xshift=-10pt, yshift=-10pt] at (2.5,4.5) {$1$};
			\node[shape=circle,draw,inner sep=2pt, xshift=-10pt, yshift=-10pt] at (-3,3) {$2$};
			\node[shape=circle,draw,inner sep=2pt, xshift=-10pt, yshift=-10pt] at (-3,0) {$6$};
			\node[shape=circle,draw,inner sep=2pt, xshift=-10pt, yshift=-10pt] at (-1.47,3) {$3$};
			
			\draw[-{Stealth}, line width=1.5pt, red] (-1.5,0) -- (-1.4,0);
			\draw[-{Stealth}, line width=1.5pt, red] (1,4.5) -- (1.1,4.5);
			
		\end{tikzpicture}
		\caption{The knot $8_{10}$}
		\label{fig:810}
	\end{figure}
	
	$$
	A = \begin{bmatrix}
		1 & -y^{-1} & 1 & 0  & 0 & 0 & 0 & 0 \\
		-w & 0 & 0 & 0 & z & 1 & 0 & 0 \\
		z & 1 & 0 & -w & 0 & 1 & 0 & 0 \\
		0 & 1 & 0 & z & 0 & 1 & -w & 0 \\
		0 & x^{-1} & 1 & 0 & 0 & 0 & 1 & -y^{-1} \\
		0 & 0 & 0 & 0 & -w & 1 & 0 & 0 \\
		0 & 0 & 0 & 0 & 0 & 1 & z & 0 \\
		0 & 0 & 1 & 0 & 0 & 0 & 1 & x^{-1}
	\end{bmatrix}.
	$$
	
	We find that
	\begin{align*}
	P_{8_{10}} &= \det(A) \\
	&= (x^{-2} + x^{-1} y^{-1} + y^{-2})(w^4+w^3z+w^2z^2+wz^3+z^4) \\
	&\ \ \ \ \ + (x^{-1} + y^{-1})(w+z)(w^2+wz+z^2),
	\end{align*}
	as stated in the introduction.
	
	\section{Root and Basepoint Independence}
	\label{sec:root_indep}
	
	\subsection{Properties of Tait digraphs}
	
		In this section, we prove the polynomial $P_K$ of Definition \ref{def:poly} does not depend on the choice of adjacent root and basepoint. First, we gather some combinatorial facts about Tait digraphs.
		
		Fix an alternating link diagram $D$ with Tait digraph $G$. We recall that {\em resolving a crossing of $D$ (according to its orientation)} is the operation shown in Figure \ref{fig:resolve}, and that resolving a crossing in an alternating link diagram produces another alternating diagram.
	
	\begin{figure}
		\centering
	
		\begin{tikzpicture}
			\draw[-{Stealth}, ultra thick, red] (-6,-1) -- (-4,1);
			\draw[ultra thick, red] (-4,-1) -- (-4.85,-0.15);
			\draw[{Stealth}-, ultra thick, red] (-6,1) -- (-5.15,0.15);
			
			\draw[->, thick] (-3.25,0) -- (-1.75,0);
			
			\draw[-{Stealth}, ultra thick, red] (-1,-1) to[out=45,in=-45] (-1,1);
			\draw[-{Stealth}, ultra thick, red] (1,-1) to[out=135,in=-135] (1,1);
			
			\draw[->, thick] (3.25,0) -- (1.75,0);
	
			\draw[-{Stealth}, ultra thick, red] (6,-1) -- (4,1);
			\draw[ultra thick, red] (4,-1) -- (4.85,-0.15);
			\draw[{Stealth}-, ultra thick, red] (6,1) -- (5.15,0.15);
		\end{tikzpicture}
		
		\caption{Resolving crossings}
		\label{fig:resolve}
	\end{figure}
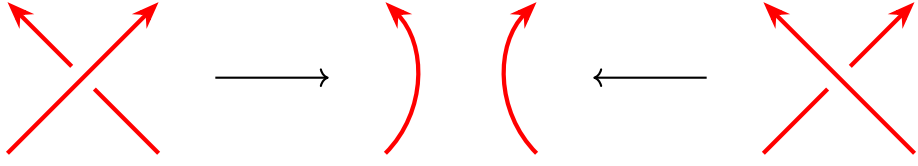
	
	\begin{lemma}
		\label{lem:resolve}
		Fix a crossing $c$ of $D$, and let $D'$ be the alternating link diagram obtained by resolving $c$. Let $e$ be the edge of $G$ passing through $c$, and let $G'$ be the Tait digraph of $D'$. Then:
		\begin{enumerate}[label=(\roman*)]
			\item If $c$ is a positive crossing of $D$, then $G' = G / e$.
			\item If $c$ is a negative crossing of $D$, then $G' = G \setminus e$.
		\end{enumerate}
	\end{lemma}
	
	We leave the proof to the reader.
	
	\begin{cor}
		\label{cor:closed}
		The set of Tait digraphs of alternating link diagrams (including split diagrams) is closed under the operations of contracting positive edges and deleting negative edges.
	\end{cor}
	
	For the next lemma we recall that a vertex in a digraph is a {\em source} if all its edges are outgoing, and a {\em sink} if all its edges are incoming. Additionally, a vertex in a plane digraph is {\em alternating} if its edges alternate between incoming and outgoing when circling it in the plane.
	
	\begin{lemma}
		\label{lem:spec_alt}
		Let $D$ be a (non-split) alternating link diagram with Tait digraph $G$. Then:
		\begin{enumerate}[label=(\roman*)]
			\item $D$ has only positive crossings if and only if every vertex of $G$ is alternating.
			\item $D$ has only negative crossings if and only if every vertex of $G$ is either a source or a sink.
		\end{enumerate}
	\end{lemma}
	
	Lemma \ref{lem:spec_alt} is implicit in \cite{must03, kmp25}, though the Tait digraph is not defined there.
	
	\begin{proof}
		For (i), suppose $D$ has only positive crossings and fix a vertex $v$ of $G$. Then $v$ sits in a shaded region of $D$ bordered by positive type I crossings, and it is not difficult to check that if a given edge adjacent to $v$ is outgoing, then the next edge encountered in a clockwise or counter-clockwise orbit of $v$ must be incoming, and vice versa. An example is shown in Figure \ref{fig:alt_region}.
		
		Conversely, suppose $v$ is an alternating vertex of $G$, and let $e$ and $e'$ be two edges adjacent to $v$ so that $e$ is encountered immediately after $e'$ in a counter-clockwise orbit around $v$. Then it follows from the alternating hypothesis that $e$ is positive, and moving around $v$ we conclude that all edges bordering $v$ are positive. The second statement of the lemma is proven similarly.
	\end{proof}
	
	\begin{figure}
		\centering
		
		\begin{tikzpicture}
			\fill[gray!20] (-1,0) -- (-2,0) -- (-2,1) -- (-1,1) -- cycle;
			\fill[gray!20] (1,0) -- (2,0) -- (2,1) -- (1,1) -- cycle;
			\fill[gray!20] (-1,0) -- (1,0) -- (1,-1) -- (-1,-1) -- cycle;
							
			\draw[-{Stealth}, ultra thick, red] (-1,-1) -- (-1,1);
			\draw[{Stealth}-, ultra thick, red] (-2, 0) -- (-1.15,0);
			\draw[ultra thick, red] (-0.85,0) -- (2,0);
			\draw[ultra thick, red] (1,-1) -- (1,-0.15);
			\draw[-{Stealth}, ultra thick, red] (1,0.15) -- (1,1);
			
			\draw[-Triangle, thick] (0,-1) -- (-1.75,0.75);
			\draw[Triangle-, thick] (0.05,-0.95) -- (1.75,0.75);
			\draw[fill=black] (0,-1) circle (1.5pt);
			
		\end{tikzpicture}
		
		\caption{Adjacent positive crossings}
		\label{fig:alt_region}
	\end{figure}
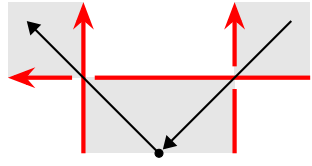
	
	From Lemma \ref{lem:spec_alt}, we obtain:
	
	\begin{lemma}
		\label{lem:cut_cyc}
		As usual, write $E_+$ and $E_-$ for the positive and negative edges of a Tait digraph $G = (V,E)$. Then:
		\begin{enumerate}[label=(\roman*)]
			\item $E_+$ can be written as a disjoint union of directed cycles of $G$, and
			\item $E_-$ can be written as a disjoint union of directed cuts.
		\end{enumerate}
	\end{lemma}
	
	For the proof, recall that a {\em cut vertex} of a connected graph is a vertex whose removal disconnects the graph.
	
	\begin{proof}
		Let $G_- = (V, E_-)$ be the graph obtained from $G$ by contracting every edge in $E_+$; then $G_-$ is a Tait digraph by Corollary \ref{cor:closed}, and we denote its link diagram by $D_-$. Lemma \ref{lem:spec_alt} (ii) may not hold for $G_-$ if the diagram $D_-$ is split, since then a region of $D_-$ may have more than one boundary component. In this case, however, we write $D_-$ as a disjoint union of non-split diagrams
		$$
		D_- = D_-^1 \sqcup D_-^2 \sqcup \cdots \sqcup D_-^m.
		$$
		Let $G_-^i$ be the Tait digraph of $D_-^i$. Then $G_-$ is equal to the union of the $G_-^i$,
		$$
		G_- = G_-^1 \cup G_-^2 \cup \cdots \cup G_-^m,
		$$
		where $G_-^i \cap G_-^j$ is either empty or a cut vertex for any $i, j \in \{1, \dots, m\}$ with $i \neq j$.
		
		Fix $i \in \{1, \dots, m\}$; then by Lemma \ref{lem:spec_alt} (ii)  each vertex of $G^i_-$ is either a source or a sink in $G^i_-$. Let $v_1, \dots, v_k$ be the set of sink vertices of $G^i_-$, and for each $j \in \{1, \dots, k\}$ let $B_j$ be the set of edges adjoining $v_j$. Then each $B_j$ is a directed cut of $G^i_-$, and since each edge of $G^i_-$ meets exactly one sink we can write $E(G^i_-)$ as the disjoint union
		$$
		E(G^i_-) = B_1 \sqcup \cdots \sqcup B_k.
		$$
		Joining graphs along cut vertices does not affect which sets are cuts, so it follows that we can write $E(G_-) = E_-$ as a disjoint union of directed cuts. Since $G_-$ was obtained from $G$ by contracting edges, each directed cut of $G_-$ is also a directed cut of $G$, proving (ii).
		
		Dually, let $G_+ = (V, E_+)$ be the graph obtained from $G$ by deleting every edge in $E_-$. By the same argument as above, using Corollary \ref{cor:closed} and Lemma \ref{lem:spec_alt} (i), we can write 
		$$
		G_+ = G_+^1 \cup G_+^2 \cup \cdots \cup G_+^m,
		$$
		where $G_+^i \cap G_+^j$ is either empty or a cut vertex for $i \neq j$, and each vertex of $G_+^i$ is alternating in $G_+^i$ for all $i \in \{1, \dots, m\}$. In particular each vertex of $G_+$ has an equal number of incoming and outgoing edges, which is a necessary and sufficient condition for writing $E_+$ as a disjoint union of directed cycles
		$$
		E_+ = C_1 \sqcup \cdots \sqcup C_k
		$$
		of $G_+$. (This can be proven by induction on $|E_+|$, by following a directed path until it closes and then deleting its edges from the graph.) Since $G_+$ was formed from $G$ by deleting edges, each $C_i$ is also a directed cycle of $G$.
	\end{proof}
	
	Lemma \ref{lem:cut_cyc} has two useful corollaries, both of which are known to experts.
	
	\begin{cor}
		\label{cor:bip_eul}
		Let $G$, $E_+$ and $E_-$ be as above. If $E_+ = \varnothing$, then $G$ is bipartite. If $E_- = \varnothing$, then $G$ is Eulerian.
	\end{cor}
	
	\begin{proof}
		Any graph whose edges are a disjoint union of directed cuts is bipartite, and any connected graph which is a disjoint union of directed cycles is Eulerian.
	\end{proof}
	
	\begin{cor}
		\label{cor:even_cut}
		Any cycle of a Tait digraph $G$ contains an even number of negative edges.
	\end{cor}
	
	\begin{proof}
		Write the set of negative edges of $G$ as a disjoint union of directed cuts,
		$$
		E_- = B_1 \sqcup \cdots \sqcup B_k,
		$$
		and let $C$ be an arbitrary cycle of $G$. Then
		$$
		C \cap E_- = (C \cap B_1) \sqcup \cdots \sqcup (C \cap B_k).
		$$
		$C \cap B_i$ is even for each $i$, so $C \cap E_-$ is as well.
	\end{proof}
	
	\subsection{Independence}
	
	We will prove $P_K$ does not depend on the choice of adjacent root and basepoint using the determinant formulation of Section \ref{sec:dimer_det}. To this end, let $G = (V,E)$ be a Tait digraph and let $F$ be its set of faces. Let $\mathcal{H} = (\mathcal{V}_1 \sqcup \mathcal{V}_2, \mathcal{E})$ be the double overlay of $G$, and recall that there are identifications
	\begin{align*}
		\mathcal{V}_1 &\leftrightarrow V \cup F \\
		\mathcal{V}_2 &\leftrightarrow E.
	\end{align*} 
	Let $A$ be an unreduced Kasteleyn matrix for $\mathcal{H}$, as in Definition \ref{def:kast}. Then each column of $A$ corresponds to an element of $\mathcal{V}_1$ and each row to an element of $\mathcal{V}_2$, so we may consider $A$ as a map
	$$
	A : \Z[x^{-1},y^{-1},z,w]^{V \cup F} \to \Z[x^{-1},y^{-1},z,w]^E.
	$$
	We will keep this perspective for the rest of the section.
	
	Our proof relies on a key lemma.
	
	\begin{lemma}
		\label{lem:key}
		With notation as above, let $r \in V$ be a vertex of $G$ and $q \in F$ a face adjacent to $r$. Let $a$ and $b$ be arbitrary elements of $\Z[x^{-1}, y^{-1}, z, w]$. Then there exists a vector
		$$
		\nu \in \Z[x^{-1}, y^{-1}, z, w]^{V \cup F}
		$$
		such that:
		\begin{enumerate}[label=(\roman*)]
			\item $\nu(r) = a$ and $\nu(q) = b$.
			\item Let $u, v \in V$ be two vertices of $G$ which share a common edge $e \in E$, and let $f$ be a face adjacent to $e$ (and therefore to $u$ and $v$). If $e \in E_+$ then
			$$
			\nu(v) = \nu(u).
			$$
			If $e \in E_-$, then
			$$
			\nu(v) = (w - z)\nu(f) - \nu(u).
			$$
			\item $A\nu = {\bf 0}$.
		\end{enumerate}
	\end{lemma}
	
	We use the notation $\nu(v)$ to indicate the $v$-coordinate of $\nu$ in $\Z[x^{-1}, y^{-1}, z, w]$, for any $v \in V \cup F$. The vector $\nu$ is also unique, but we won't need this fact here.
	
	\begin{proof}
		The proof is by induction on $|E_-|$, the number of negative edges of $G$. If $|E_-| = 0$, then $G$ is Eulerian by Corollary \ref{cor:bip_eul} and its planar dual $G^*$ is bipartite. We can thus partition the faces of $G$ into two sets
		$$
		F = F' \sqcup F'',
		$$
		such that $q \in F'$ and no two faces in the same set abut the same edge of $G$. Let
		$$
		\nu \in \Z[x^{-1}, y^{-1}, z, w]^{V \cup F}
		$$
		be the vector
		$$
		\nu(p) = \begin{cases}
			a & p \in V \\
			b & p \in F' \\
			(y^{-1} - x^{-1})a - b & p \in F''
		\end{cases}.
		$$
		Clearly $\nu$ satisfies properties (i) and (ii) above, and we claim $\nu$ also satisfies property (iii). For this, let $e \in E$ be an arbitrary edge of $G$. Then since $e \in E_+$, one can check using Figure \ref{fig:dimer_weights} and Lemma \ref{lem:a_signs} that
		$$
		\big(A\nu\big)(e) = x^{-1}a - y^{-1}a + b + (y^{-1} - x^{-1})a -b = 0,
		$$
		where $\big(A\nu\big)(e)$ indicates the $e$-coordinate of $A\nu$. This completes the base case.
		
		Now suppose $|E_-| > 0$. Then $E_-$ can be written as a disjoint union of directed cuts by Lemma \ref{lem:cut_cyc}, and we fix one such cut $B \subset E_-$. Let $G'$ be the graph obtained from $G$ by deleting the edges of $B$. Then $G'$ has two connected components, which we label $G_1 = (V_1, E_1)$ and $G_2 = (V_2, E_2)$ so that $r \in V_1$. Each of these is a Tait digraph by Lemma \ref{cor:closed}, and an edge is negative in $E_1$ or $E_2$ if and only if it is negative in $E$. Let $F_1, F_2 \subset F$ be the bounded faces of $G_1$ and $G_2$ respectively, and $f$ the unique annular face of $G'$ which abuts both components. Let $F_B \subset F$ be the faces of $G$ corresponding to $f \setminus (\bigcup B)$, so that
		$$
		F = F_1 \cup F_2 \cup F_B.
		$$
		By the induction hypothesis, there exists a vector
		$$
		\nu_1 \in \Z[x^{-1}, y^{-1}, z, w]^{V_1 \cup F_1 \cup f}
		$$
		satisfying (i)--(iii) for $G_1$. (If $q \in F_B$, in which case $q$ is not a face of $G_1$, then we define $\nu_1$ by requiring that $\nu_1(f) = b$ instead of $\nu_1(q) = b$.)
		
		Let $e$ be an edge in the cut $B$ with endpoints $u_1$ and $u_2$, so that $u_i \in V_i$. Then applying the induction hypothesis again, we construct a vector
		$$
		\nu_2 \in \Z[x^{-1}, y^{-1}, z, w]^{V_2 \cup F_2 \cup f}
		$$
		such that
		\begin{align*}
			\nu_2(u_2) &= (w - z)\nu_1(f) - \nu_1(u_1), \\
			\nu_2(f) &= \nu_1(f),
		\end{align*}
		and $\nu_2$ satisfies properties (ii) and (iii) for $G_2$. Finally, define a vector
		$$
		\nu \in \Z[x^{-1}, y^{-1}, z, w]^{V \cup F}
		$$
		by
		$$
		\nu(p) = \begin{cases}
			\nu_1(p) & p \in V_1 \cup F_1 \\
			\nu_1(f) \big(\text{$=$ } \nu_2(f)\big)& p \in F_B \\
			\nu_2(p) & p \in V_2 \cup F_2
		\end{cases}.
		$$
		We claim $\nu$ satisfies properties (i)---(iii). It's clear that $\nu$ satisfies (i), and that $\nu$ satisfies (ii) for any two vertices which are both contained either in $V_1$ or $V_2$. Fix an edge $e' \in B$, and let $w_1 \in V_1$ and $w_2 \in V_2$ be its endpoints; we must check that (ii) is satisfied for $w_1$ and $w_2$.
		
		The edges $e, e' \in B$ divide the annular region $f$ into two disk components, and we call one of these disks $f'$. Let $C \subset E$ be the cycle of $G$ on the boundary of $f'$, so that 
		$$
		C = \{e,e'\} \cup C_1 \cup C_2,
		$$
		where $C_i = C \cap E_i$ is a path of edges connecting $u_i$ and $w_i$ in $G_i$, for $i = 1, 2$. All of the vertices in $C_1$ are adjacent to $f$ in $G_1$; thus, it follows from property (ii) of $\nu_1$ that if $C_1$ contains only positive edges, then
		$$
		\nu(w_1) = \nu_1(w_1) = \nu_1(u_1)  = \nu(u_1).
		$$
		Indeed, if $C_1$ contains only positive edges, then $\nu$ takes the same value on every vertex in the path. More generally, property (ii) implies that if $C_1$ contains an even number of negative edges then
		$$
		\nu(w_1) = \nu(u_1),
		$$
		while if $C_1$ contains an odd number of negative edges then
		$$
		\nu(w_1) = (w - z)\nu_1(f) - \nu(u_1).
		$$
		The analogous statement is true for $C_2$: if $C_2$ contains an even number of negative edges then
		$$
		\nu(w_2) = \nu(u_2) = (w - z)\nu_1(f) - \nu(u_1),
		$$
		and if $C_2$ contains an odd number of negative edges then
		$$
		\nu(w_2) = (w - z)\nu_1(f) - \nu(u_2) = \nu(u_1).
		$$
		The cycle $C$ contains an even number of negative edges by Corollary \ref{cor:even_cut}, and $e$ and $e'$ are both negative edges, so the numbers of negative edges in $C_1$ and $C_2$ are either both even or both odd. In either case, the four equations above show that
		$$
		\nu(w_2) = (w - z)\nu_1(f) - \nu(w_1).
		$$
		If $f''$ is any face of $G$ bordering $e'$ then $\nu(f'') = \nu_1(f)$ by definition, and consequently property (ii) is satisfied for $w_1$ and $w_2$. Since $e' \in B$ was arbitrary, (ii) holds for $\nu$.
		
		It remains to check (iii). As in the base case this can be done by checking that the $e$-coordinate of $A\nu$ is $0$ for each edge $e \in E$, and as in the previous step this is clear for any edge $e \in E_1 \cup E_2$. If $e \in B$, then the preceding discussion and the definition of $\nu$ show that
		$$
		\big(A\nu\big)(e) = \nu(u_1) + \big((w - z)\nu_1(f) - \nu(u_1)\big) + z\nu_1(f) - w\nu_1(f) = 0,
		$$
		so (iii) holds for $\nu$.
	\end{proof}
	
	We now prove the main theorem of the section.
	
	\begin{thm}
		\label{thm:root_indep}
		Let $G = (V,E)$ be the Tait digraph of an alternating link diagram $D$. Then the polynomial $P_G$ of Definition \ref{def:poly} does not depend on the choice of adjacent root/basepoint pair.
	\end{thm}
	
	\begin{proof}
		Fix an edge $e \in E$, and let $r, r' \in V$ be its vertices. Let $q$ be a basepoint in a face bordered by $e$, and let $q'$ be a basepoint in the face on $e$'s other side. It suffices to show that
		\begin{equation}
			\label{eq:root_shift}
			P_{(G, r, q)} = P_{(G, r', q)}
		\end{equation}
		and 
		\begin{equation}
			\label{eq:face_shift}
			P_{(G, r, q)} = P_{(G, r, q')};
		\end{equation}
		indeed, since $G$ is connected, any adjacent root/basepoint pair can be changed into any other pair through a sequence of such moves. In fact it suffices to show (\ref{eq:root_shift}); equation (\ref{eq:face_shift}) can be obtained from (\ref{eq:root_shift}) by considering the mirror $m(D)$ of $D$ and applying Proposition \ref{prop:companions}.
		
		Let $F$ be the faces of $G$. Conflating notation, we let $q$ denote the face of $G$ containing the basepoint $q$. Let $A$ be an unreduced Kastelyn matrix for $G$, and let
		$$
		\nu \in \big(\Z[x^{-1}, y^{-1}, z, w]\big)^{V \cup F}
		$$
		be a vector satisfying properties (i)--(iii) of Lemma \ref{lem:key} with the coordinates $\nu(r) = 1$ and $\nu(f) = 0$ for (i). By property (ii) we also have
		$$
		\nu(r') = \pm 1l
		$$
		Specifically, $\nu(r') = 1$ if $e$ is a positive edge and
		$$
		\nu(r') = (w - z)\nu(f) - \nu(r) = -1
		$$ 
		if $e$ is a negative edge.
		
		Let
		$$
		\xi_1, \xi_2 \in \big(\Z[x^{-1}, y^{-1}, z, w]\big)^{V \cup F}
		$$
		be two vectors defined as follows: let $\xi_1$ be the vector with $\xi_1(q) = 1$ and $\xi_1(p) = 0$ for all $p \in V \cup F$ with $p \neq q$. Let $\xi_2$ be the vector with $\xi_2(r) = 1$,
		$$
		\xi_2(r') = -\nu(r') = \pm 1,
		$$
		and $\xi_2(p) = 0$ for all $p \in V \cup F$ with $p \notin \{r, r'\}$. Then by construction
		\begin{equation}
			\label{eq:zero}
			\xi_1^\top \cdot \nu = \xi_2^\top \cdot \nu = {\bf 0},
		\end{equation}
		where $\cdot$ is the dot product defined using $V \cup F$ as an orthonormal basis. Finally, let $\bar{A}$ be the matrix $A$ with $\xi_2^\top$ and $\xi_1^\top$ appended to the bottom as two additional rows. Then $\bar{A}$ is a square matrix, and since $\nu$ satisfies (\ref{eq:zero}) and $A\nu = {\bf 0}$ we have
		$$
		\bar{A}\nu = {\bf 0}.
		$$
		It follows that $\det(\bar{A}) = 0$, and a direct calculation using $\xi_1$ and $\xi_2$ shows that
		$$
		0 = \det(\bar{A}) = \det(A^{red}_{(r, q)}) \pm \det(A^{red}_{(r', q)}).
		$$
		From Proposition \ref{prop:p_det} we conclude that $P_{(G, r, q)} = \pm P_{(G, r', q)}$, and in fact the two polynomials are equal since both have non-negative coefficients.
	\end{proof}
	
	\begin{rmk}
		Our proof of Theorem \ref{thm:root_indep} is inspired by Milnor's and Turaev's conception of the Alexander polynomial as the torsion of a chain complex \cite{mil62, tur01}. The theorem itself can also be thought of as a variation and quasi-generalization of the fact that the number of spanning arborescences of a rooted Eulerian digraph does not depend on the choice of root \cite{eb51, tusm41}---this fact is discussed in relation to the Alexander polynomial in \cite{must03, hmv25}. The author hopes this context will be helpful to readers who feel the proof of Theorem \ref{thm:root_indep} is somewhat opaque, since the author feels this way as well. It would be nice to have a topological proof.
	\end{rmk}
	
	\begin{rmk}
		Reading Lemma \ref{lem:cut_cyc} and the proof of Lemma \ref{lem:key}, it is tempting to try to generalize the polynomial $P_G$ in the following way: let $G$ be any plane digraph whose edges can be written as a disjoint union of directed cycles and directed cuts. Then we can define $P_G$ as in Definition \ref{def:poly} by saying an edge is ``positive'' if it is contained in a directed cycle, and ``negative'' if it lies in a directed cut. This extension may well satisfy a root/basepoint independence property, but our proof of Theorem \ref{thm:root_indep} does not work for this class of graphs. The issue is Kauffman's sign function is no longer defined: our proof of Lemma \ref{lem:key} relies on Lemma \ref{lem:a_signs}, which does not seem to generalize beyond Tait digraphs.
	\end{rmk}
	
	\section{Diagram Independence}
	\label{sec:link_invar}
	
	\subsection{Extending $P$ to non-Tait digraphs}

	In this section we complete the proof of Theorem \ref{thm:invariant} by showing $P$ does not depend on the choice of alternating diagram of a link. To do this, it is necessary to extend $P$ to plane digraphs which may not be the Tait digraph of any alternating link diagram.
	
	\begin{defn}
		\label{def:non_tait}
		Let $G = (V,E)$ be a plane digraph with distinguished root $r \in V$ and basepoint $q \in \R^2 \setminus G$, which need not be adjacent to $r$. Suppose $G$ is equipped with a partition of its edges into classes $E_+$ and $E_-$. Then we define a polynomial $P_{(G,r,q)} \in \Z[x^{-1},y^{-1},z,w]$ as in Definition \ref{def:poly}, by
		$$
			P_{(G,r,q)}(x, y, z, w) = \sum_{T \in \mathcal{T}} x^{-\iota_+(T)} y^{-\overline{\iota}_+(T)} z^{\varepsilon_-(T)} w^{\overline{\varepsilon}_-(T)}.
		$$
		The terms $\iota_+(T)$, $\overline{\iota}_+(T)$, $\varepsilon_-(T)$ and $\overline{\varepsilon}_-(T)$ are as in Definition \ref{def:plus_minus_acts}, where edges in $E_+$ and $E_-$ are thought of as positive and negative edges respectively. We call the partition $E = E_+ \sqcup E_-$ a {\em sign partition} of $E$.
	\end{defn}
	
	Since $P_{(G,r,q)}$ depends on the choice of $r$ and $q$ for general plane digraphs, it is important to include them in our notation here. Additionally, if $G = (E,V)$ is a graph equipped with a sign partition, and $G' = (E', V')$ is a {\em minor} of $G$---that is, a graph constructed by deleting some edges of $G$ and contracting others---then $E'$ is naturally a subset of $E$. In this case, we give $E'$ the sign partition $E'_\pm = E' \cap E_\pm$.

	\subsection{Diagram independence}
	
	A crossing $c$ in a link diagram $D$ is {\em nugatory} if there exists a simple closed curve $\gamma \subset \R^2$ which intersects $D$ only at $c$. Such crossings can be removed by rotating one component of $D \setminus \gamma$ 180 degrees out of the plane, untwisting the crossing, and a diagram is {\em reduced} if it does not contain any nugatory crossings.
	
	Menasco and Thistlethwaite proved any two reduced alternating diagrams of a link $K$ are related by a sequence of {\em flypes} \cite{meth91}. Thus, to show $P$ does not depend on which alternating diagram we choose, we'll show it is it is unchanged by
	\begin{itemize}
		\item Adding or removing nugatory crossings, and 
		\item Flyping.
	\end{itemize} 
	A flype is defined by first decomposing an alternating diagram into three pieces: two four-ended tangles $T_1$ and $T_2$, and a single crossing $c$ which joins $T_1$ at two of its ends and $T_2$ at the other two. We can then ``flype'' the diagram by rotating $T_1$ 180 degrees out of the plane and moving $c$ to the other side of $T_1$, so that the isotopy class of the link is unchanged. This operation is depicted in Figure \ref{fig:flype}.
	
	\begin{figure}
		\centering
		
		\begin{tikzpicture}
			
			\draw[ultra thick, red!30] (-4.4,0.65) to[out=135,in=-90] (-4.75,1) to[out=90,in=90] (-1.25,1) -- (-1.25,-1) to[out=-90,in=-90] (-4.75,-1) -- (-4.75,0) to[out=90,in=-90] (-3.75,1) to[out=90,in=90] (-2.25,1) -- (-2.25,-1) to[out=-90,in=-90] (-3.75,-1) -- (-3.75,0) to[out=90,in=-45] (-4.1,0.35); 
			
			\draw[thick, black,fill=yellow!20] (-5,0) rectangle (-3.5,-1) node[pos=.5] {$T_1$};
			\draw[thick, black,fill=gray!30] (-1,1) rectangle (-2.5,-1) node[pos=.5] {$T_2$};
			\node at (-4.75, 0.5) {$c$};
			
			\draw[ultra thick, red!30] (1.6,-0.35) to[out=135,in=-90] (1.25,0) -- (1.25,1) to[out=90,in=90] (4.75,1) -- (4.75,-1) to[out=-90,in=-90] (1.25,-1) to[out=90,in=-90] (2.25,0) -- (2.25,1) to[out=90,in=90] (3.75,1) -- (3.75,-1) to[out=-90,in=-90] (2.25,-1) to[out=90,in=-45] (1.9,-0.65); 
			
			\draw[thick, black,fill=yellow!20] (1,1) rectangle (2.5,0) node[pos=.5] {\reflectbox{$T_1$}};
			\draw[thick, black,fill=gray!30] (5,1) rectangle (3.5,-1) node[pos=.5] {$T_2$};
			\node at (1.25, -0.5) {$c$};
			
			\draw[->, thick] (-0.5,0) -- (0.5,0);
			
		\end{tikzpicture}
		
		\caption{A flype}
		\label{fig:flype}
	\end{figure}
	
	As a first step toward our goal, we examine how flipping over a link diagram affects its Tait digraph.
	
	\begin{lemma}
		\label{lem:diagram_flip}
		Let $D$ be an alternating diagram of a link $K$, with Tait digraph $G$. Let $\rho(D)$ be the diagram resulting from rotating $K$ 180 degrees around an axis $\ell$ lying in the projection plane, and let $\rho(G)$ be the Tait digraph of $\rho(D)$. Then $\rho(G)$ is equal to the graph given by reflecting $G$ across $\ell$ and reversing the directions of all negative edges. Furthermore, the sign of an edge of $G$ matches the sign of its reflected image in $\rho(G)$.
	\end{lemma}
	
	\begin{proof}
		The diagram $\rho(D)$ can be obtained by reflecting $D$ across $\ell$ and then taking its mirror. Since the operations of reflecting and mirroring both change the signs of crossings, any crossing of $D$ has the same sign as its reflected-and-mirrored image in $\rho(D)$. Similarly, reflecting the type I checkerboard shading of $D$ gives the type I checkerboard shading of $\rho(D)$. It follows that $\rho(G)$ coincides with the reflection of $G$ across $\ell$ as an {\em undirected} graph, and that the reflected image of a positive (resp.~negative) edge of $G$ is a positive (resp.~negative) edge of $\rho(G)$.
		
		To finish the proof, as in Proposition \ref{prop:companions}, let $e$ be an edge of $G$ through a crossing $c$ of $D$, let $e'$ be the oriented reflection of $e$ across $\ell$, and let $c'$ be the reflected-and-mirrored image of $c$ in $\rho(D)$. If $c$ is a positive crossing then its undercrossing strand points toward the head of $e$, so the overcrossing strand of $c'$ points toward the head of $e'$. Thus, in this case, the direction $e'$ inherits as the reflection of $e$ agrees with its orientation as an edge of $\rho(G)$. On the other hand, if $c$ is a negative crossing then its undercrossing strand points toward the tail of $e$. Thus the overcrossing strand of $c'$ points toward the tail of $e'$, so the direction of $e'$ disagrees with its direction in $\rho(G)$ in this case.
	\end{proof}
	
	Although we've stated Lemma \ref{lem:diagram_flip} for entire link diagrams, an analogous result is true if we consider flipping {\em part} of a link diagram as in the case of a flype. Lemma \ref{lem:diagram_flip} motivates the following definition.
	
	\begin{defn}
		\label{def:rho_g}
		Let $G = (V,E)$ be any plane digraph with sign partition $E = E_+ \sqcup E_-$ as in Definition \ref{def:non_tait}, and let $\ell \subset \R^2$ be a line. We define $\rho(G)$ to be the plane digraph given by reflecting $G$ across $\ell$, with sign partition inherited from $G$, and then reversing the directions of all negative edges.
	\end{defn}
	
	For any signed plane digraph, the operation $\rho$ preserves the polynomial $P$.
	
	\begin{lemma}
		\label{lem:graph_flip}
		Let $G = (V,E)$ be a plane digraph with sign partition $E = E_+ \sqcup E_-$, and let $\rho(G)$ be as in Definition \ref{def:rho_g} for some line $\ell$. Fix a root $r \in V$ and basepoint $q \in \R^2 \setminus G$, and let $\rho(r)$ and $\rho(q)$ be the respective reflections of $r$ and $q$ across $\ell$. Then
		$$
		P_{(G,r,q)} = P_{(\rho(G), \rho(r), \rho(q))}.
		$$
	\end{lemma}
	
	\begin{proof}
		Fix a spanning tree $T$ of $G$, and let $T'$ be its image under the operation $\rho$. Additionally, fix an edge $e \in E$ and let $e'$ be the corresponding edge of $\rho(G)$. Then it suffices to show that $e$'s contribution to the weight of $T$ in $P_{(G,r,q)}$ matches the contribution of $e'$ to the weight of $T'$ in $P_{(\rho(G),\rho(r),\rho(q))}$.
		
		First we consider the possibility that $e \in E_+ \cap T$, so that $e' \in E(\rho(G))_+ \cap T'$. Then the direction of $e' \in \rho(G')$ agrees with the direction of $e$-reflected-across-$\ell$, and it's clear that $e$ points away from $r$ in $T$ if and only if $e'$ points away from $\rho(r)$ in $T'$. (Indeed, this data does not depend on the embedding $G \hookrightarrow \R^2$.) Thus $e$ is internally semi-active with respect to $T$ if and only if $e'$ is internally semi-active with respect to $T'$.
		
		Next suppose $e \in E_- \setminus T$. Let $C$ be the fundamental cycle of $e$ with respect to $T$, oriented so that it runs counter-clockwise around $q$ in $\R^2 \cup \infty$; then $e$ is externally semi-active with respect to $T$ if and only if its direction agrees with that of $C$. Let $C'$ be the reflection of $C$ across $\ell$---since reflecting reverses orientation, $C'$ runs clockwise around $\rho(r)$ and $e'$ is externally semi-active with respect to $T'$ if and only if its direction disagrees with the orientation of $C'$. By the definition of $\rho(G)$ the direction of $e'$ also disagrees with the direction of $e$-reflected-across-$\ell$, and these two reversals cancel out: we conclude that $e$ is externally semi-active with respect to $T$ if and only if $e'$ is externally semi-active with respect to $T'$.
	\end{proof}
	
	Combining Lemmas \ref{lem:diagram_flip} and \ref{lem:graph_flip} with Theorem \ref{thm:root_indep}, we obtain:
	
	\begin{cor}
		\label{cor:diag_flip}
		Let $D$ be an alternating link diagram, and $\rho(D)$ the result of rotating $D$ 180 degrees around a line in the projection plane. Then $P_D = P_{\rho(D)}$.
	\end{cor}
	
	To handle flypes, we will decompose Tait digraphs into ``flipped'' and ``non-flipped'' pieces. To this end, we call two vertices $\{v_1, v_2\}$ of a connected graph $G$ a {\em separating pair} if $G \setminus \{v_1, v_2\}$ is disconnected. Our next lemma gives an identity for the polynomial $P_{(G,r,q)}$ whenever $G$ admits such a pair.
	
	\begin{lemma}
		\label{lem:two_cut}
		Let $G = (V,E)$ be any connected plane digraph with sign partition $E = E_+ \sqcup E_-$, such that $G$ contains a separating pair of vertices $r, v \in V$. Let $G_1 = (V_1, E_1)$ and $G_2 = (V_2, E_2)$ be the subgraphs of $G$ which are the closures of the two components of $G \setminus \{r, v\}$, and let $q$ be a basepoint in a face adjacent to $r$. Then
		\begin{equation}
			\label{eq:pid}
			P_{(G, r, q)} = P_{(G \setminus E_2, r, q)} P_{(G / E_1, r, q)} + P_{(G \setminus E_1, r, q)} P_{(G / E_2, r, q)}.
		\end{equation}
	\end{lemma}
	
	As usual, the graphs on the right side of (\ref{eq:pid}) inherit their planar embeddings and sign partitions from $G$.
	
	\begin{proof}
		Let $\mathcal{T}$ be the set of spanning trees of $G$, and let $\mathcal{T}_{\setminus E_2}$, $\mathcal{T}_{/ E_1}$, $\mathcal{T}_{\setminus E_1}$, and $\mathcal{T}_{/ E_2}$ denote the respective spanning tree sets of $G \setminus E_2$, $G / E_1$, $G \setminus E_1$, and $G / E_2$. We first observe that there is a bijection
		\begin{equation}
			\label{eq:tree_bij}
			\mathcal{T} \leftrightarrow (\mathcal{T}_{\setminus E_2} \times \mathcal{T}_{/ E_1}) \sqcup (\mathcal{T}_{\setminus E_1} \times \mathcal{T}_{/ E_2}).
		\end{equation}
		Indeed, any spanning tree $T \in \mathcal{T}$ contains a unique path connecting $r$ and $v$, which lies in either $E_1$ or $E_2$ since $r$ and $v$ separate $G$. In the first case $T \cap E_1$ gives a spanning tree of $G \setminus E_2$, while $T \cap E_2$ gives a spanning tree of $G / E_1$. The second case is the same with the indices switched, which completes one direction of the bijection. For the other direction, suppose $T_1 \in \mathcal{T}_{\setminus E_2}$ and $T_2 \in \mathcal{T}_{/ E_1}$. Then it is easy to check that $T_1 \cup T_2$ is a spanning tree of $\mathcal{T}$, with
		$$
		T_1 \cup T_2 \to (T_1, T_2)
		$$
		under the map just described. The same is true if $T_1 \in \mathcal{T}_{\setminus E_1}$ and $T_2 \in \mathcal{T}_{/ E_2}$.
		
		Now fix a tree $T \in \mathcal{T}$, and without loss of generality let $(T_1, T_2) \in \mathcal{T}_{\setminus E_2} \times \mathcal{T}_{/ E_1}$ be its image under (\ref{eq:tree_bij}). To prove the lemma we will show that
		\begin{align*}
			\iota_+(T) &= \iota_+(T_1) + \iota_+(T_2)  \\
			\overline{\iota}_+(T) &= \overline{\iota}_+(T_1) + \overline{\iota}_+(T_2)\\
			\varepsilon_-(T) &= \varepsilon_-(T_1)  + \varepsilon_-(T_2) \\
			\overline{\varepsilon}_-(T) &= \overline{\varepsilon}_-(T_1) + \overline{\varepsilon}_-(T_2).
		\end{align*}
		Then the weight of $T$ in $P_{(G,r,q)}$ is equal to the product of the weights of $T_1$ in $P_{(G \setminus E_2, r, q)}$ and $T_2$ in $P_{(G / E_1, r, q)}$, which implies the result.
		
		Fix an edge $e \in E$---then it suffices to prove that $e$'s activity with respect to $T$ is the same as its activity with respect to $T_1$ if $e \in E_1$, or with respect to $T_2$ if $e \in E_2$. This amounts to some case checking. First suppose $e \in T$, and let $\gamma$ be the unique path in $T$ from the root $r$ through $e$. Then $e$ is internally semi-active in $T$ if and only if $e$ points away from $r$ in $\gamma$. If $e \in E_1$ then $\gamma \subset E_1$ as well, since by assumption $T \cap E_2$ is disconnected. In this case $\gamma$ is also the unique path from $r$ through $e$ in $T_1$, and it's clear that $e$ is internally semi-active in $T$ if and only if $e$ is internally semi-active in $T_1$.
		
		If $e \in T \cap E_2$, then the path $\gamma$ may still pass through $E_1$. Let $\gamma_1 = \gamma \cap E_1$; then the unique path from $r$ through $e$ in $T_2$ is given by the contraction
		$$
		\gamma' = \gamma / \gamma_1.
		$$
		Clearly $e$ points away from $r$ in $\gamma$ if and only if it points away from $r$ in $\gamma'$, so $e$ is internally semi-active in $T$ if and only if $e$ is internally semi-active in $T_2$ as desired. External activity is handled similarly, by considering fundamental cycles instead of paths.
	\end{proof}
	
	\begin{rmk}
		In \cite[Proposition 2.21]{ajk24}, Azarpendar, Juh\'asz and K\'alm\'an  prove an identity similar to (\ref{eq:pid}) for the Alexander polynomial of any link; an alternate proof is given in \cite[Theorem 5.3]{msbv25}. In the case of alternating links, this identity can be obtained from our Lemma \ref{lem:two_cut} by passing to the Alexander polynomial and considering the special case where the separating pair of vertices is part of a diagrammatic Murasugi sum.
	\end{rmk}
	
	To put the preceding lemmas in context, let $D$ be an alternating link diagram with sub-tangles $T_1$ and $T_2$, and crossing $c$ as in Figure \ref{fig:flype}. Let $G$ be the Tait digraph of $D$, and let $G_i$ be the subgraph of $G$ corresponding to $T_i$ for $i \in \{1,2\}$. Let $e$ be the edge passing through $c$; then the effect of flyping the alternating link diagram of Figure \ref{fig:flype} on its Tait digraph is shown in Figure \ref{fig:flype_tait}. We observe that the vertices $r = G_1 \cap \text{cl}(e)$ and $v = G_1 \cap G_2$ form a separating pair.
	
	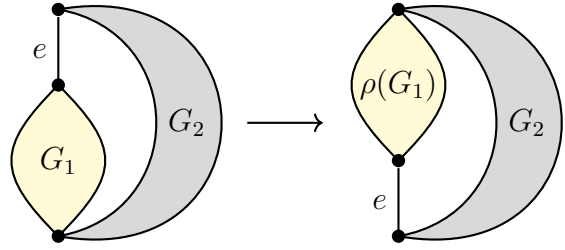
\begin{figure}
		\centering
		
		\begin{tikzpicture}
			
			\coordinate (tl) at (-3,1.5);
			\coordinate (ml) at (-3,0.5);
			\coordinate (bl) at (-3,-1.5);
			
			\fill[yellow!20] (ml) to[out=-45,in=45, looseness=1.5] (bl) to[out=135,in=-135, looseness=1.5] (ml);
			\draw[thick,black] (ml) to[out=-45,in=45, looseness=1.5] (bl);
			\draw[thick,black] (bl) to[out=135,in=-135, looseness=1.5] (ml);
			\node at (-3,-0.5) {$G_1$};
			
			\fill[gray!30] (tl) to[out=-10,in=10,looseness=1.5] (bl) to[out=-10,in=10,looseness=2.5] (tl);
			\draw[thick,black] (tl) to[out=-10,in=10,looseness=1.5] (bl);
			\draw[thick,black] (bl) to[out=-10,in=10,looseness=2.5] (tl);
			\node at (-1.3,0) {$G_2$};
			
			\fill[black] (tl) circle (2.5pt);
			\fill[black] (ml) circle (2.5pt);
			\fill[black] (bl) circle (2.5pt);
			
			\draw[thick, black] (ml) -- node[left] {$e$} (-3,1.4);
			
			\coordinate (tr) at (1.5,1.5);
			\coordinate (mr) at (1.5,-0.5);
			\coordinate (br) at (1.5,-1.5);
			
			\fill[yellow!20] (tr) to[out=-45,in=45,looseness=1.5] (mr) to[out=135,in=-135,looseness=1.5] (tr);
			\draw[thick,black] (tr) to[out=-45,in=45, looseness=1.5] (mr);
			\draw[thick,black] (mr) to[out=135,in=-135,looseness=1.5] (tr);
			\node at (1.5,0.5) {$\rho(G_1)$};
			
			\fill[gray!30] (tr) to[out=-10,in=10,looseness=1.5] (br) to[out=-10,in=10,looseness=2.5] (tr);
			\draw[thick,black] (tr) to[out=-10,in=10,looseness=1.5] (br);
			\draw[thick,black] (br) to[out=-10,in=10,looseness=2.5] (tr);
			\node at (3.2,0) {$G_2$};
			
			\fill[black] (tr) circle (2.5pt);
			\fill[black] (mr) circle (2.5pt);
			\fill[black] (br) circle (2.5pt);
			
			\draw[thick, black] (br) -- node[left] {$e$} (1.5,-0.6);
			
			\draw[->, thick] (-0.5,0) -- (0.5,0);
			
		\end{tikzpicture}
		
		\caption{Flyping a Tait digraph}
		\label{fig:flype_tait}
	\end{figure}
	
	Before proving that $P$ is unchanged by flypes, we need one more lemma allowing us to add and remove bridges. Recall that a {\em bridge} of a connected graph is an edge whose removal disconnects it.
	
	\begin{lemma}
		\label{lem:bridge}
		Let $G = (V,E)$ be a plane digraph with root $r \in V$, basepoint $q \in \R^2 \setminus G$, and sign partition $E = E_+ \sqcup E_-$. Suppose $G$ contains a bridge $e$.
		\begin{enumerate}[label=(\roman*)]
			\item If $e \in E_-$, then $P_{(G,r,q)} = P_{(G/e, r, q)}$.
			\item If $e \in E_+$ and $e$ points away from the component of $G \setminus e$ containing $r$, then $P_{(G,r,q)} = x^{-1} P_{(G / e, r, q)}$.
		\end{enumerate}
	\end{lemma}
	
	\begin{proof}
		Let $\mathcal{T}$ and $\mathcal{T}_{/e}$ be the respective spanning tree sets of $G$ and $G/e$. As a bridge $e$ is contained in every spanning tree of $G$, and we have a bijection
		$$
			\mathcal{T} \leftrightarrow \mathcal{T}_{/e}
		$$
		given by $T \mapsto T \setminus e$ for any $T \in \mathcal{T}$. For (i), if $e \in E_-$ then $e$ does not count toward $\iota_+(T)$, $\overline{\iota}_+(T)$, $\varepsilon_-(T)$ or $\overline{\varepsilon}_-(T)$ for any tree $T$---the first two because $e \in E_-$, and the last two because $e \in T$. It is then easy to check that the above bijection preserves weights, so the two polynomials are equal. 
		
		The proof of (ii) is similar. We observe that $e$ counts toward $\iota_+(T)$ for every spanning tree, so that $P_{(G,r,q)}$ contains a factor of $x^{-1}$.
	\end{proof}
	
	We can now show $P$ does not detect flypes.
	
	\begin{lemma}
		\label{lem:flype_invar}
		Let $K$ be an alternating link, and let $D$ and $D'$ be two alternating diagrams of $K$ related by a flype. Then $P_D = P_{D'}$.
	\end{lemma}
	
	\begin{proof}
		Let $T_1$ and $T_2$ be the sub-tangles of $D$ relevant to the flype, and $c$ the additional crossing. Let $G$ be the Tait digraph of $D$, $G_i$ the subgraph corresponding to $T_i$ for $i \in \{1,2\}$, and $e$ the edge of $G$ passing through $c$. By mirroring $D$ if necessary and applying Proposition \ref{prop:companions} (ii), we assume that $c$ is arranged relative to $T_1$ as on the left side of Figure \ref{fig:flype}. Then the effect of flyping $D$ on the Tait digraph is shown in Figure \ref{fig:flype_tait}.
		
		\begin{figure}
			\centering
			
			\begin{subfigure}[t]{0.22\textwidth}
				\centering
				\begin{tikzpicture}
					\draw[-{Stealth}, ultra thick, red] (-0.5,0) -- (0.5,1);
					\draw[ultra thick, red] (0.5,0) -- (0.15,0.35);
					\draw[{Stealth}-, ultra thick, red] (-0.5,1) -- (-0.15, 0.65);
					
					\draw[thick, black,fill=yellow!20] (-0.75,0) rectangle (0.75,-1) node[pos=.5] {$T_1$};
					
					\draw[-{Stealth}, ultra thick, red] (-0.5,-1.5) -- (-0.5,-1);
					\draw[-{Stealth}, ultra thick, red] (0.5,-1.5) -- (0.5,-1);
				\end{tikzpicture}
				\caption{}
			\end{subfigure}
			\begin{subfigure}[t]{0.22\textwidth}
				\centering
				\begin{tikzpicture}
					\draw[-{Stealth}, ultra thick, red] (-0.5,0) -- (0.5,1);
					\draw[{Stealth}-, ultra thick, red] (0.5,0) -- (0.15,0.35);
					\draw[ultra thick, red] (-0.5,1) -- (-0.15, 0.65);
					
					\draw[thick, black,fill=yellow!20] (-0.75,0) rectangle (0.75,-1) node[pos=.5] {$T_1$};
					
					\draw[-{Stealth}, ultra thick, red] (-0.5,-1.5) -- (-0.5,-1);
					\draw[{Stealth}-, ultra thick, red] (0.5,-1.5) -- (0.5,-1);
				\end{tikzpicture}
				\caption{}
			\end{subfigure}
			\begin{subfigure}[t]{0.22\textwidth}
				\centering
				\begin{tikzpicture}
					\draw[-{Stealth}, ultra thick, red] (-0.5,0) -- (0.5,1);
					\draw[{Stealth}-, ultra thick, red] (0.5,0) -- (0.15,0.35);
					\draw[ultra thick, red] (-0.5,1) -- (-0.15, 0.65);
					
					\draw[thick, black,fill=yellow!20] (-0.75,0) rectangle (0.75,-1) node[pos=.5] {$T_1$};
					
					\draw[{Stealth}-, ultra thick, red] (-0.5,-1.5) -- (-0.5,-1);
					\draw[-{Stealth}, ultra thick, red] (0.5,-1.5) -- (0.5,-1);
				\end{tikzpicture}
				\caption{}
			\end{subfigure}
			
			\caption{Possible flype orientations}
			\label{fig:flype_ors}
		\end{figure}
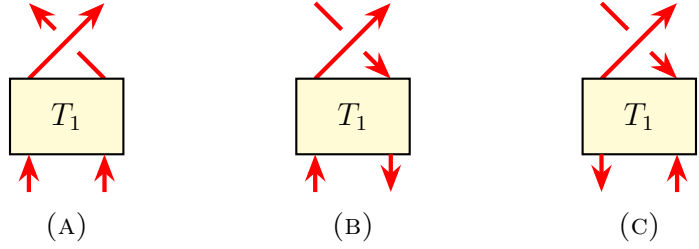
		
		Up to orientation reversal (which may be handled with Proposition \ref{prop:companions} (i)), there are three possibilities for the orientation of $K$ between the tangles $T_1$ and $T_2$. These are shown in Figure \ref{fig:flype_ors} in a neighborhood of $T_1$. Cases (a) and (b) can be handled by the same argument, which is most easily expressed diagrammatically. To simplify notation we write $\langle G \rangle$ to mean $P_G$ for a rooted, basepointed plane digraph $G$. Then for orientations (a) and (b), we calculate:
		\begin{align*}
		\Big\langle \tikzone \Big\rangle &= \Big\langle \tikztwo \Big\rangle \Big\langle \tikzthree \Big\rangle + \Big\langle \tikzfour \Big\rangle \Big\langle \tikzfive \Big\rangle \\
		\vspace{-5cm}
		&= \Big\langle \tikztwof \Big\rangle \Big\langle \tikzthree \Big\rangle + \Big\langle \tikzfourf \Big\rangle \Big\langle \tikzfive \Big\rangle \\
		&= \Big\langle \tikztwof \Big\rangle \Big\langle \tikzthreef \Big\rangle + \Big\langle \tikzfourff \Big\rangle \Big\langle \tikzfivef \Big\rangle \\
		&= \Big\langle \tikzonef \Big\rangle
		\end{align*}
		
		The initial diagram shows the Tait digraph $G$---we've circled the vertex $r = G_1 \cap \text{cl}(e)$ to distinguish it as the root of $G$, and the red asterisk marks the basepoint $q$. We then obtain the first equation by applying Lemma \ref{lem:two_cut} to the separating pair of vertices $r$ and $v = G_1 \cap G_2$. In the second equation we use Lemma \ref{lem:graph_flip} to perform the operation $\rho$ the first and third graphs, $G \setminus (G_2 \cup e)$ and $G / (G_2 \cup e)$.
		
		Multiple moves occur between the second and third lines. First we use Theorem \ref{thm:root_indep} to change the root and basepoint of the second graph, $G / G_1$---this is possible because $G / G_1$ is the Tait digraph of the alternating link diagram in Figure \ref{fig:int_diag}.\footnote{This is where the calculation fails for  Figure \ref{fig:flype_ors}c: in that case the strands above and below the tangle $T_1$ cannot be glued together in the desired way without changing their orientations.} (This diagram is alternating by Lemma \ref{lem:alt_crossing_signs}, since all its crossings have the same type.) Next, we change the third graph by isotoping it through the point at infinity, which preserves the polynomial $P$. Finally, in the fourth graph, we've applied Lemma \ref{lem:bridge} twice to move the bridge edge $e$ from one side of $G \setminus G_1$ to the other.
		
		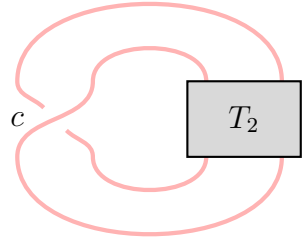
\begin{figure}
			\centering
			
			\begin{tikzpicture}
				
				\draw[ultra thick, red!30] (-4.4,0.65) to[out=135,in=-90] (-4.75,1) to[out=90,in=90] (-1.25,1) -- (-1.25,0) to[out=-90,in=-90] (-4.75,0) to[out=90,in=-90] (-3.75,1) to[out=90,in=90] (-2.25,1) -- (-2.25,0) to[out=-90,in=-90] (-3.75,0) to[out=90,in=-45] (-4.1, 0.35); 
				
				\draw[thick, black,fill=gray!30] (-1,1) rectangle (-2.5,0) node[pos=.5] {$T_2$};
				\node at (-4.75, 0.5) {$c$};
				
			\end{tikzpicture}
			
			\caption{The link diagram for $G / G_1$}
			\label{fig:int_diag}
		\end{figure}
		
		The last equation follows by a second application of Lemma \ref{lem:two_cut}, and we conclude that $P_D = P_{D'}$ whenever the diagram $D$ is oriented as in Figure \ref{fig:flype_ors} (a) or (b). We still need to show invariance under flypes with orientation (c)---this can be done using a similar chain of equivalences, or by combining case (b) with Corollary \ref{cor:diag_flip}. We do the latter:
		
		\begin{align*}
		\Big\langle \tikzdone \Big\rangle &= \Big\langle \tikzdtwo \Big\rangle \\
		&= \Big\langle \tikzdthree \Big\rangle \\
		&= \Big\langle \tikzdfour \Big\rangle
		\end{align*}
		
		In the first equation above, we've used invariance under orientation (b) to flype the diagram so that $T_2$ flips and $T_1$ does not. In the second equality, we apply Corollary \ref{cor:diag_flip} to rotate the enter diagram around a vertical line in the plane. Finally, we isotope the tangle $T_2$ through the point at infinity to obtain the desired identity.
	\end{proof}
	
	\begin{thm}
		\label{thm:diag_indep}
		Let $K$ be an alternating link, and let $D$ and $D'$ be any two alternating diagrams for $K$. Then $P_D = P_{D'}$.
	\end{thm}
	
	\begin{proof}
		Let $G$ be the Tait digraph of the diagram $D$. If $D$ contains a nugatory crossing $c$, then the corresponding edge $e$ is either a bridge or a loop; by mirroring $D$ if necessary, we assume the former. Then we may write $G = G_1 \cup \{e\} \cup G_2$, where $G_1$ and $G_2$ are the components of $G \setminus e$, and let $v_i$ be the vertex $G_i \cap e$ for $i \in \{1,2\}$.
		
		Let $D''$ be the diagram obtained from $D$ by removing $c$, as described at the beginning of the section. Let $G''$ be the Tait digraph of $D''$---then without loss of generality, $G''$ is the result of gluing $\rho(G_1)$ and $G_2$ together by identifying $v_1$ and $v_2$:
		$$
		G'' = \rho(G_1) \cup_{v_1 \sim v_2} G_2. 
		$$
		
		By Lemma \ref{lem:cut_cyc} (or by a quick inspection), $e \in E_-$. Thus, applying Lemma \ref{lem:bridge}, Proposition \ref{prop:sum} and Lemma \ref{lem:diagram_flip}, we have:
		$$
		P_D = P_G = P_{G_1 \cup_{v_1 \sim v_2} G_2} = P_{G_1}P_{G_2} = P_{\rho(G_1)}P_{G_2} = P_{G''} = P_{D''}.
		$$
		We conclude that removing $c$ does not change the value of $D$, and proceeding inductively we may assume $D$ and $D'$ are both reduced. Then $D$ and $D'$ are related by a sequence of flypes \cite{meth91}, and the result follows from Lemma \ref{lem:flype_invar}.
	\end{proof}

	\begin{rmk}
		A {\em mutation} is an operation similar to a flype, where any four-ended tangle in a diagram is cut out, rotated 180 degrees along any axis, and then glued back in. Unlike flyping, mutation does not necessarily preserve either the orientation or the isotopy type of a link. An argument much like the proof of Lemma \ref{lem:flype_invar} can be used to show $P$ is invariant under orientation-preserving mutations, but we do not know if $P$ is mutation invariant in general. This is a natural question for further study.
	\end{rmk}
	
	\section{Symmetry and Trapezoidality Results}
	\label{sec:symmetry}
	
	In this section we prove Theorems \ref{thm:sym_one} and \ref{thm:trap}. First, we introduce a homogenization of the {\em Murasugi-Stoimenow polynomial} defined in \cite{must03}.
	
	\begin{defn}
		Let $G$ be a connected digraph with distinguished root vertex $r$, and let $\mathcal{T}$ be its set of spanning trees. Define a polynomial $R_{(G,r)} \in \Z[x,y]$ by
		$$
		R_{(G,r)}(x,y) = \sum_{T \in \mathcal{T}} x^{\iota(T)}y^{\overline{\iota}(T)} = \sum_{T \in \mathcal{T}} x^{\iota(T)}y^{\text{rk}(G) - \iota(T)},
		$$
		where $\iota$ and $\overline{\iota}$ are as in Definition \ref{def:int}.
	\end{defn}
	
	We will also need the following results of Hafner-M\'esz\'aros-Vidinas and Gao-Yuan.
	
	\begin{thm}[{\cite[Theorem 1.3]{hmv25}}]
		\label{thm:hmv}
		If $G$ is Eulerian, then the polynomial $R_{(G,r)}$ is independent of the choice of root $r$, and is fixed by the involution that swaps $x$ and $y$.
	\end{thm}
	
	\begin{thm}[{\cite[Corollary 1.3]{gayu26}}]
		\label{thm:gayu}
		If $G$ is Eulerian with root $r$, and $c_j$ denotes the coefficient of $x^jy^{\text{rk}(G) - j}$ in $R_{(G,r)}$, then the sequence $c_0, c_1, \dots, c_{\text{rk}(G)}$ is trapezoidal.
	\end{thm}
	
	Theorem \ref{thm:hmv} is also implied by the earlier result \cite[Remark 3.5]{kmp25}, but the proof referenced there has not yet been published. Our theorems will follow from these and from the next two lemmas.
	
	For the rest of the section, given a Tait digraph $G = (V, E_- \cup E_+)$, define
	$$
	\mathcal{I}_- = \{I \subset E_- \mid \text{there exists $J \subset E_+$ such that $I \cup J$ is a spanning tree}\}.
	$$
	
	\begin{lemma}
		\label{lem:contracting}
		Let $G = (V, E_+ \sqcup E_-)$ be a Tait digraph with root $r$ and basepoint $q$. Then
		$$
		P_{(G,r,q)}(x,y,1,1) = \sum_{I \in \mathcal{I}_-} R_{\big((G/I)\setminus (E_- \setminus I), \overline{r} \big)}(x^{-1},y^{-1}),
		$$
		where we write $\overline{r}$ to indicate the image of $r$ under the contraction $G \to G/I$.
	\end{lemma}
	
	\begin{proof}
		Let $\mathcal{T}$ be the set of spanning trees of $G$; then by definition
		$$
		P_{(G,r,q)}(x,y,1,1) = \sum_{T \in \mathcal{T}}x^{-\iota_+(T)}y^{-\overline{\iota}_+(T)}.
		$$
		Given a tree $T \in \mathcal{T}$, write $T_\pm = T \cap E_\pm$. As in the proof of Lemma \ref{lem:two_cut}, one may check that $T_+$ is a spanning tree of $(G/T_-)\setminus (E_- \setminus T_-)$, and that $\iota(T_+) = \iota_+(T)$ and $\overline{\iota}(T_+) = \overline{\iota}_+(T)$, since the non-contracted edges of $T$ in $T_+$ are precisely the positive edges, and a non-contracted edge has the same activity as its image.
		
		We can thus write
		$$
		P_{(G,r,q)}(x,y,1,1) = \sum_{T \in \mathcal{T}}x^{-\iota(T_+)}y^{-\overline{\iota}(T_+)},
		$$
		where $\iota(T_+)$ and $\overline{\iota}(T_+)$ are computed in the graph $(G/T_-)\setminus (E_- \setminus T_-)$ for a given $T = T_- \cup T_+$. Letting $\mathcal{T}_{(G/T_-)\setminus (E_- \setminus T_-)}$ be the spanning trees of $(G/T_-)\setminus (E_- \setminus T_-)$ for a given $T_-$, this becomes
		\begin{align*}
		P_{(G,r,q)}(x,y,1,1) &=  \sum_{I \in \mathcal{I}_-} \Big( \sum_{T_+ \subset E_+ \text{ with } I \cup T_+ \in \mathcal{T}} x^{-\iota(T_+)}y^{-\overline{\iota}(T_+)} \Big) \\
		&= \sum_{I \in \mathcal{I}_-} \sum_{T \in \mathcal{T}_{(G/I)\setminus (E_- \setminus I)}}x^{-\iota(T)}y^{-\overline{\iota}(T)} \\
		&= \sum_{I \in \mathcal{I}_-} R_{\big((G/I)\setminus (E_- \setminus I), \overline{r} \big)}(x^{-1},y^{-1})
		\end{align*}
		as desired. In the second equation, we use the fact that $I \cup T_+$ is a spanning tree of $G$ if and only if $T_+$ is a spanning tree of $(G/I) \setminus (E_- \setminus I)$.
	\end{proof}
	
	\begin{lemma}
		\label{lem:eulerian}
		Let $G = (V, E_+ \sqcup E_-)$ be a Tait digraph. Then for any subset $I \in \mathcal{I}_-$, the graph $(G / I) \setminus (E_- \setminus I)$ is Eulerian.
	\end{lemma}
	
	\begin{proof}
		Let $G' = (G / I) \setminus (E_- \setminus I)$ for some $I \in \mathcal{I}_-$; then there is a natural identification of $E(G')$ with $E_+$. By Lemma \ref{lem:cut_cyc} (i), the set $E_+$ can be written as a disjoint union of directed cycles of $G$. Since $G'$ is obtained from $G$ by contracting and deleting edges, each of these cycles is a union of directed cycles of $G'$. Thus $E(G')$ can be written as a disjoint union of directed cycles, and since $I$ extends to a spanning tree by adding positive edges, $G'$ is connected. It follows that $G'$ is Eulerian.
	\end{proof}
	
	We now prove the main theorems, which we restate here.
	
	\begin{named_thm}{\refthm{sym_one}}
		For any alternating link $K$,
			$$
			P_K(x,y,1,1) = P_K(y,x,1,1)
			$$
		and
			$$
			P_K(1,1,z,w) = P_K(1,1,w,z).
			$$
	\end{named_thm}
	
	\begin{proof}
		By mirroring (Proposition \ref{prop:companions} (ii)), it suffices to prove $P_K(x,y,1,1)$ is preserved by the involution which swaps $x$ and $y$. Let $D$ be an alternating diagram of $K$, $G$ its Tait digraph, and fix a root $r$ and adjacent basepoint $q$. Then by Lemma \ref{lem:contracting}
		$$
			P_K(x,y,1,1) = P_{(G,r,q)}(x,y,1,1) = \sum_{I \in \mathcal{I}_-} R_{\big((G/I)\setminus (E_- \setminus I), \overline{r} \big)}(x^{-1},y^{-1}),
		$$
		and each graph $(G/I) \setminus (E_- \setminus I)$ is Eulerian by Lemma \ref{lem:eulerian}. Thus, by Theorem \ref{thm:hmv}, each polynomial $R_{\big((G/I)\setminus (E_- \setminus I), \overline{r} \big)}(x^{-1},y^{-1})$ is fixed by the involution $x \leftrightarrow y$. It follows that $P_K$ is as well.
	\end{proof}
	
	\begin{named_thm}{\refthm{trap}}
		For any alternating link $K$, each of the following sequences is trapezoidal:
		\begin{itemize}
			\item The even-degree coefficients of the polynomial $P_K(t^{-1}, t,1,1)$.
			\item The odd-degree coefficients of the polynomial $P_K(t^{-1}, t,1,1)$.
			\item The even-degree coefficients of the polynomial $P_K(1, 1,t^{-1},t)$.
			\item The odd-degree coefficients of the polynomial $P_K(1, 1,t^{-1},t)$.
		\end{itemize}
	\end{named_thm}
	
	\begin{proof}
		We follow the notation of the previous argument. As in that proof it suffices to check the first two statements, and we use the identity
		$$
		P_K(x,y,1,1) = \sum_{I \in \mathcal{I}_-} R_{\big((G/I)\setminus (E_- \setminus I), \overline{r} \big)}(x^{-1},y^{-1}),
		$$
		where each graph $(G/I) \setminus (E_- \setminus I)$ is Eulerian. For each $I \in \mathcal{I}_-$, let
		$$
		G_I = (G/I)\setminus (E_- \setminus I).
		$$
		Then the sequence of coefficients of $R_{(G_I, \overline{r} )}(x^{-1},y^{-1})$ is trapezoidal by Theorem \ref{thm:gayu}, and therefore so is the sequence of nonzero coefficients of
		$$
		R_{(G_I, \overline{r})}((t^{-1})^{-1},(t)^{-1}) = \sum_{T \in \mathcal{T}_{G_I}} t^{2\iota(T) - \text{rk}(G_I)}.
		$$
		By the BEST Theorem \cite{eb51,tusm41}, since $G_I$ is Eulerian, it admits spanning trees $T$ with $\iota(T) = 0$ and with $\iota(T) = |T| = \text{rk}(G_I)$. Thus, writing $c_j$ for the coefficient of $t^j$, the sequence of nonzero coefficients of $R_{(G_I, \overline{r})}$ is precisely
		$$
		c_{-\text{rk}(G_I)}, c_{-\text{rk}(G_I) + 2}, \dots, c_{\text{rk}(G_I) - 2}, c_{\text{rk}(G_I)},
		$$
		and the index of every coefficient is either odd or even depending on $\text{rk}(G_I)$. Note that
		$$
		\text{rk}(G_I) = \text{rk}(G) - |I|.
		$$
		Furthermore, since $R_{(G_I, \overline{r})}(x,y)$ is fixed by the involution $x \leftrightarrow y$, $c_j = c_{-j}$ for all $j$ and we say the coefficient sequence is {\em symmetric (about zero)}.
		
		Let $m_1 = \min(\{|I| \mid I \in \mathcal{I}_-\})$ and $m_2 = \max(\{|I| \mid I \in \mathcal{I}_i\})$, and for any $k$ with $m_1 \leq k \leq m_2$ let
		$$
		R_k = \sum_{I \in \mathcal{I}_-, \ |I| = k} R_{(G_I, \overline{r})}((t^{-1})^{-1},(t)^{-1}).
		$$
		Each polynomial on the right has the same support, so by \cite[Proposition 2.1]{mur85} the (nonzero) coefficient sequence of $R_k$ is symmetric and trapezoidal for all $k$. Next, write
		$$
		P_K(x,y,1,1) = \sum_{m_1 \leq k \leq m_2, \ k \text{ even}} R_k + \sum_{m_1 \leq k \leq m_2, \ k \text{ odd}} R_k.
		$$
		To prove the theorem, we must show the coefficient sequence of each summation on the right is symmetric and trapezoidal.
		
		We will prove the even case; the odd case is identical. To simplify notation, we also assume $m_1$ and $m_2$ are both even. \cite[Proposition 2.1]{mur85} implies that, if $P$ and $P'$ are two Laurent polynomials whose even-degree coefficient sequences are symmetric and trapezoidal, and if the minimum degree of $P$ is exactly two less than the minimum of degree of $P'$, then the sequence of even degree coefficients of $P + P'$ is also symmetric and trapezoidal. Write
		$$
		\sum_{m_1 \leq k \leq m_2, \ k \text{ even}} R_k = R_{m_1} + (R_{m_1 + 2} + (R_{m_1 + 4} + \cdots + ( R_{m_2 - 2} + R_{m_2}))).
		$$
		By the above fact, if $R_{m_2 - 2}$ is not identically zero, then the sum $R_{m_2 - 2} + R_{m_2}$ in the innermost parantheses is a Laurent polynomial with minimum degree $m_2 - 2 - \text{rk}(G)$, whose even-degree coefficient sequence is symmetric and trapezoidal. Proceeding inductively, we find that the even-degree coefficients of the sum $\sum_{m_1 \leq k \leq m_2, \ k \text{ even}} R_k$ are symmetric and trapezoidal so long as none of the $R_k$ are zero.
		
		Thus, to complete the proof it suffices to show that for all $k$ with $m_1 < k < m_2$, there exists $I \in \mathcal{I}_-$ with $|I| = k$, so that $R_k$ is nonzero. Fix $I_1, I_2 \in \mathcal{I}_-$ with $|I_1| = m_1$ and $|I_2| = m_2$. Then by definition there exist spanning trees $T_1$ and $T_2$ of $G$ such that $T_1 \cap E_- = I_1$ and $T_2 \cap E_- = I_2$. By the basis exchange property for spanning trees, if $T_1 \neq T_2$, then there exist edges $e \in T_1 \setminus T_2$ and $e' \in T_2 \setminus T_1$ such that $(T_1 \setminus \{e\}) \cup \{e'\}$ is a spanning tree of $G$. Therefore there is a sequence of spanning trees of $G$,
		$$
		T_1 = T'_1, T'_2, \dots, T'_r = T_2
		$$
		such that $T'_{j + 1}$ is obtained from $T'_j$ by exchanging a single edge. It follows that
		$$
		||T'_{j + 1} \cap E_-| - |T'_j \cap E_-|| \leq 1
		$$
		for all $j = 1, \dots, r - 1$. Since $|T'_1 \cap E_-| = m_1$ and $|T'_r \cap E_-| = m_2$, there exists $T'_j$ with $|T'_j \cap E_-| = k$ for all $m_1 \leq k \leq m_2$. Setting $I = T'_j \cap E_-$ gives the desired $I \in \mathcal{I}_-$.
	\end{proof}

	\bibliography{main_bib}{}
	\bibliographystyle{amsplain}
	
\end{document}